\documentclass{amsart}
\usepackage{amssymb,amsmath,amsfonts,amsthm,epsfig,amscd,hyperref,enumitem,comment}
\usepackage{comment}
\usepackage[all,cmtip,poly]{xy}
\usepackage{svn, url}
\usepackage{mathrsfs}
\usepackage{xr}
\usepackage{tikz-cd}

\swapnumbers 

\newtheorem{theorem}{Theorem}

\newtheorem{thm}[subsubsection]{Theorem}

\newtheorem{lem}[subsubsection]{Lemma}

\newtheorem{cor}[subsubsection]{Corollary}

\newtheorem{prop}[subsubsection]{Proposition}

\newtheorem{defn}[subsubsection]{Definition}

\theoremstyle{remark}
\newtheorem{remark}[subsubsection]{Remark}
\newtheorem{rem}[subsubsection]{Remark}

\newtheorem{notn}[subsubsection]{Notation}

\numberwithin{equation}{subsubsection}

\def\nummultline{\addtocounter{subsubsubsection}{1}\begin{multline}}
\def\anumequation{\addtocounter{subsection}{1}\begin{equation}}

\newif\iffinalrun
\iffinalrun
\else
\fi

\iffinalrun
  \newcommand{\need}[1]{}
  \newcommand{\mar}[1]{}
\else
  \newcommand{\need}[1]{{\tiny *** #1}}
  \newcommand{\mar}[1]{\marginpar{\raggedright\tiny #1}}
\fi

\newcommand{\A}{\AA}
\newcommand{\C}{\CC}
\newcommand{\F}{\FF}

\newcommand{\Q}{\QQ}
\newcommand{\R}{\RR}
\newcommand{\Z}{\ZZ}

\newcommand{\D}{\cD}
\newcommand{\G}{\cG}

\renewcommand{\O}{\cO}

\newcommand{\m}{\frakm}
\newcommand{\p}{\frakp}

\newcommand{\X}{\frakX}

\renewcommand{\AA}{{\mathbb A}}

\newcommand{\CC}{{\mathbb C}}

\newcommand{\FF}{{\mathbb F}}

\newcommand{\QQ}{{\mathbb Q}}
\newcommand{\RR}{{\mathbb R}}

\newcommand{\ZZ}{{\mathbb Z}}

\renewcommand{\bf}{\ensuremath{\mathbf{f}}}

\newcommand{\cD}{{\mathcal D}}

\newcommand{\cG}{{\mathcal G}}

\newcommand{\cO}{{\mathcal O}}

\newcommand{\frakm}{\mathfrak{m}}

\newcommand{\frakp}{\mathfrak{p}}

\newcommand{\frakX}{\mathfrak{X}}

\newcommand{\Fp}{\F_p}

\newcommand{\Shs}{\overline{Sh}_U(H)}

\DeclareMathOperator{\Gal}{Gal}
\DeclareMathOperator{\GL}{GL}

\DeclareMathOperator{\PGL}{PGL}

\DeclareMathOperator{\PSL}{PSL}

\DeclareMathOperator{\SL}{SL}

\DeclareMathOperator{\Spec}{Spec}

\DeclareMathOperator{\Spf}{Spf}

\newcommand{\Frob}{\mathrm{Frob}}

\newcommand{\HT}{\mathrm{HT}}
\newcommand{\Id}{\mathrm{Id}}

\newcommand{\et}{\mathrm{\acute{e}t}}

\newcommand{\toisom}{\buildrel\sim\over\to}

\newcommand{\Spd}{\mathrm{Spd}}

\newcommand{\T}{\mathbb{T}}
\newcommand{\Spa}{\mathrm{Spa}}

\begin{document}

\title[On the cohomology of Hilbert modular varieties]{On the \'etale cohomology 
of Hilbert modular varieties with torsion coefficients}

\author{Ana Caraiani}
\address{Department of
  Mathematics, Imperial College London,
  London SW7 2AZ, UK}
\email{a.caraiani@imperial.ac.uk}
\author{Matteo Tamiozzo}
\address{Department of
  Mathematics, Imperial College London,
  London SW7 2AZ, UK}
\email{m.tamiozzo@imperial.ac.uk}

\maketitle

\begin{abstract}
We study the \'etale cohomology of Hilbert modular varieties, 
building on the methods introduced for unitary Shimura varieties in~\cite{cs17, cs19}. 
We obtain the analogous vanishing theorem: in the ``generic'' case, 
the cohomology with torsion coefficients is concentrated in the middle degree. 
We also probe the structure of the cohomology beyond the generic case, 
obtaining bounds on the range of degrees where cohomology with torsion 
coefficients can be non-zero. 
The proof is based on the geometric Jacquet--Langlands functoriality 
established by Tian--Xiao
and avoids trace formula computations for the cohomology of Igusa varieties. 

As an application, we show that, when $p$ splits completely in the totally real field 
and under certain technical assumptions, the $p$-adic local Langlands correspondence 
for $\GL_2(\Q_p)$ occurs in the completed homology of Hilbert modular varieties.  
\end{abstract}

\setcounter{tocdepth}{1}
\tableofcontents

\section{Introduction} 

\subsection{Statement of results}

In this paper, we study the \'etale cohomology of Hilbert modular varieties, 
building on the methods introduced for unitary Shimura varieties in~\cite{cs17, cs19}. 

Let us first discuss a general vanishing conjecture for the 
cohomology of locally symmetric spaces. 
Let $G/\mathbb{Q}$ be a connected reductive group and let $X$ be the symmetric space for $G(\R)$, 
in the sense of~\cite{borel-serre}. For a neat compact open subgroup $K\subset G(\A_f)$, we can 
consider the associated locally symmetric space $X_K(G)$. 
Also define the invariants
\[
l_0:= \mathrm{rk}(G(\R)) - \mathrm{rk}(K_{\infty}) - \mathrm{rk}(A_{\infty})\ \mathrm{and}\ 
q_0:= \frac{1}{2}(\dim_{\R} X - l_0)\footnote{Here, $K_{\infty}\subset G(\R)$ is a maximal 
compact subgroup and $A_{\infty}$ is the group of $\R$-points of the maximal 
$\Q$-split torus in the centre of $G$.}. 
\] 
A folklore conjecture predicts that the cohomology 
of the locally symmetric space $X_K(G)$ with $\Z_{\ell}$-coefficients vanishes, after imposing an appropriate
non-degeneracy condition, outside the range of degrees $[q_0, q_0 +l_0]$ (which is symmetric 
about $\frac{1}{2}\dim_{\R} X$). See, for example,  
the discussion around~\cite[Conjecture 3.3]{emerton-icm} 
and also~\cite[Conjecture B]{calegari-geraghty}. As these references
explain, this conjecture has important consequences for
automorphy lifting theorems and the $p$-adic Langlands programme.

When $F$ is a totally real field and $G:=\mathrm{Res}_{F/\Q}\GL_2$, the corresponding locally 
symmetric spaces are closely related to Hilbert modular varieties, which are 
(non-compact) Shimura varieties of abelian type. 
When working with Shimura varieties
rather than with locally symmetric spaces, it is more natural to consider 
$l_0(G^{\mathrm{ad}})$, which is equal to $0$ by the second axiom 
in the definition of a Shimura datum. In this case,  
the conjecture mentioned above predicts that 
the non-degenerate part of the cohomology is concentrated in the middle degree. 

To make this more precise, let 
$K\subset G(\A_{f})$ be a neat compact open subgroup and let $Sh_K(G)$ be the corresponding 
Hilbert modular variety. This is a smooth, quasi-projective scheme over $\Q$, of dimension $g:=[F:\Q]$. 
Its complex points can be described as
\[
Sh_K(G)(\C) = G(\Q) \backslash (\C\smallsetminus \R)^g\times G(\A_f)/K. 
\]
Let $\ell$ be a prime number; we have a spherical Hecke algebra $\T$ defined in \S~\ref{hecke-alg}, generated by 
the standard Hecke operators $T_v, S_v^{\pm 1}$ for $v$ not belonging to a suitable finite set $S$ of places of $F$. 
The algebra $\T$ acts on the \'etale cohomology groups $H^i_{(c)}(Sh_K(G), \F_{\ell})$\footnote{This denotes the 
\'etale cohomology (with compact support) of the variety base changed to $\bar{\Q}$, but we suppress 
$\bar{\Q}$ to simplify the notation.}. 
Take a maximal ideal $\m \subset \T$ in 
the support of $H^i_{(c)}(Sh_K(G), \F_{\ell})$. It follows from Scholze's work, at least when $\ell>2$, cf. 
Theorem \ref{constr-gal-repn}, that there exists a
unique continuous, semisimple Galois representation 
\[
\bar{\rho}_{\m}: \Gamma_F \to \GL_2(\bar{\F}_{\ell}),
\]
where $\Gamma_F$ denotes the absolute Galois group of $F$, 
which is characterised as follows:  for every $v \not\in S$, $\bar{\rho}_\m$ is 
unramified at $v$ and the characteristic 
polynomial of $\bar{\rho}_{\m}(\mathrm{Frob}_v)$ is equal to $X^2-T_vX+S_vN(v) \pmod {\m}$.  

We say that $\m$ is \emph{non-Eisenstein} if $\bar{\rho}_{\m}$ is absolutely 
irreducible. The most optimistic vanishing conjecture 
predicts that the localisation $H^i_{(c)}(Sh_K(G), \F_{\ell})_{\m}$ should
be concentrated in the middle degree $i=g$ if $\m$ is non-Eisenstein. 
We make significant progress towards this. 

\begin{theorem}\label{thm:generic} (see Theorem~\ref{mainpart2}) 
Let $\ell>2$ be a prime and let $\m \subset \T$ be 
a maximal ideal in the support of $H^i_{c}(Sh_K(G), \F_{\ell})$ or 
$H^i(Sh_K(G), \F_{\ell})$. 
Assume that the image of $\bar{\rho}_{\m}$ is not solvable.
Then $H^i(Sh_K(G), \F_{\ell})_\m=H^i_c(Sh_K(G), \F_{\ell})_\m$ is non-zero only for $i=g$.
\end{theorem}

\begin{remark}\label{dicksonrem}\leavevmode
\begin{enumerate}
\item By Dickson's theorem, cf.~\cite[Theorem 2.47 (b)]{ddt}, the projective image of $\bar{\rho}_{\m}$ 
is either conjugate to a subgroup of the upper triangular matrices, or to $\PGL_2(\F_{\ell^k})$ or $\PSL_2(\F_{\ell^k})$,
for some $k\geq 1$, or it
is isomorphic to one of $D_{2n}$, for some $n\in \Z_{> 1}$ prime to $\ell$, 
$A_4$, $S_4$, or $A_5$. The image of $\bar{\rho}_\m$ is not solvable if and only if the following condition is satisfied:
\begin{enumerate}
\item if $\ell = 3$, the projective image of $\bar{\rho}_{\m}$ 
is isomorphic to $A_5$ or contains a conjugate of $\PSL_2(\F_9)$;
\item if $\ell>3$, the projective image of $\bar{\rho}_{\m}$
is isomorphic to $A_5$ or contains a conjugate of $\PSL_2(\F_{\ell})$.
\end{enumerate}
\item If $\m$ is non-Eisenstein, then $H^i(Sh_K(G), \C)_\m$ 
is concentrated in the middle degree. 
This follows from the explicit description of the 
cohomology with complex coefficients of Hilbert modular varieties (see \cite[Chapter III]{frei13}).
\item Previously, Dimitrov had obtained a vanishing theorem
for the cohomology of Hilbert modular varieties with torsion coefficients 
in~\cite[Theorem 2.3]{dim09} (see also~\cite{dim05}), 
under stronger assumptions. 
More precisely, Dimitrov proves a theorem for cohomology with 
coefficients in certain local systems on $Sh_K(G)$. In addition to 
a large image assumption on $\bar{\rho}_{\m}$, he also requires that
the level is prime to $\ell$ and that $\ell$ is large compared to the 
weight giving rise to the local system. Since we make no assumption 
on the level at $\ell$, a standard argument using the Hochschild--Serre 
spectral sequence allows us to upgrade Theorem~\ref{thm:generic} 
to also apply to cohomology with twisted coefficients.  
\end{enumerate}
\end{remark}

Let $p\not = \ell$ be a prime which splits completely in $F$ and such that 
$\bar{\rho}_{\m}$ is unramified at every place of $F$ above $p$. 
If $v$ is such a place, we say that $\bar{\rho}_{\m}$ is \emph{generic 
at $v$} if the eigenvalues of $\bar{\rho}_{\m}(\Frob_v)$ 
have ratio different from $p^{\pm 1}$. If the projective image of 
$\bar{\rho}_{\m}$ satisfies the condition in Remark \ref{dicksonrem}$(1)$, 
then the Chebotarev density theorem implies, cf. Lemma~\ref{cebot}, 
that there exists a prime $p$ as above such that $\bar{\rho}_{\m}$ is generic
at every $v\mid p$.\footnote{Note, however, that Frank Calegari has produced
an example of an absolutely irreducible mod $\ell$ Galois representation with projective image
$D_4 \simeq (\Z/2\Z)^2$ for which no such prime $p$ exists.}
As in~\cite{cs17, cs19}, Theorem~\ref{thm:generic}
relies on the study of the cohomology of (perfectoid) Igusa varieties and on the geometry of the Hodge--Tate period map at the auxiliary prime $p$.

The key new idea in our situation is to establish a geometric
Jacquet--Langlands transfer comparing the cohomology of Igusa varieties attached to different quaternionic Shimura varieties. This replaces the direct computation of the cohomology of Igusa varieties via the trace formula carried out in \cite{cs17}. We also exploit the relation between Igusa varieties and fibres of the Hodge--Tate period map for compact quaternionic Shimura varieties.
We explain this idea more in \S~\ref{sec:method}.

Furthermore, Theorem~\ref{thm:generic} is obtained as a special case 
of the following more precise result, that probes the structure of the cohomology beyond
the generic case. Given a prime $p$ which splits completely in $F$ 
and such that $\bar{\rho}_{\m}$ is unramified at every place of $F$ above $p$, 
denote by $\delta_p(\m)\in [0,g]$ the number of places above $p$ where 
$\bar{\rho}_{\m}$ is \emph{not} generic.  

\begin{theorem}\label{thm:beyond generic} (see Theorem~\ref{mainpart2-ref}) 
Let $\ell>2$ be a prime. Let 
$p\not = \ell$ be an odd prime which splits completely in $F$ and such that 
$K = K^pK_p$ with $K_p$ hyperspecial.  Let $\m\subset \T$ be a 
non-Eisenstein maximal ideal. Then 
\[
H^i_c(Sh_K(G), \F_{\ell})_{\m} \toisom H^i(Sh_K(G), \F_{\ell})_{\m}
\]
vanishes outside $i\in [g- \delta_p(\m), g+\delta_p(\m)]$. 
\end{theorem}

\noindent Theorem~\ref{thm:beyond generic} is inspired
by Arthur's conjectures~\cite{arthur}, describing the interplay between the Hecke action and the
Lefschetz structure on the cohomology of Shimura 
varieties with $\C$-coefficients. Our result 
is consistent with the existence of such a structure 
on the cohomology with $\F_{\ell}$-coefficients as well. 

In the case of Harris--Taylor unitary Shimura varieties, 
Boyer established the analogue
of Theorem~\ref{thm:beyond generic} in~\cite{boyer} 
(see also the discussion around~\cite[Theorem 1.3]{koshikawa}). 
Our argument is partly inspired by Boyer's, 
but it is different: we rely on the ingredients
mentioned above, 
as well as on the affineness and smoothness of Newton strata.
This allows us to apply 
Artin vanishing at a crucial step in the proof, 
followed by Poincar\'e duality. 

\begin{remark}\leavevmode
\begin{enumerate}
\item If $\delta_p(\m) = 0$, we expect that
$H^i_c(Sh_K(G), \F_{\ell})_{\m} =  0$ 
for $i>g$ (and, dually, $H^i(Sh_K(G), \F_{\ell})_{\m}=0$ for $i<g$) 
even without the 
non-Eisenstein assumption on $\m$. This is
the analogue of~\cite[Theorem 1.1]{cs19} in our context. 
To establish this, one may need the semi-perversity 
result~\cite[Theorem 4.6.1]{cs19} 
in our setting, which relies on a detailed
understanding of toroidal compactifications 
of Igusa varieties. 
 
\item Using the same method based on
geometric Jacquet--Langlands, we
also obtain analogues of Theorems~\ref{thm:generic} 
and~\ref{thm:beyond generic} for compact quaternionic Shimura varieties,
without the non-Eisenstein assumption:
see Theorem~\ref{quat-maithm}.
\end{enumerate}
\end{remark}

As a quick application of Theorem~\ref{thm:generic}, we show that the $p$-adic local
Langlands correspondence occurs in the completed homology of Hilbert modular varieties.
More precisely, fix a prime $p>3$ which splits completely in $F$;
in the following discussion, the prime $p$ will play the role of the prime denoted by $\ell$ above.
Fix a large enough finite extension $L/\Q_p$ with ring of integers $\O$. Under suitable assumptions (e.g. hypothesis $(3)$ in
Theorem~\ref{thm:p-adic Langlands} below) we can attach to the restriction of $\bar{\rho}_{\m}$ to places
$v\mid p$ a universal local deformation 
ring $R_p^{\mathrm{loc}} :=\hat{\otimes}_{v\mid p, \O} R_v^{\mathrm{def}}$ 
and an $R_p^{\mathrm{loc}}$-module 
$P: =\hat{\otimes}_{v \mid p, \O}P_v$. The latter represents a large part of 
the $p$-adic local Langlands correspondence 
in this setting, cf.~\cite{paskunas}. 
(See \S~\ref{sec:p-adic local Langlands} for more details on the notation.)

\begin{theorem}\label{thm:p-adic Langlands}
Let $p>3$ be a prime which splits completely in $F$. 
Assume that the following assertions hold true.
\begin{enumerate}
\item The projective image of the Galois representation $\bar{\rho}_\m$ 
attached to $\m$ contains a conjugate of $\PSL_2(\F_p)$ or is isomorphic to $A_5$.
\item If $\bar{\rho}_\m$ is ramified at some place $v$ 
not lying above $p$, then $v$ is not a vexing prime.
\item For each place $v \mid p$, the restriction of 
$\bar{\rho}_\m$ to $\Gamma_{F_v}$ is absolutely irreducible.
\end{enumerate}
Then the completed homology $\tilde{H}_g(\bar{K}_1(N(\bar{\rho}_\m))^p, \O)_\m$ 
can be described as a module over $\mathbb{T}(\bar{K}_1(N(\bar{\rho}_\m))^p)_\m[\bar{G}(\Q_p)]$ as
\begin{equation*}
\tilde{H}_g(\bar{K}_1(N(\bar{\rho}_\m))^p, \O)_\m \simeq 
\mathbb{T}(\bar{K}_1(N(\bar{\rho}_\m))^p)_\m \hat{\otimes}_{R_p^{\mathrm{loc}}}P^{\oplus m}
\end{equation*}
for some $m \geq 1$.
\end{theorem}

\noindent Theorem~\ref{thm:p-adic Langlands} points towards an extension of 
Emerton's landmark local-global compatibility theorem
for modular curves~\cite{emerton-lgc}
to the case of Hilbert modular varieties. Here, 
we apply the axiomatic approach via patching developed 
in~\cite{ceggps2, gee-newton} in order to obtain an unconditional ``proof of concept''.
We note that, for this application, it is crucial to know that 
completed homology is a projective object in an appropriate 
category of $\prod_{v\mid p}\PGL_2(\Z_p)$-modules. 
This relies on knowing Theorem~\ref{thm:generic} with twisted coefficients 
coming from arbitrary Serre weights, something that is not available
from the earlier results of~\cite{dim09}. In Remark~\ref{rem:further lgc},
we also sketch how to obtain a
version of compatibility with the $p$-adic 
local Langlands correspondence without
any assumptions on the tame level. 

Theorem~\ref{thm:generic} should have numerous other applications
to the $p$-adic Langlands programme for $\GL_2$ over a totally 
real field $F$, including in the case when $p$ ramifies in $F$. 
Once cohomology is concentrated in one degree,  
one can combine the Taylor--Wiles--Kisin patching method
with purely local techniques, as in~\cite{egs, bhs} (for example). 
Traditionally, these methods have been used for Shimura sets
attached to definite unitary groups or for Shimura curves. The 
advantage of Hilbert modular varieties is that their 
(co)homology may be more suitable to studying low-weight forms,
such as Hilbert modular forms of parallel weight $1$. 

\subsection{The method of proof}\label{sec:method}
As we mentioned above, the proof of Theorems~\ref{thm:generic}
and~\ref{thm:beyond generic} uses the geometry of the Hodge--Tate 
period morphism at an auxiliary prime $p \neq \ell$ for quaternionic Shimura varieties, particularly the relation between 
fibres of this morphism and perfectoid Igusa varieties.
The key new idea in the proof, inspired by the work of Tian--Xiao on the Goren--Oort stratification on quaternionic Shimura varieties ~\cite{tixi16}, is to establish instances of geometric Jacquet--Langlands
functoriality for Igusa varieties. This idea, which was not present in either of~\cite{cs17,cs19},
allows us to handle the cohomology of Igusa varieties. We exploit this 
(and further ingredients mentioned below) to transfer systems of 
Hecke eigenvalues in the cohomology of non-ordinary strata from 
Hilbert modular varieties to (perfectoid) compact quaternionic Shimura varieties. 
Using this we show that, in the setting of 
Theorem~\ref{thm:generic}, cohomology localised at $\m$ must be
supported on the ordinary locus; it is then relatively easy to prove that the 
cohomology is concentrated in one degree.

We illustrate a version of the key new idea mentioned above first using the toy model 
of the modular curve, the Shimura variety for $G:=\GL_2/\Q$.
Let $K^p\subset G(\A_f)$ be a sufficiently small 
compact open subgroup. By~\cite[\S 3]{sch15},
we have the Hodge--Tate period morphism 
$\pi_{\HT}: \mathcal{S}h_{K^p}(G) \to \mathbb{P}^{1,\mathrm{ad}}$
and the Hodge--Tate period domain admits the Newton stratification 
\[
\mathbb{P}^{1,\mathrm{ad}} = \mathbb{P}^1(\Q_p)
\bigsqcup \Omega^2.
\]
Here, $\mathbb{P}^1(\Q_p)$ is the closed stratum corresponding 
to the ordinary locus 
and $\Omega^2$, the Drinfeld upper half plane, is the open
stratum corresponding to the supersingular locus. 
For every geometric 
point $x: \Spa (C, \mathcal{O}_C)\to \Omega^2$, the fibre $\pi_{\HT}^{-1}(x)$ 
is a perfectoid version of a supersingular Igusa variety. 
In view of~\cite[Corollary 3.7.4]{howe} (which originates in an observation of Serre~\cite{serre}), 
this can, in turn, be identified with the double coset
$D^\times \backslash D^\times (\A_f)/K^p$, 
where $D/\Q$ is the quaternion algebra that is ramified
precisely at $p$ and $\infty$. 
If $\m\subset \mathbb{T}$ is a non-Eisenstein 
maximal ideal such that $(R\pi_{\HT*}\F_{\ell})_{\m}$ is supported on 
$\Omega^2$, then $\m$ is also in the support of 
$H^0(D^\times \backslash D^\times (\A_f)/K^p, \Q_{\ell})$. 
This means that there exists a cuspidal automorphic representation
of $D^\times$ whose 
associated Galois representation is a characteristic $0$
lift of $\bar{\rho}_{\m}$.  

However, if we choose $p$ such that $\bar{\rho}_{\m}|_{\Gamma_{\Q_p}}$ 
is generic, then any characteristic $0$ lift corresponds, under
the classical local Langlands correspondence, to a generic
principal series representation of $\GL_2(\Q_p)$. Such a representation 
cannot be the local component at $p$ of the
global Jacquet--Langlands transfer of a cuspidal automorphic 
representation of $D^\times$. Therefore, if $\bar{\rho}_{\m}$ 
is generic at $p$, then the complex of sheaves $(R\pi_{\HT*}\F_{\ell})_{\m}$ 
is supported only on the ordinary locus $\mathbb{P}^1(\Q_p)$.  

Morally, the Hilbert case for a totally real field $F$ of degree $g$ 
behaves like a product of $g$ copies of the modular curve case. 
In this setting, Tian--Xiao establish geometric instances of the 
Jacquet--Langlands correspondence in~\cite{tixi16}. 
Inspired by their result in the case when $p$ splits completely 
in $F$, we express the geometric Jacquet--Langlands relation as a Hecke-equivariant isomorphism 
between Igusa varieties attached to different
Shimura varieties, cf. Theorem~\ref{isoigusa}\footnote{To 
make sense of Igusa varieties, we want an integral model
with a nice moduli interpretation. This is not directly available
for quaternionic Shimura varieties. For this reason, in
the body of the paper, we work
with certain auxiliary unitary Shimura varieties that are
closely related to Hilbert and quaternionic Shimura varieties; 
we ignore this point throughout the introduction.}.
In Theorem~\ref{ord-par-ind},
we obtain a clean geometric description 
of the $\mu$-ordinary locus at infinite level on quaternionic Shimura varieties. 
This relies on (a perfectoid version of) the Mantovan product formula \cite{man05}, 
established in \S~\ref{sec:geometryHT}, and on the 
induced structure of the corresponding Rapoport--Zink spaces,
cf.~\cite{han18, gaim16}. Theorems~\ref{isoigusa} and~\ref{ord-par-ind} 
allow us to transfer systems of mod 
$\ell$ Hecke eigenvalues between different 
Shimura varieties, which forms the basis for
an inductive argument to prove 
Theorem~\ref{thm:beyond generic}. 

Some technical difficulties we encounter are that
Hilbert modular varieties are Shimura varieties of abelian type 
and that they are non-compact. However, the non-compactness
does not impose serious difficulties. The only Newton stratum
that intersects the boundary of Hilbert modular varieties is the 
ordinary one. This, together with the geometric 
Jacquet--Langlands relation on the level of interior Igusa varieties, 
makes the analysis of interior Newton strata 
similar to the one in~\cite{cs17}. Because we are localising
at a non-Eisenstein maximal ideal, we avoid employing 
partial minimal and toroidal compactifications of Igusa varieties 
or semi-perversity; in particular, our argument is independent of~\cite{cs19}, and only relies on the (easier) study of the geometry of the Hodge--Tate period map in the compact setting of ~\cite{cs17}.

Recently, Koshikawa~\cite{koshikawa21} 
gave a new strategy for proving the
vanishing theorems of~\cite{cs17,cs19}, removing 
most technical assumptions from these results. 
This strategy relies on~\cite{fargues-scholze}
together with the geometry of the Hodge--Tate period morphism.
The idea is to show that only the cohomology of the ordinary 
locus contributes to the generic part of the cohomology 
of Shimura varieties, by proving a local vanishing 
theorem for the generic part of the cohomology of Rapoport--Zink spaces. 
After reducing to the ordinary locus, ~\cite{koshikawa21}
uses the semi-perversity result mentioned above 
to control the range of degrees of cohomology.  
Koshikawa's arguments could be applied in the Hilbert setting as well,
though this would require some version of our results in 
\S~\ref{sec:geometryHT} and semi-perversity. 
One advantage of our method is that
it gives information about the complexes of sheaves 
$(R\pi_{\HT*}\F_{\ell})_{\m}$ rather 
that just the cohomology groups $H^*(Sh_K(G), \F_{\ell})_{\m}$. 

\subsection{Acknowledgements} 
We are grateful to David Hansen for asking us to consider
the cohomology of Hilbert modular varieties with torsion coefficients.
A.C. would like to thank Xinwen Zhu for the invitation to a 2019 AIM workshop 
on geometric instances of the Jacquet--Langlands correspondences,  
which inspired our method. We are grateful to 
Frank Calegari, Toby Gee, James Newton, Vytautas Pa{\v{s}}k\={u}nas, and Jack Sempliner
for useful conversations
and correspondence, and to Christophe Cornut, Toby Gee, Jo\~ao Louren\c{c}o, Stefano Morra, and Peter Scholze for
comments on a previous draft of this paper.

This project has received funding from the European Research Council (ERC)
under the European Union’s Horizon 2020 research and innovation programme
(grant agreement No. 804176). A.C. was supported by a Royal Society University
Research Fellowship. 

\subsection{Notation}  We use the following notation throughout the paper, unless otherwise stated.

The cardinality of a set $T$ is denoted by $|T|$. 

The symbol $\otimes$ stands for $\otimes_{\Z}$. If $A$ is an abelian 
group, we let $\hat{A} = A\otimes \hat{\Z}$ and 
$\hat{A}^{(p)} = A\otimes \hat{\Z}^{(p)}$.  We also set $\A_f = \hat{\Q}$, the finite adeles of $\Q$, 
and $\A_f^{(p)} = \hat{\Q}^{(p)}$, the finite adeles of $\Q$ away from $p$. 

If $L$ is a perfect field, we denote by $\bar{L}$ 
an algebraic closure of $L$, and by $\Gamma_L$ the absolute Galois 
group $\Gal(\bar{L}/L)$. Assume now that $L$ is a number field. 
If $v$ is a finite place of $L$, we denote by $L_v$ the completion of $L$ at $v$, 
and by $\cO_v$ the ring of integers of $L_v$. We will denote by 
$\varpi_v$ a uniformiser of $\cO_v$, and by $N(v)$ the cardinality 
of the residue field $\cO_v/(\varpi_v)$. We denote by $\Frob_v$ the geometric
Frobenius at $v$. We let $\A_{L,f} = \hat{L}$,
the ring of finite adeles of $L$.

If $A\to B$ is a ring morphism and $S\to \Spec A$ is a scheme, we 
denote by $S_B$ the base change $S\times_{\Spec A} \Spec B$. We use a similar notation for adic spaces. If 
$S$, $T$ are two schemes over $\Q$, we denote by $S\times T$ 
the fibre product $S\times_{\Spec \Q} T$.  

If $L$ is a field and $S$ is a variety or adic space over $L$, the symbols
$H^i(S, \F_{\ell})$, $H^i(S, \Z_{\ell})$ and $H^i(S, \Q_{\ell})$ denote
the \'etale (resp. $\ell$-adic \'etale) cohomology of $S_{\bar{L}}$. 

\section{Hilbert modular varieties and Galois representations}

In this section, we establish some preliminary results about Hilbert modular varieties. In particular, 
we construct Galois representations attached to systems of Hecke eigenvalues 
occurring in their cohomology with $\F_{\ell}$-coefficients, using~\cite{sch15}. 

\subsection{Shimura varieties and locally symmetric spaces for $\mathrm{Res}_{F/\Q}\mathrm{GL}_2$} Fix a totally real number field $F$ of degree $[F: \Q]=: g$, and let $\O_F$ be the ring of integers of $F$. Totally positive elements in $F$ (resp. $\O_F$) will be denoted by $F^+$ (resp. $\O_F^+$). Let $\Sigma_{\infty}: =\{\tau_1, \ldots, \tau_g\}$ be the set of real embeddings of $F$; let $G:=\mathrm{Res}_{F/\Q}\mathrm{GL}_2$ with centre $Z \simeq \mathrm{Res}_{F/\Q}\mathbb{G}_m$.

\subsubsection{} Let $\mathbb{S}:=\mathrm{Res}_{\C/\R}\mathbb{G}_m$ be the Deligne torus, and let $h : \mathbb{S} \rightarrow G_{\R}\simeq \prod_{i=1}^g \GL_{2, \R}$ be the morphism which on $\R$-points is given by 
\[
z=a+ib \in \mathbb{S}(\R) \mapsto \left(\begin{smallmatrix}
a & b\\
-b & a
\end{smallmatrix}\right)^g \in G(\R).
\] 
Let $K^\circ_\infty: =\prod_{i=1}^g\mathrm{SO}_2(\R)\subset G(\R)$ (a maximal compact connected subgroup). The $G(\R)$-conjugacy class of $h$, denoted by $X$, is identified with $G(\R)/Z(\R)K^\circ_\infty\simeq (\C\smallsetminus \R)^g$, and the couple $(G, X)$ is a Shimura datum with reflex field $\Q$. For every compact open subgroup $K \subset G(\A_f)$ the Shimura variety $Sh_K(G)$ is a quasi-projective variety with a canonical model over the reflex field $\Q$, with complex analytic uniformisation
\begin{equation*}
Sh_K(G)(\C)=G(\Q)\backslash X\times G(\A_f)/K.
\end{equation*}
The Shimura varieties $Sh_K(G)$ are called \emph{Hilbert modular varieties}. If $K\subset G(\A_f)$ is neat (which we will always assume in what follows) then $Sh_K(G)$ is smooth; moreover, if $K'\subset K$ is a normal compact open subgroup then the map $Sh_{K'}(G)\rightarrow Sh_{K}(G)$ is a finite étale Galois cover.

We have a map
\begin{equation}\label{conn-cpts}
Sh_K(G)(\C) \rightarrow F^\times \backslash \{\pm 1\}^g \times \A_{F, f}^\times / \det(K)\simeq F^{\times, +} \backslash \A_{F, f}^\times / \det(K)
\end{equation}
induced by the map sending $(x_\infty, x_f) \in G(\R) \times G(\A_f)$ to $(\mathrm{sgn} \det x_\infty, \det x_f)$. An element in $F^\times$ acts on $\{\pm 1\}^g$ by multiplying by the sign of the image via each real embedding. Fibres of the map in \eqref{conn-cpts} are connected components of the source.

\subsubsection{Hecke action}\label{hecke-alg} We have the Hecke algebra $\mathbb{T}_K(G) :=\Z\left[K\backslash G(\A_f)/K\right]$ of compactly supported, $K$-bi-invariant functions on $G(\A_f)$, with multiplication given by convolution. Every element of $\mathbb{T}_K(G)$ gives rise to a correspondence on $Sh_K(G)$ as follows: given $g \in G(\A_f)$, let $K_g:=K \cap g K g^{-1}$; we have a correspondence $[KgK]$ on $Sh_K(G)$ given by the following diagram:

\begin{center}
\begin{tikzcd}
Sh_{K_g}(G) \arrow[r] \arrow[d] & Sh_{K_{g^{-1}}}(G) \arrow[d] \\
Sh_K(G) \arrow[r, dashed, "\lbrack KgK\rbrack"] & Sh_K(G),
\end{tikzcd}
\end{center}
where the vertical maps are the canonical projections and the upper horizontal map on complex points is induced by right multiplication by $g$ on $G(\A_f)$. 
Therefore we obtain an action of $\mathbb{T}_K(G)$ on the cohomology groups $H^i(Sh_K(G), \F_\ell)$ as well as on the cohomology groups with compact support $H^i_{c}(Sh_K(G), \F_\ell)$, where $\ell$ is a prime number. In the rest of the paper we will rather work with a smaller Hecke algebra, defined as follows: fix a finite set $S$ of places of $F$ containing
all the infinite places, all the places $v\mid \ell$, and all the finite places $v$ 
such that $K_v$ is not conjugate to $\GL_2(\cO_v)$, where $\cO_v$ is the ring of 
integers in the completion $F_v$ of $F$ at $v$.  
Let 
\[
\T:= \bigotimes_{v\not \in S} {' \; \Z[\GL_2(\cO_v)\backslash \GL_2(F_v)/\GL_2(\cO_v)]}
\] 
denote the abstract spherical Hecke algebra away from $S$. For 
every $v\not\in S$, choose a uniformiser $\varpi_v$ of $\cO_v$. 
We denote by $T_v$ (resp. $S_v$) the double coset $\GL_2(\cO_v)\left(\begin{smallmatrix}
\varpi_v & 0\\
0 & 1
\end{smallmatrix}\right)\GL_2(\O_v)$ (resp. $\GL_2(\O_v)\left(\begin{smallmatrix}
\varpi_v & 0\\
0 & \varpi_v 
\end{smallmatrix}\right)\GL_2(\O_v)$) seen as an element of $\T$. The algebra $\T$ is commutative and is generated by the operators $T_v$, $S_v^{\pm 1}$ for $v \not \in S$. The inversion map on $G$ induces a map $\iota: \mathbb{T}\rightarrow \mathbb{T}$; for a maximal ideal $\m \subset \mathbb{T}$ we denote by $\m^\vee$ its image via $\iota$. 

\subsubsection{Locally symmetric spaces} We have the locally symmetric space $X_K(G)$ attached to $G$ and $K$, defined as follows:
\begin{equation*}
X_K(G):=G(\Q)\backslash G(\R)\times G(\A_f)/(\R_{>0} K_\infty^\circ) K.
\end{equation*}
Letting $Z_K:=Z(\Q)\cap K$, the quotient $T_K:=Z(\R)/\{\pm 1\}^g\R_{>0}Z_K$ is a torus of dimension $g-1$ by Dirichlet's unit theorem, and the projection $X_K(G) \rightarrow Sh_K(G)$ is a $T_K$-bundle (see \cite[Lemma 3.1.2]{gra16}). Let us define $K_\infty:=\prod_{i=1}^g\mathrm{O}_2(\R)$, and let us consider the space
\begin{equation*}
\bar{X}_K(G):=G(\Q)\backslash G(\R)\times G(\A_f)/(\R_{>0}K_\infty)K.
\end{equation*}
In other words, we are quotienting by a maximal compact subgroup at infinity instead of its connected component of the identity. This is the space used in \cite{sch15} (see the introduction of \emph{loc. cit.}). The matrix $\left(\begin{smallmatrix}-1 & 0\\
0 & 1\end{smallmatrix}\right)$ normalises $\mathrm{SO}_2(\R)$; hence we get a right action of $C:=(\Z/2\Z)^g$ on $X_K(G)$, and we have $\bar{X}_K(G)=X_K(G)/C$. The recipe given above endows the (Betti) cohomology of $X_K(G)$ and $\bar{X}_K(G)$ with an action of $\mathbb{T}$, which commutes with the action of $C$. Furthermore, if we make $C$ act on $F^\times \backslash \{\pm 1\}^g \times \A_{F, f}^\times / \det(K) $ by switching signs at archimedean places, then the map in \eqref{conn-cpts} is $C$-equivariant.

Let $\hat{C}$ be the set of characters of $C$ with values in $\F_\ell^\times$. The following lemma relates the cohomology of the spaces $Sh_K(G), X_K(G)$ and $\bar{X}_K(G)$.

\begin{lem}\label{compar-cohom}
Let $\ell$ be an odd prime.
\begin{enumerate}
\item Assume that $\det(K)=K \cap Z(\A_f)$. Then for every $i \geq 0$ the pullback map $H^i(Sh_K(G), \F_\ell) \rightarrow H^i(X_K(G), \F_\ell)$ is injective.
\item Assume that $\det(K) \cap F^\times$ only consists of totally positive elements. Then there is a Hecke-equivariant decomposition
\begin{equation*}
H^*(X_K(G), \F_\ell)=\oplus_{\chi \in \hat{C}}H^*(X_K(G), \F_\ell)^\chi
\end{equation*}
and pullback induces an isomorphism $H^*(\bar{X}_K(G), \F_\ell)\simeq H^*(X_K(G), \F_\ell)^{\Id}$.
\end{enumerate}
\end{lem}
\begin{proof}\leavevmode
\begin{enumerate}
\item Under the assumption that $\det(K)=K \cap Z(\A_f)$, \cite[Proposition 3.3.9]{gra16} constructs cohomology classes in $H^*(X_K(G), \F_\ell)$ whose restrictions to each fibre of the map $X_K(G)\rightarrow Sh_K(G)$ give a basis of its cohomology; the statement hence follows from the Leray--Hirsch theorem.
\item Since $\det(K) \cap F^\times$ consists of totally positive elements by assumption, an explicit computation shows that $C$ acts freely on $F^\times \backslash \{\pm 1\}^g \times \A_{F, f}^\times / \det(K)$. In other words, the group $C$ acts on $X_K(G)$ freely permuting connected components. This implies that $H^*(\bar{X}_K(G), \F_\ell)\simeq H^*(X_K(G), \F_\ell)^{\Id}$. The Hecke-equivariance of the direct sum decomposition in the statement follows from the fact that the action of $C$ commutes with the Hecke action.
\end{enumerate}
\end{proof}

\subsection{Construction of Galois representations}

The aim of this section is to prove the following result.

\begin{thm}\label{constr-gal-repn} Let $\ell$ be an odd prime and $K\subset G(\A_f)$ a neat compact open subgroup. Let 
$\m \subset \mathbb{T}$ be a maximal ideal in the support of $H^i(Sh_K(G), \F_{\ell})$ or $H^i_{c}(Sh_K(G), \F_{\ell})$ for some $i \geq 0$.
There is a unique continuous, semisimple, totally odd Galois representation 
\[
\bar{\rho}_{\m}: \Gamma_F\to \GL_2(\overline{\F}_{\ell})
\]
such that, for all but finitely many places $v$ of $F$, 
$\bar{\rho}_\m$ is unramified at $v$ and the characteristic polynomial of $\bar{\rho}_{\m}(\mathrm{Frob}_v)$ 
is equal to $X^2-T_vX+S_vN(v) \pmod {\m}$.
\end{thm}
\begin{rem}
One could prove the above result by adapting the arguments in Chapter IV of \cite{sch15}. The main technical point one needs to deal with is the construction of ad hoc compactifications of Hilbert modular varieties\footnote{This is slightly subtle because Hilbert modular varieties are Shimura varieties of abelian type and do not directly embed into Siegel modular varieties. However, one can handle this issue by carefully choosing the tame level.}. For the sake of brevity, we will instead explain below how to deduce the theorem from (a special case of) the main result of \cite{sch15}; however, at the time of writing, the totally real case of the latter is conditional on Arthur's work \cite{ar13}. 
\end{rem}
\begin{proof}
Uniqueness of $\bar{\rho}_\m$ follows from the Chebotarev density theorem and the Brauer--Nesbitt theorem; to prove existence of $\bar{\rho}_{\m}$ it suffices to consider the cohomology groups $H^*(Sh_K(G), \F_{\ell})$ (by Poincaré duality and the discussion in \cite[p. 35]{cs19}). Furthermore, if $K' \subset K$ is a normal compact open subgroup, then there is a Hecke-equivariant Hochschild--Serre spectral sequence relating cohomology of $Sh_{K'}(G)$ and $Sh_K(G)$. Therefore, at the cost of possibly enlarging the set $S$ in \S~\ref{hecke-alg}, we may replace $K$ by a normal compact open subgroup. In particular we may take $K=K(N)$ for $N$ large enough, where $K(N):=\{M \in GL_2(\hat{\mathcal{O}}_F)\mid M \equiv \left(\begin{smallmatrix} 1 & 0\\
0 & 1
\end{smallmatrix}\right) \pmod N\}$. This ensures that the assumption in the first point of Lemma \ref{compar-cohom} is satisfied. Furthermore Chevalley's theorem on units guarantees that we may choose $N$ in such a way that the hypothesis in the second point of Lemma \ref{compar-cohom} is satisfied as well. We fix such an integer $N$ from now on. By assumption there is $i \geq 0$ such that $H^i(Sh_K(G), \F_\ell)_\m\neq 0$. By Lemma \ref{compar-cohom} we deduce that $H^i(X_K(G), \F_\ell)_\m \neq 0$, hence $H^i(X_K(G), \F_\ell)_\m^\chi\neq 0$ for some $\chi \in \hat{C}$. If $\chi$ is the trivial character then  by Lemma \ref{compar-cohom} we have $H^i(\bar{X}_K(G), \F_\ell)_\m\neq 0$, and by \cite[Theorem 1.3]{sch15} and \cite[Theorem 1.2]{clh16} we can attach to $\m$ a Galois representation as in the statement of the theorem.

Let us now suppose that $\chi$ is not the trivial character. The map $\{\pm 1\}^g \rightarrow F^\times \backslash \{\pm 1\}^g \times \A_{F, f}^\times / \det(K)$ sending $\varepsilon$ to the equivalence class of $(\varepsilon, 1)$ is injective. We may extend $\chi: \{\pm 1\}^g\rightarrow \F_\ell^\times$ to a character $\psi: F^\times \backslash \{\pm 1\}^g \times \A_{F, f}^\times / \det(K)\rightarrow k^\times$ for a large enough finite extension $k$ of $\F_\ell$. The function $\psi$ gives rise to a cohomology class $c_\psi \in H^0(X_K(G), k)$. Let $f_\psi: \mathbb{T} \rightarrow \mathbb{T}$ the map sending the operator $T_g$ attached to the double coset of $g \in G(\A_f)$ to $\psi(\det(g)^{-1})T_g$. Cup product with $c_\psi$ induces a map $H^i(X_K(G), k)\rightarrow H^i(X_K(G), k)$ which is Hecke-equivariant if we endow the source (resp. target) with the usual Hecke action (resp. the composite of the usual Hecke action and $f_\psi$). This is proved in great generality in \cite[Proposition 2.2.14]{cal18}, and it can be checked in our situation by a direct computation.

Furthermore the above map sends $H^i(X_K(G), k)^\chi$ to $H^i(X_K(G), k)^{\Id}$. Indeed, for $\varepsilon \in C$ we have $\varepsilon^*(c_\psi)=\chi(\varepsilon)c_\psi$ hence, for every $c \in H^i(X_K(G), k)^\chi$:
\begin{equation*}
\varepsilon^*(c_\psi \cup c)=\varepsilon^*c_\psi \cup \varepsilon^*c=\chi(\varepsilon) c_\psi \cup \chi(\varepsilon)c=c_\psi \cup c.
\end{equation*}
The outcome of our discussion is that, letting $\m(\psi):=f_\psi(\m)$, we have the equivalence
\begin{equation*}
H^i(X_K(G), k)^\chi_\m \neq 0 \Leftrightarrow H^i(X_K(G), k)_{\m(\psi)}^{\Id}\neq 0.
\end{equation*}
We deduce that $H^i(\bar{X}_K(G), k)_{\m(\psi)}\neq 0$, hence we have a Galois representation $\bar{\rho}_{\m(\psi)}$ attached to $\m(\psi)$ by \cite[Theorem 1.3]{sch15}. We can finally take $\bar{\rho}_\m$ to be the twist of $\bar{\rho}_{\m(\psi)}$ by the character of $\Gamma_F$ corresponding to $\psi$ via global class field theory.
\end{proof}

\subsection{Cohomology of the boundary} Recall that a maximal ideal $\m\subset \mathbb{T}$ in the support of $H^*(Sh_K(G), \F_\ell)$ is said to be \emph{non-Eisenstein} if the associated Galois representation $\bar{\rho}_\m$ is absolutely irreducible.

\begin{lem}\label{cohom-noneis}
Let $\ell>2$ be a prime and assume that $\m$ is non-Eisenstein. Then:
\begin{enumerate}
\item for every integer $i$, the natural map from compactly supported cohomology to cohomology induces an isomorphism $H^i_c(Sh_K(G), \F_{\ell})_\m\simeq H^i(Sh_K(G), \F_{\ell})_\m$;
\item take $\delta \geq 0$ and assume that $H^i_c(Sh_K(G), \F_{\ell})_\m=H^i_c(Sh_K(G), \F_{\ell})_{\m^\vee}=0$ 
for each $i < g-\delta$. Then $H^i(Sh_K(G), \F_{\ell})_\m=0$ if $i$ is outside the interval $[g-\delta, g+\delta]$.
\end{enumerate}
\end{lem}

\begin{proof}
Let $\partial$ be the boundary of the Borel--Serre compactification $Sh_K(G)^{BS}$ of $Sh_K(G)$, constructed in \cite{har87}. 
Recall that the cohomology of $Sh_K(G)^{BS}$ agrees with that of $Sh_K(G)$, 
as the two spaces are homotopy equivalent, and the 
compactly supported cohomology of $Sh_K(G)$ coincides with the 
cohomology of $Sh_K(G)^{BS}$ relative to the boundary. Hence we have a $\mathbb{T}$-equivariant 
long exact sequence
\begin{equation*}
\ldots H^i_c(Sh_K(G), \F_{\ell}) \rightarrow H^i(Sh_K(G), \F_{\ell}) 
\rightarrow H^{i}(\partial, \F_{\ell}) \rightarrow H^{i+1}_c(Sh_K(G), \F_{\ell}) \ldots.
\end{equation*}
We claim that if $\m$ is non-Eisenstein then cohomology of the boundary vanishes after localising at $\m$; this imples the first point. The second follows using Poincaré duality (which interchanges $\m$ and $\m^\vee$, as in \cite[p. 35]{cs19}). It remains to justify our claim. By \cite[p. 46]{har87} we have $H^*(\partial, \F_\ell)\simeq H^*(P(\Q)\backslash (\C \smallsetminus \R)^g \times G(\A_f)/K, \F_\ell)$, where $P \subset G$ is the standard upper parabolic. To prove our claim we may replace $K$ by a normal compact open subgroup, so that the condition $\det(K)=K \cap Z(\A_f)$ is satisfied. We will assume that this is the case in the rest of the proof. The space $P(\Q)\backslash (\C \smallsetminus \R)^g \times G(\A_f)/K$ is a disjoint union of quotients $P(\Q)\cap g_iKg_i^{-1}\backslash (\C\smallsetminus \R)^g$, where $g_i$ runs over a set of representatives of $P(\Q)\backslash G(\A_f)/K$.

On the other hand we have the locally symmetric space $X_K(G)$ and its Borel--Serre compactification, whose boundary is homotopy equivalent to the double quotient $P(\Q) \backslash G(\R)\times G(\A_f)/(\R_{>0}K^\circ_\infty) K$ (see \cite[p. 94]{gra16}). The latter is a disjoint union of quotients $P(\Q)\cap g_iKg_i^{-1}\backslash G(\R)/(\R_{>0}K^\circ_\infty)$ with $g_i$ as above. The action of $Z(\R)$ on $G(\R)/(\R_{>0}K^\circ_\infty)$ induces an action on $P(\Q) \backslash G(\R)\times G(\A_f)/(\R_{>0}K^\circ_\infty) K$. This action preserves the fibres of the projection map
\begin{equation}\label{boundary-bundle}
P(\Q) \backslash G(\R)\times G(\A_f)/(\R_{>0}K^\circ_\infty) K\rightarrow P(\Q)\backslash (\C \smallsetminus \R)^g \times G(\A_f)/K,
\end{equation} 
and factors through an action of the torus $T_K=Z(\R)/\{\pm 1\}^g\R_{>0}Z_K$. The argument in the proof of \cite[Lemma 3.1.2]{gra16} shows that each projection map $P(\Q)\cap g_iKg_i^{-1}\backslash G(\R)/(\R_{>0}K^\circ_\infty) \rightarrow P(\Q)\cap g_iKg_i^{-1}\backslash (\C\smallsetminus \R)^g$ is a $T_K$-bundle. Hence the same is true for the map in \eqref{boundary-bundle}. Furthermore we have a commutative diagram
\begin{center}
\begin{tikzcd}
P(\Q) \backslash G(\R)\times G(\A_f)/(\R_{>0}K^\circ_\infty) K \arrow[r] \arrow[d]
& G(\Q)\backslash G(\R) \times G(\A_f)/(\R_{>0}K^\circ_\infty) K \arrow[d] \\
P(\Q) \backslash (\C \smallsetminus \R)^g\times G(\A_f)/K\arrow[r]
& G(\Q)\backslash (\C \smallsetminus \R)^g \times G(\A_f)/K.
\end{tikzcd}
\end{center}
As we already mentioned above, by \cite[Lemma 3.1.2]{gra16} the right vertical map is a $T_K$-bundle, and \cite[Proposition 3.3.9]{gra16} constructs cohomology classes in $H^*(X_K(G), \F_\ell)$ whose restrictions to each fibre of the right vertical map give a basis of its cohomology. Pulling back via the upper horizontal map we obtain classes in $H^*(P(\Q) \backslash G(\R)\times G(\A_f)/(\R_{>0}K^\circ_\infty) K, \F_\ell)$ enjoying the same property. In particular by Leray--Hirsch the pullback map in cohomology induced by the left vertical map is injective. Now assume that $H^*(\partial, \F_\ell)_\m=H^*(P(\Q) \backslash (\C \smallsetminus \R)^g\times G(\A_f)/K, \F_\ell)_\m\neq 0$. Then we deduce that the cohomology $H^*(P(\Q) \backslash G(\R)\times G(\A_f)/(\R_{>0}K^\circ_\infty) K, \F_\ell)_\m$ of the boundary of $X_K(G)$ is non-zero. Finally the argument in \cite[Section 4]{newto16} shows that $\m$ must be Eisenstein.
\end{proof}

\section{Quaternionic and unitary Shimura varieties}

In this section, we introduce quaternionic Shimura varieties
as well as certain closely related unitary Shimura varieties that 
admit nice integral models. 

\subsection{Quaternionic Shimura data and the associated unitary Shimura data}

\subsubsection{Quaternionic Shimura data}\label{quat-datum}

As in the previous section we denote $G:=\mathrm{Res}_{F/\Q}\mathrm{GL}_2$; we also let $T_F:=\mathrm{Res}_{F/\Q}\mathbb{G}_m$. Let $K\subset G(\A_f)$ be a neat compact open subgroup. Fix a prime $p>2$ which is totally split in $F$ and such that $K=K^{p}K_p$ with $K_p=\GL_2(\O_F\otimes\Z_p)$.

Recall that $\Sigma_\infty$ denotes the set of real places of $F$. We will denote by $\Sigma_p$ the set of embeddings of $F$ into $\bar{\Q}_p$, which we identify with the set of prime ideals $\p \subset \O_F$ lying above $p$. We fix an isomorphism $\iota_p : \bar{\Q}_p \toisom \C$, inducing a bijection $\iota_{p, \infty}: \Sigma_p \toisom \Sigma_\infty$. For every subset $T \subset \Sigma_p$ we set $T_\infty:=\iota_{p, \infty}(T)$ and we denote by $B_T$ the quaternion algebra over $F$ ramified precisely at $T \coprod T_\infty$. We let $G_T:=\mathrm{Res}_{F/\Q}B_T^\times$, and we fix an isomorphism $G_T(\A_f^{(p)})\simeq G(\A_f^{(p)})$. The group $G_{T, \R}$ is isomorphic to $\prod_{\tau \in T_\infty} \mathbb{H}^\times \times \prod_{\tau \in \Sigma_\infty \smallsetminus T_\infty} \GL_{2, \R}$ where $\mathbb{H}$ is the algebra of Hamilton quaternions. Let $X_T$ be the $G_T(\R)$-conjugacy class of the morphism $h_T: \mathbb{S} \rightarrow G_{T, \R}$ sending an $\R$-point $z=a+ib \in \mathbb{S}(\R)$ to $(z^\tau)_{\tau \in \Sigma_\infty}$, where $z^\tau=1$ for $\tau \in T_\infty$ and $z^\tau=\left(\begin{smallmatrix}
a & b\\
-b & a
\end{smallmatrix}\right)$ if $\tau \in \Sigma_\infty \smallsetminus T_\infty$. The couple $(G_T, X_T)$ is a \emph{weak Shimura datum} (in the sense of \cite[Section 2.2]{tixi16}) whose reflex field $F_T$ can be described as follows: the group $\mathrm{Aut}(\C/\Q)$ acts on $\Sigma_\infty$ by post-composition. Let $\Gamma_T \subset \mathrm{Aut}(\C/\Q)$ be the subgroup preserving $T_\infty$; then $F_T=\C^{\Gamma_T}\subset \C$.

Let $K_T=K_{T, p}K^{p} \subset G_T(\A_f)$ be the compact open subgroup such that $K^p\subset G_T(\A_f^{(p)})\simeq G(\A_f^{(p)})$ is the subgroup chosen above and $K_{T, p}=\prod_{\p \mid p}K_{T, \p}$ is of the following type: if $\p \in \Sigma_p \smallsetminus T$ then we fix an isomorphism $\rho_\p: B_T(F_\p)^\times\rightarrow \GL_2(F_\p)$ and we take $K_{T, \p}:=\rho_\p^{-1}(\GL_2(\mathcal{O}_\p))$. If $\p \in T$ then we take $K_{T, \p}$ to be the group of units in the unique maximal order in the division quaternion algebra $B_T \otimes_F F_\p$.

\subsubsection{The auxiliary $CM$ extension}\label{auxil-CM}

Choose a $CM$ extension $E/F$ such that every place $\p \in \Sigma_p$ is \emph{inert} in $E$; in particular $B_T \otimes_F E$ is isomorphic to the matrix algebra $M_2(E)$. Let $c \in \mathrm{Gal}(E/F)$ be the non trivial element and let $\Sigma_{E, \infty}$ be the set of complex embeddings of $E$: it comes with a restriction map $\Sigma_{E, \infty}\rightarrow \Sigma_\infty$ whose fibres are the orbits for the action of $\mathrm{Gal}(E/F)$ sending an embedding $\tilde{\tau} \in \Sigma_{E, \infty}$ to $\tilde{\tau}^c:=\tilde{\tau}\circ c$. Choose a set $\tilde{T} \subset \Sigma_{E, \infty}$ containing exactly one lift of each $\tau \in T_\infty$. For every $\tilde{\tau} \in \Sigma_{E, \infty}$, define an integer $s_{\tilde{\tau}} \in \{0, 1, 2\}$ as follows:
\begin{itemize}
\item if $\tilde{\tau}_{|F} \not \in T_\infty$ then $s_{\tilde{\tau}}=1$;
\item if $\tilde{\tau} \in \tilde{T}$ then $s_{\tilde{\tau}}=0$;
\item if $\tilde{\tau}^c \in \tilde{T}$ then $s_{\tilde{\tau}}=2$.
\end{itemize}

For each $\tau \in T_\infty$ (resp. $\tau \in \Sigma_\infty\smallsetminus T_\infty$) we choose the isomorphism $E\otimes_{F, \tau}\R\simeq \C$ induced by the embedding $\tilde{\tau} \in \Sigma_{E, \infty}\smallsetminus \tilde{T}$ lifting $\tau$ (resp. an arbitrary lift of $\tau$). Let $T_E:=\mathrm{Res}_{E/\Q}\mathbb{G}_m$; via the previous choices we obtain an isomorphism $T_{E}(\R)= \prod_{\tau \in \Sigma_\infty}(E\otimes_{F, \tau}\R)^\times\simeq \prod_{\tau \in \Sigma_\infty}\C^\times$. Let $h_{T, E}: \mathbb{S} \rightarrow T_{E, \R}$ be the morphism which on $\R$-points sends $z=a+ib \in \mathbb{S}(\R)$ to $(z^\tau)_{\tau \in \Sigma_\infty}$, where $z^\tau=z$ (resp. $z^\tau=1$) if $\tau \in T_\infty$ (resp. $\tau \in \Sigma_\infty \smallsetminus T_\infty$). The reflex field $E_T\supset F_T$ of the $T_E(\R)-$conjugacy class of $h_{T, E}$ is the subfield of $\C$ fixed by the stabiliser in $\mathrm{Aut}(\C/\Q)$ of $\tilde{T}$.

\subsubsection{Unitary Shimura data}\label{unit-datum} We denote by $H_T$ the algebraic group fitting in the exact sequence
\begin{equation*}
0 \rightarrow T_F \rightarrow G_T \times T_E \rightarrow H_T \rightarrow 0,
\end{equation*}
where the map $T_F \rightarrow G_T \times T_E$ is given by $a \mapsto (a, a^{-1})$. Notice that, because of Hilbert's theorem 90, the map $G_T \times T_E \rightarrow H_T$ induces surjections on $\Q$-points and $\A_f$-points. The $(G_T \times T_E)(\R)-$conjugacy class of the map $h_T\times h_{T, E}: \mathbb{S}\rightarrow (G_T \times T_E)_\R$ and the $H_T(\R)$-conjugacy class of the induced map $h_{H_T}: \mathbb{S}\rightarrow H_{T, \R}$ can both be identified with $X_T \simeq \prod_{\tau \in \Sigma_\infty \smallsetminus T_\infty}(\C \smallsetminus \R)$, and they give rise to weak Shimura data $(G_T \times T_E, X_T)$ and $(H_T, X_T)$ with reflex field $E_T$. We let $X_{T}^+:=\prod_{\tau \in \Sigma_\infty \smallsetminus T_\infty}(\C \smallsetminus \R)^+$, where $(\C \smallsetminus \R)^+ \subset (\C \smallsetminus \R)$ is the upper half-plane.

Let us denote by $D_T$ the tensor product $B_T\otimes_F E$ and by $\overline{(\cdot)}: D_T \rightarrow D_T$ the tensor product of the main involution on $B_T$ and complex conjugation on $E$. For every $\Q-$algebra $R$ we have a canonical isomorphism
\begin{equation*}
(B_T\otimes_\Q R)\otimes_{F \otimes_\Q R}(E\otimes_\Q R)\simeq (B_T\otimes_F E)\otimes_{\Q}R
\end{equation*}
hence we get a map $G_T(R) \times T_E(R) \rightarrow (D_T \otimes_\Q R)^\times$. This yields an identification 
\begin{equation*}
H_T(R)=\{g \in D_T\otimes_\Q R \mid g \bar{g} \in (F \otimes_\Q R)^\times\}.
\end{equation*}
The latter description allows to see $H_T$ as a unitary group as follows: given an element $\sigma \in D_T$ such that $\bar{\sigma}=-\sigma$ the map sending $g \in D_T$ to $g^*=\sigma^{-1}\bar{g}\sigma$ is an involution of $D_T$. Let us denote by $V_T$ the $\Q$-vector space underlying $D_T$, together with its natural structure of left $D_T$-module, so that $\mathrm{End}_{D_T}(V_T)=D_T^{op}$. Consider the skew-Hermitian pairing
\begin{align*}
\psi_T: V_T \times V_T & \rightarrow E\\
(v, w) & \mapsto \mathrm{Tr}_{D_T/E}(v\sigma w^*);
\end{align*}
for evey $g \in D_T$ and $v, w \in V_T$ we have
\begin{equation*}
\psi_T(vg, wg)=\mathrm{Tr}_{D_T/E}(vg \sigma (wg)^*)=\mathrm{Tr}_{D_T/E}(vg \bar{g}\sigma w^*)=g\bar{g}\psi_T(v, w)
\end{equation*}
hence, for every $\Q$-algebra $R$,
\begin{equation*}
H_T(R)=\{(g, c(g)) \in \mathrm{End}_{D_T \otimes_\Q R}(V_T \otimes_\Q R)\times (F \otimes_\Q R)^\times \mid \psi_T(vg, wg)=c(g)\psi(v, w)\}.
\end{equation*}

\subsection{Relation between quaternionic and unitary Shimura varieties}\label{quatvsun}
\subsubsection{} Fix a subset $T\subset \Sigma_p$; the reduced norm on $B_T$ gives rise to a map $\mathrm{Nm}: G_T \rightarrow T_F$. Let $T_{H_T}:=(T_F \times T_E)/T_F$, where the embedding of $T_F$ in $T_F \times T_E$ is given by $a \mapsto (a^2, a^{-1})$. Letting $N_{H_T}: H_T \rightarrow T_{H_T}$ be the map induced by $\mathrm{Nm} \times \mathrm{Id}: G_T \times T_E \rightarrow T_F \times T_E$, we obtain a commutative diagram 
\begin{center}
\begin{tikzcd}
G_T \times T_E \arrow[r] \arrow[d, "\mathrm{Nm}\times \mathrm{Id}"] & H_T \arrow[d, "N_{H_T}"] \\
T_F \times T_E \arrow[r] & T_{H_T}
\end{tikzcd}
\end{center}
where the top arrow is compatible with the Deligne homomorphisms. Recall that we have fixed a compact open subgroup $K_T \subset G_T(\A_f)$. Take a compact open subgroup $K_E=K_{E, p} K_E^{p} \subset T_E(\A_f)$ and a compact open subgroup $U_T \subset H_T(\A_f)$ containing the image of $K_T \times K_E$.
Let
\begin{align*}
C_{K_T}&:=F^{\times, +}\backslash \A_{F, f}^\times/\mathrm{Nm}(K_T), \; \; C_{K_E}:=E^{\times}\backslash \A_{E, f}^\times/K_E\\
C_{K_T \times K_E}&:=C_{K_T}\times C_{K_E}, \; \; C_{U_T}:=T_{H_T}(\Q)^+\backslash T_{H_T}(\A_f)/N_{H_T}(U_T),
\end{align*}
where $T_{H_T}(\Q)^+:=(F^{\times, +}\times E^\times)/F^\times$. We have maps $C_{K_T}\rightarrow C_{K_T}\times C_{K_E} \rightarrow C_{U_T}$, where the first map sends $C_{K_T}$ to (the equivalence class of) 1 on the second component. Letting $G_T(\Q)^+\subset G_T(\Q)$ be the subgroup of elements with totally positive norm, and $H_T(\Q)^+:=(G_T(\Q)^+\times T_E(\Q))/T_F(\Q)$, we obtain the following commutative diagram: 
\begin{equation}\label{quat-to-un}
\begin{tikzcd}[sep=small]
G_T(\Q)^+\backslash X_{T}^+\times G_T(\A_f)/K_T \arrow[r]\arrow[d] & G_T(\Q)^+\backslash X_{T}^+\times G_T(\A_f)/K_T \times C_{K_E} \arrow[r]\arrow[d] & H_T(\Q)^+\backslash X_{T}^+ \times H_T(\A_f)/U_T \arrow[d]\\
C_{K_T} \arrow[r]& C_{K_T}\times C_{K_E} \arrow[r]& C_{U_T}
\end{tikzcd}.
\end{equation}
If $T \neq \Sigma_p$ (so that the spaces in the top row are positive dimensional), the fibres of the vertical maps are connected components of the complex analytic spaces in the first row. The upper left (resp. upper right) space is identified with the complex points of $Sh_{K_T}(G_T)$ (resp. $Sh_{U_T}(H_T)$). If $T=\Sigma_p$ then we take this as a definition of (the complex points of) $Sh_{K_T}(G_T)$ and $Sh_{U_T}(H_T)$.

Following \cite{dks20}, we will study the relation between the top left and the top right space of the above diagram. To do so, we introduce the following notion, borrowed from \cite[\S~2.3.1]{dks20}.
\begin{defn}\label{defn:suff small}
We say that $K_E$ is \emph{sufficiently small with respect to} $K_T$ if the following conditions are satisfied:
\begin{enumerate}
\item $E^\times \cap\{\frac{y}{y^c} \mid y \in K_E\}=\{1\}$;
\item $K_E\cap T_F(\A_f)\subset K_T$;
\item $N_{E/F}(K_E)\subset \mathrm{Nm}(K_T)$.
\end{enumerate}
\end{defn}

\subsubsection{} Given $K_T$, it is always possible to choose $K_E$ sufficiently small with respect to $K_T$, cf. \cite[\S~2.3.1]{dks20}. The main point is that the norm map $\O_E^\times \rightarrow \O_F^\times$ has finite kernel, hence we can choose $K_E$ satisfying the first condition. Notice that this can be done independently of $K_T$; given $K_T$, we can then shrink $K_E$ so that it satisfies the second and third conditions. Furthermowe we may choose $K_E$ of the form  $(\mathcal{O}_E\otimes \Z_p)^\times K_E^{p}$ (if $K_T$ is as in \S~\ref{quat-datum}).

\subsubsection{Hecke algebras}
Let $u: G_T \rightarrow H_T$ be the composite of the inclusion $G_T \rightarrow G_T \times T_E$ (whose second component is the composition of the structure map and the identity section) and the projection $G_T \times T_E \rightarrow H_T$. If $U_T \subset H_T(\A_f)$ contains the image of $K_T \times K_E$ then $u$ induces a map $K_T\backslash G_T(\A_f)/K_T \rightarrow U_T\backslash H_T(\A_f)/U_T$, which in turn induces a (set theoretic) pullback map $r: \mathbb{T}_{U_T}(H_T)\rightarrow \mathbb{T}_{K_T}(G_T)$. On the other hand the inclusion $G_T \rightarrow G_T \times T_E$ gives rise to a morphism of Hecke algebras $i: \mathbb{T}_{K_T}(G_T) \rightarrow \mathbb{T}_{U_T}(H_T)$ sending the characteristic function of a double coset $K_TgK_T$ to that of $U_T(g, 1)U_T$.

\begin{lem}
Assume that $K_E$ is sufficiently small with respect to $K_T$ and that $U_T$ is the image of $K_T \times K_E$. Then the composite $r \circ i: \mathbb{T}_{K_T}(G_T)\rightarrow \mathbb{T}_{K_T}(G_T)$ is the identity.
\end{lem}
\begin{proof}
Let us see $G_T$ as a subgroup of $G_T \times T_E$ and denote by $q:(G_T \times T_E)(\mathbb{A}_f)\rightarrow H_T(\mathbb{A}_f)$ the quotient map. Given $g \in G_T(\mathbb{A}_f)$, we need to prove that
\begin{equation*}
q^{-1}(U_T(g, 1)U_T)\cap G_T(\mathbb{A}_f)=K_TgK_T.
\end{equation*}
Take $x$ belonging the the set on the left hand side above. Then there exist $(k_1, e_1), (k_2, e_2) \in K_T \times K_E$ and $a \in T_F(\mathbb{A}_f)$ such that
\begin{equation*}
x=(a, a^{-1})(k_1, e_1)(g, 1)(k_2, e_2).
\end{equation*}
As $x \in G_T(\mathbb{A}_f)$ we must have $a^{-1}e_1e_2=1$ hence $e_1e_2=a \in T_F(\mathbb{A}_f)\cap K_E\subset K_T$. It follows that
\begin{equation*}
x=(ak_1gk_2, 1) \in K_TgK_T.
\end{equation*}
\end{proof}

\subsubsection{}\label{def-hecke} The above lemma implies in particular that the map $i: \mathbb{T}_{K_T}(G_T) \rightarrow \mathbb{T}_{U_T}(H_T)$ is injective. In what follows we will use this map to identify $\mathbb{T}_{K_T}(G_T)$ with a sub-algebra of $\mathbb{T}_{U_T}(H_T)$. As in $\S~$\ref{hecke-alg} we will work with the Hecke algebra $\mathbb{T}=\otimes_{v}'\Z[\GL_2(\O_v)\backslash \GL_2(F_v)/\GL_2(\O_v)]$, where the product runs over the set of places of $F$ lying above a rational prime different from $p$ and such that the component of $K_T$ at $v$ is hyperspecial.

The statements in the next lemma are established, in a slightly different setting, in the proof of \cite[Lemma 2.3.1]{dks20}.

\begin{lem}
Assume that $K_E$ is sufficiently small with respect to $K_T$ and that $U_T$ is the image of $K_T \times K_E$. Then:
\begin{enumerate}
\item The map $C_{K_T}\rightarrow C_{U_T}$ obtained composing the bottom arrows in \eqref{quat-to-un} is injective.
\item The map $j: Sh_{K_T}(G_T)(\C)\rightarrow Sh_{U_T}(H_T)(\C)$ obtained composing the top arrows in \eqref{quat-to-un} restricts to an isomorphism between each connected component of $Sh_{K_T}(G_T)(\C)$ and a connected component of the target.
\end{enumerate}
\end{lem}
\begin{proof}
For the reader's convenience, let us briefly explain how the main points of the argument in the proof of \cite[Lemma 2.3.1]{dks20} adapt to our situation.
\begin{enumerate}
\item This follows from the fact that the inclusion $T_F\rightarrow T_{H_T}$ sending $a$ to $(a, 1)$ is split by the map $(a, b) \mapsto aN_{E/F}(b)$, which induces splittings
\begin{align*}
T_{H_T}(\Q)^+=(F^{\times, +}\times E^\times)/F^\times & \rightarrow F^{\times, +}, \;  & \mathrm{N}_{H_T}(U_T)& \rightarrow \mathrm{Nm}(K_T)\\
(a, b) & \mapsto aN_{E/F}(b) & (\mathrm{Nm}(k), l) & \mapsto \mathrm{Nm}(k)N_{E/F}(l)
\end{align*}
(notice that the last formula gives a well defined map $\mathrm{N}_{H_T}(U_T) \rightarrow \mathrm{Nm}(K_T)$ because of the assumptions that $U_T$ is the image of $K_T \times K_E$ and $N_{E/F}(K_E)\subset \mathrm{Nm}(K_T)$) of the inclusions
\begin{align*}
F^{\times, +} & \rightarrow T_{H_T}(\Q)^+=(F^{\times, +}\times E^\times)/F^\times, \; & \mathrm{Nm}(K_T)& \rightarrow \mathrm{N}_{H_T}(U_T)\\
a & \mapsto (a, 1) & \mathrm{Nm}(k) & \mapsto (\mathrm{Nm}(k), 1).
\end{align*}
\item If $T=\Sigma_p$ there is nothing to prove. Assume that $T \neq \Sigma_p$. Then the claim follows from the equality, valid for every $g \in G_T(\A_f)$:
\begin{equation*}
G_T(\Q)^+\cap g K_T g^{-1}=H_T(\Q)^+\cap g U_T g^{-1}.
\end{equation*}
Clearly the group on the left hand side is contained in the one on the right hand side. Conversely, let $h \in H_T(\Q)^+\cap g U_T g^{-1}$. On the one hand we can write (the equivalence class of) $h$ as $h=a \cdot e$ for some $a \in G_T(\Q)^+, e \in E^\times$; on the other hand $h=g(k \cdot y)g^{-1}$ for some $k \in K_T, y \in K_E$. It follows that $ae=g(ky)g^{-1}$, which implies that $e\cdot c(e)^{-1}=y\cdot c(y)^{-1} \in E^\times \cap \{\frac{y}{y^c} \mid y \in K_E\}=\{1\}$. Therefore $y \in K_E \cap T_F(\A_f)\subset K_T$. This yields $h \in G_T(\Q)^+\cap g K_T g^{-1}$.
\end{enumerate}
\end{proof}

\subsubsection{} 
Keeping the notations and assumptions of the previous lemma, we set $I:=C_{U_T}/C_{K_T}$. For each $\alpha \in I$, let $Sh_{U_T}^{\alpha}(H_T)(\C) \subset Sh_{U_T}(H_T)(\C)$ be the subspace consisting of connected components mapping to the $C_{K_T}$-coset inside $C_{U_T}$ given by $\alpha$. We obtain a decomposition
\begin{equation*}
Sh_{U_T}(H_T)(\C)=\coprod_{\alpha \in I} Sh_{U_T}^{\alpha}(H_T)(\C)
\end{equation*}
into open and closed subspaces; the subspace corresponding to the coset containing the identity is identified with $Sh_{K_{T}}(G_T)(\C)$ (if $\Sigma_p=T$, cf. the argument in the second part of the proof of \cite[Lemma 2.3.1]{dks20}).

\begin{cor}\label{Hilbtounitary}
Assume that $K_E$ is sufficiently small with respect to $K_T$ and that $U_T$ is the image of $K_T \times K_E$. Let $\m\subset \mathbb{T}$ be a maximal ideal and $i \geq 0$ an integer. Then:
\begin{enumerate}
\item $H^i(Sh_{K_T}(G_{T}), \F_{\ell})_{\m}\neq 0$ if and only if $H^i(Sh_{U_T}(H_T), \F_{\ell})_{\m}\neq 0$, and the same assertion is true for the cohomology with $\Q_{\ell}$-coefficients.
\item The natural map $H^i_c(Sh_{U_T}(H_T), \F_{\ell})_\m\rightarrow H^i(Sh_{U_T}(H_T), \F_{\ell})_\m$ is an isomorphism if and only if the same is true for the map $H^i_c(Sh_{K_T}(G_{T}), \F_{\ell})_{\m}\rightarrow H^i(Sh_{K_T}(G_{T}), \F_{\ell})_{\m}$.
\end{enumerate}
\end{cor}
\begin{proof}
We will work over the complex numbers throughout this proof, and omit this from the notation for simplicity. The decomposition $Sh_{U_T}(H_T)=\coprod_{\alpha \in I} Sh_{U_T}^{\alpha}(H_T)$ introduced above yields a direct sum decomposition
\begin{equation}\label{cohom-decompos}
H^i(Sh_{U_T}(H_{T}), \F_{\ell})=\oplus_{\alpha \in I} H^i(Sh_{U_T}^{\alpha}(H_T), \F_{\ell})
\end{equation}
and each summand on the right hand side is preserved by the action of $\mathbb{T}$. Indeed, the Hecke algebra $\mathbb{T}$ is spanned by characteristic functions of double cosets $K_TgK_T$, whose component at a place $v$ is of the form $\GL_2(\mathcal{O}_v)\left(\begin{smallmatrix}
\varpi_v^{a_v} & 0\\
0 & \varpi_v^{b_v}
\end{smallmatrix}\right)\GL_2(\mathcal{O}_v)$, for some uniformiser $\varpi_v$ of $\mathcal{O}_v$ and $a_v, b_v \in \Z$. The group $K_E$ is sufficiently small with respect to $K_{T, g}:=K_T \cap gK_Tg^{-1}$; letting $U_{T, g}\subset H_T(\mathbb{A}_f)$ be the image of $K_{T, g} \times K_E$ (which coincides with the group $ U_{T} \cap gU_Tg^{-1}$), the Hecke correspondence

\begin{center}
\begin{tikzcd}
& Sh_{U_{T, g}}(H_T) \arrow[dr] \arrow[dl] & \\
Sh_{U_T}(H_T) & & Sh_{U_T}(H_T)
\end{tikzcd}
\end{center}
restricts to a correspondence on each $Sh_{U_T}^{\alpha}(H_T)$. Indeed the map $Sh_{U_{T, g}}(H_T)\rightarrow Sh_{U_{T, g^{-1}}}(H_T)$ induced by right multiplication by $(g, 1) \in H_T(\A_f)$ preserves each $Sh_{U_{T, g}}^\alpha(H_T)$, as $N_{H_T}((g, 1)) \in \A_{F, f}^\times\subset T_{H_T}(\A_f)$. Hence each summand in \eqref{cohom-decompos} is preserved by the action of $\mathbb{T}$.

On the other hand we have an action of $T_E(\mathbb{A}_f)$ on $Sh_{U_T}(H_{T})$: the action of $e \in T_E(\mathbb{A}_f)$ sends each $Sh_{U_T}^{\alpha}(H_T)$ to $Sh_{U_T}^{e\alpha}(H_T)$. In particular $T_E(\mathbb{A}_f)$ acts transitively on the set of subvarieties $Sh_{U_T}^{\alpha}(H_T)$, inducing isomorphisms, for every $\alpha \in I$:
\begin{equation*}
H^i(Sh_{U_T}^{\alpha}(H_T), \F_{\ell})\simeq H^i(Sh_{U_T}^{1}(H_T), \F_{\ell})=H^i(Sh_{K_{T}}(G_T), \F_{\ell}).
\end{equation*}
Finally, the action on cohomology induced by the $T_E(\mathbb{A}_f)$-action on $Sh_{U_T}(H_{T})$ commutes with the action of $\mathbb{T}$. It follows that the above isomorphism induces an isomorphism
\begin{equation}\label{cohom-decompos-bis}
H^i(Sh_{U_T}(H_T), \F_{\ell})_\mathfrak{m}\simeq \oplus_{\alpha \in I} H^i(Sh_{K_{T}}(G_T), \F_{\ell})_\mathfrak{m}.
\end{equation}
The isomorphism \eqref{cohom-decompos-bis} (which also holds with $\Q_\ell$-coefficients) implies $(1)$. Furthermore the previous argument also applies to compactly supported cohomology, yielding a direct sum decomposition of $H^i_c(Sh_{U_T}(H_T), \F_{\ell})_\mathfrak{m}$ as in \eqref{cohom-decompos-bis}. Hence we obtain a similar decomposition for the kernel and cokernel of the map $H^i_c(Sh_{U_T}(H_T), \F_{\ell})_\m\rightarrow H^i(Sh_{U_T}(H_T), \F_{\ell})_\m$, from which $(2)$ follows.
\end{proof}

\subsection{Integral models of unitary Shimura varieties}\label{mod-unit}

\subsubsection{} Fix a totally negative element $\mathfrak{d} \in \mathcal{O}_F$ coprime to $p$; choose isomorphisms $\theta_T: M_2(E)=D_{\emptyset} \toisom D_{T}$ for every non-empty subset $T \subset \Sigma_p$. For every such $T$ choose an element $\delta_T \in D_T$ as in \cite[Lemma 3.8]{tixi16}, and let $\sigma_T=\sqrt{\mathfrak{d}}\delta_T$. Via the construction in \S~\ref{unit-datum}, such a choice gives an involution $*_T$ on each $D_T$. By \cite[Lemma 5.4]{tixi16} we may, and will, choose the elements $\delta_T$ in such a way that these involutions are respected by the isomorphisms $\theta_T$.
Let $\mathcal{O}_{D_\emptyset}=M_2(\mathcal{O}_E)\subset D_{\emptyset}$ and $\mathcal{O}_{D_T}=\theta_T(\mathcal{O}_{D_\emptyset})$ for every non-empty subset $T \subset \Sigma_p$.

Take $K_T\subset G_T(\mathbb{A}_f)$ as in \S~\ref{quat-datum}, $K_E=(\mathcal{O}_E\otimes \Z_p)^\times K_E^{p} \subset T_E(\A_f)$ sufficiently small with respect to $K_T$ and let $U_T\subset H_T(\mathbb{A}_f)$ be the image of $K_T \times K_E$. The inverse of the chosen isomorphism  $\iota_p: \bar{\Q}_p \toisom \C$ determines a distinguished $p$-adic place $\wp$ of the reflex field $E_T \subset \C$. Let $E_\wp \subset \bar{\Q}_p$ be the completion of $E_T$ at $\wp$, and $\O_\wp$ its ring of integers. Following \cite[Section 2]{dks20} and \cite[Section 3]{tixi16}, an integral model of $Sh_{U_T}(H_T)$ over $\O_\wp$ can be constructed as follows: consider the functor sending an $\O_\wp$-scheme $S$ to the set of isomorphism classes of tuples $(A, \iota, \lambda, \eta)$ where:

\begin{enumerate}
\item $A/S$ is an abelian scheme of dimension $4g$.
\item $\iota: \mathcal{O}_{D_T}\rightarrow \mathrm{End}_S(A)$ is an embedding.
\item $\lambda: A \rightarrow A^\vee$ is a $\Z_{(p)}^\times$-polarisation whose attached Rosati involution coincides with $*_T$ on $\mathcal{O}_{D_T}$.
\item $\eta$ is a $U_T$-level structure, in the sense of \cite[Section 2.2.2]{dks20}. 
\end{enumerate}

Furthermore the above data are required to satisfy the following conditions: 

\begin{enumerate}[label=(\alph*)]
\item For every $b \in \mathcal{O}_E$, the characteristic polynomial of $\iota(b)$ acting on $\mathrm{Lie}(A/S)$ equals
\begin{equation*}
\prod_{\tilde{\tau} \in \Sigma_{E, \infty}}(X-\tilde{\tau}(b))^{2s_{\tilde{\tau}}},
\end{equation*}
where the integers $s_{\tilde{\tau}}$ were defined in \S~\ref{auxil-CM}.
\item $\ker(\lambda[p^\infty]): A[p^\infty]\rightarrow A^\vee[p^\infty]$ is a finite flat subgroup scheme contained in $\prod_{\p \in T}A[\p]$ and such that for each $\p \in T$ the rank of $\ker(\lambda[p^\infty])\cap A[\p]$ is $p^4$.
\item The cokernel of $\lambda_*: H_1^{dR}(A/S)\rightarrow H_1^{dR}(A^\vee/S)$ is locally free of rank two over $\oplus_{\p \in T} \mathcal{O}_S\otimes_{\Z_p} (\mathcal{O}_E\otimes_{\mathcal{O}_F}\mathcal{O}_F/\p)$ (here $H_1^{dR}(A/S)$ denotes the de Rham homology sheaf of $A/S$).
\end{enumerate}

This functor is represented by a scheme $Y_{U_T}(H_T)$ which is an infinite disjoint union of smooth, quasi-projective (resp. projective if $T$ is non-empty) $\O_\wp$-schemes. The group $\mathcal{O}_{F, (p)}^{\times, +}$ acts on this scheme as follows: an element $u \in \mathcal{O}_{F, (p)}^{\times, +}$ sends $(A, \iota, \lambda, \eta)$ to $(A, \iota, u\lambda, \eta)$. This action factors through the group $\mathcal{O}_{F, (p)}^{\times, +}/N_{E/F}(U_T \cap E^\times)$, and the resulting quotient is a smooth quasi-projective scheme giving the desired integral model of $Sh_{U_T}(H_T)$.

\begin{rem}\label{rem-models}\leavevmode
\begin{enumerate}
\item More precisely, the above description of the moduli problem, using the level structure as in \cite[Section 2.2.2]{dks20}, is valid for locally noetherian schemes $S$. We will need later on to work with non locally noetherian schemes as well; in this case one can use a different definition of the level structure, adapting \cite[Definition 1.3.7.6]{la18}.
\item For any field extension $L$ of $E_T$, one can define schemes $Y_{U_T}(H_T)_L$ over $\mathrm{Spec} \; L$ for arbitrary $U_T=U_T^pU_{T, p}$ representing a moduli problem as above, including $U_{T, p}$-level structure.
\end{enumerate}
\end{rem}

\begin{rem}\label{cond-pdiv}
Notice that, if $S$ is the spectrum of a ring where $p$ is nilpotent, then conditions $(a), (b), (c)$ in the definition of the above moduli problem can be stated in terms of the $p$-divisible group (with $\mathcal{O}_E$-action) of the abelian scheme $A$. Indeed, by \cite[p. 164]{mes72} the Lie algebra of $A/S$ (resp. of the universal vector extension of $A/S$) is isomorphic to the Lie algebra of $A[p^\infty]$ (resp. of the universal vector extension of $A[p^\infty]$). Hence the first condition is the Kottwitz condition on $A[p^\infty]$; the second condition manifestly only depends on $A[p^\infty]$, and the same is true for the third, as the first de Rham cohomology of $A/S$ is identified with the Lie algebra of the universal vector extension of the dual of $A$ by \cite[Chapter 1, \S~4]{mame06}.
\end{rem}

\section{Comparison of unitary Igusa varieties}

In this section, we establish an isomorphism between Igusa varieties that are a priori 
attached to different unitary Shimura varieties. This relies on a reinterpretation 
of (some of) the arguments in~\cite{tixi16} in the case of a totally split prime. 

\subsection{Kottwitz sets}\label{kotsets} We keep the notation introduced in the previous section. In particular we fixed in \S~\ref{quat-datum} a prime $p$ which splits completely in $F$; hence we can write $D_T\otimes_{\Q}\Q_p=\prod_{\p \in \Sigma_p}D_{T, \p}$ and $H_{T, \Q_p}:=H_T \times_\Q \Q_p=\prod_{\p \in \Sigma_p}H_{T, \p}$, where $D_{T, \p}:=D_T\otimes_F F_\p$ and, for every $F_\p-$algebra $R$,

\begin{equation*}
H_{T, \p}(R)=\{g \in D_{T, \p}\otimes_{F_\p}R \mid g \bar{g} \in R^\times\}.
\end{equation*}

\noindent According to \cite[Lemma 3.8]{tixi16}, the subgroup defined by
\[
H_{T, \p}^1(R):=\{g\in H_{T, \p}(R)\mid g \bar{g}=1\}
\]
is an unramified (resp. non quasi-split) group over $F_\p$ if $\p \in \Sigma_p \smallsetminus T$ (resp. $\p \in T$). Furthermore the natural maps $G_T \leftarrow G_T \times T_E \rightarrow H_T$ are compatible with the Deligne homomorphisms and induce isomorphisms of derived and adjoint groups. Denoting by $\mu_T$ the cocharacter of $H_{T, \bar{\Q}_p}$ induced by $(H_T, X_T)$ and by the isomorphism $\iota_p$, we have the Kottwitz set $B(H_{T, \Q_p}, \mu_T)$. We can write $B(H_{T, \Q_p}, \mu_T)=\prod_{\p \in \Sigma_p} B(H_{T, \p}, \mu_{T, \p})$, where $\mu_{T, \p}$ is the $\p$-component of $\mu_T$, and it follows from \cite[(6.5.1), (6.5.2)]{kot97} that $B(H_{T, \p}, \mu_{T, \p})$ can be described as follows:
\begin{enumerate}
\item it contains two elements - the basic one and the $\mu$-ordinary one - if $\p \in \Sigma_p \smallsetminus T$;
\item it is a singleton if $\p \in T$. 
\end{enumerate}

Let $b=(b_\p)_{\p \in \Sigma_p} \in B(H_{T, \Q_p}, \mu_T)$ and let $B \subset \Sigma_p \smallsetminus T$ be the set of places such that $b_\p$ is basic. Let $T':=T \coprod B$ and $b' \in B(H_{T', \Q_p}, \mu_{T'})$ be the element which is $\mu$-ordinary at every place $\p \in \Sigma_p \smallsetminus T'$. We call $b'$ the element in $B(H_{T', \Q_p}, \mu_{T'})$ \emph{associated with} $b$.

\subsection{Igusa varieties}\label{igvar}
Every $b \in B(H_{T, \Q_p}, \mu_T)$ corresponds to a Newton stratum in the special fibre of $Y_{U_T}(H_T)$, hence to a Newton stratum in the special fibre of $Sh_{U_T}(H_T)$. Let $(\mathbb{X}^b_T, \iota^b_T, \lambda^b_T)$ be a $p$-divisible group with extra structure attached to an $\bar{\mathbb{F}}_p$-point in such a stratum. Notice that, as $p$ splits completely in $F$, this datum is equivalent to the datum of a collection of $p$-divisible groups with extra structure $(\mathbb{X}^{b_\p}_T, \iota^{b_\p}_T, \lambda^{b_\p}_T)$, for each $\p \in \Sigma_p$. Similarly, if $(\G, \iota_\G, \lambda_\G)$ is a $p$-divisible group with extra structure then an isomorphism (or, more generally, a quasi-isogeny) $\phi: \mathbb{X}^b_T \rightarrow \mathcal{G}$ commuting with the $\O_{D_T}$-action splits as a sum of isomorphisms (or quasi-isogenies) $\phi_\p$ indexed by places $\p \in \Sigma_p$. We will write $\phi=(\phi_\p)_{\p \in \Sigma_p}$.

\subsubsection{}\label{igvar-def} Let $\overline{\mathrm{Ig}}_T^b$ be the scheme representing the functor that sends a ring $R$ of characteristic $p$ to the set of isomorphism classes of data $(r, A, \iota, \lambda, \eta, \phi)$ defined as follows:
\begin{enumerate}
\item $r: \bar{\mathbb{F}}_p\rightarrow R$ is a ring morphism.
\item $(A, \iota, \lambda, \eta)/\mathrm{Spec} \; R$ is a datum as in the definition of the integral model of $Sh_{U_T}(H_T)$, satisfying conditions $(a), (b), (c)$ in \S~\ref{mod-unit}.
\item $\phi=(\phi_\p)_{\p \in \Sigma_p}: A[p^\infty] \toisom \mathbb{X}^b_T\times_{\bar{\mathbb{F}}_p, r}R$ is an isomorphism commuting with the $\mathcal{O}_{D_T}$-action and such that each $\phi_\p$ respects the polarisation up to a $\Z_p^\times$-factor.
\end{enumerate}
Two tuples $(r, A, \iota, \lambda, \eta, \phi), (r', A', \iota', \lambda', \eta', \phi')$ are said to be isomorphic if $r=r'$ and there is an isomorphism of abelian schemes over $\mathrm{Spec}\; R$ between $A$ and $A'$ commuting with all the additional data.

Forgetting everything but the structure map $r$ we see that $\overline{\mathrm{Ig}}_T^b$ is fibred over $\bar{\mathbb{F}}_p$; in what follows, if there is no danger of confusion, we will abusively denote its points with values in an $\bar{\mathbb{F}}_p$-algebra $(R, r)$ just by $(A, \iota, \lambda, \eta, \phi) \in \overline{\mathrm{Ig}}_T^b(R)$. The functor defining $\overline{\mathrm{Ig}}_T^b$ is representable: indeed, $\overline{\mathrm{Ig}}_T^b\rightarrow Y_{U_T}(H_T)_{\bar{\mathbb{F}}_p}$ relatively represents the moduli problem parametrising isomorphisms $\phi: A[p^\infty]\toisom \mathbb{X}^b_T\times_{\bar{\mathbb{F}}_p, r}R$ as above. By definition, such an isomorphisms is a compatible sequence of isomorphisms $\phi_k: A[p^k]\toisom \mathbb{X}^b_T[p^k]\times_{\bar{\mathbb{F}}_p, r}R$ for $k \geq 0$. Hence $\overline{\mathrm{Ig}}_T^b$ is the inverse limit of the schemes $\overline{\mathrm{Ig}}_{T, k}^b\rightarrow Y_{U_T}(H_T)_{\bar{\mathbb{F}}_p}$ parametrising isomorphisms $\phi_k$ as above; each of these moduli problem is relatively representable (cf. \cite[Proposition 4.3.3]{cs17}) and the transition maps are finite.

Recall that we have an action of $\O_{F, (p)}^{\times, +}$ on $Y_{U_T}(H_T)_{\bar{\mathbb{F}}_p}$. For each $k \geq 0$ let $\Delta_k:=\mathcal{O}_{F, (p)}^\times/\{N_{E/F}(u), u \in U_T \cap \mathcal{O}_E^\times, u \equiv 1 \pmod {p^k}\}$; set $\Delta:=\varprojlim_k \Delta_k$. The action of $\mathcal{O}_{F, (p)}^{\times, +}$ on $\overline{\mathrm{Ig}}_{T, k}^b$ lifting the action on $Y_{U_T}(H_T)_{\bar{\mathbb{F}}_p}$ factors through $\Delta_k$, hence we get an action of $\Delta$ on $\overline{\mathrm{Ig}}_T^b$. We define $\mathrm{Ig}_T^b:=\overline{\mathrm{Ig}}_T^b/\Delta$.

For each $\p \in \Sigma_p$ fix an element $x_\p \in \O_F^+$ with $\p$-adic valuation one, and with $\p'$-adic valuation zero for every $\p' \in \Sigma_p \smallsetminus \{\p\}$.
The following lemma gives a description of the functor of points of $\overline{\mathrm{Ig}}_T^b$ in terms of abelian schemes up to $p$-quasi-isogeny; see also \cite[Lemma 4.3.4]{cs17}.
\begin{lem}\label{ig-isog}
Let $R$ be a ring of characteristic $p$. There is a bijection, functorial in $R$, between $\overline{\mathrm{Ig}}_T^b(R)$ and the set of isomorphism classes of data $(r, A, \iota, \lambda, \eta, \rho)$ defined as follows:
\begin{enumerate}
\item $r: \bar{\mathbb{F}}_p\rightarrow R$ is a ring morphism.
\item $A$ is an abelian scheme of dimension $4g$ over $\mathrm{Spec} \; R$.
\item $\iota: \mathcal{O}_{D_T}\rightarrow \mathrm{End}_R(A)$ is an embedding.
\item $\lambda: A \rightarrow A^\vee$ is a $\Z_{(p)}^\times$-polarisation such that the attached Rosati involution coincides with $*_T$ on $\mathcal{O}_{D_T}$.
\item $\eta$ is a level structure, defined as in \S~\ref{mod-unit}.
\item $\rho=(\rho_\p)_{\p \in \Sigma_p}: A[p^\infty] \rightarrow \mathbb{X}^b_T\times_{\bar{\mathbb{F}}_p, r}R$ is a quasi-isogeny commuting with $\mathcal{O}_{D_T}$-action and such that each $\rho_\p$ respects the polarisation up to a $\Q_p^\times$-factor.
\end{enumerate}
Two tuples $(r, A, \iota, \lambda, \eta, \rho), (r', A', \iota', \lambda', \eta', \rho')$ are regarded as isomorphic if $r=r'$ and there is a $p$-quasi-isogeny from $A$ to $A'$ commuting with $\mathcal{O}_{D_T}$-action, level structure and the quasi-isogenies $\rho, \rho'$, and respecting polarisations up to a product of integral powers of the elements $x_\p$.
\end{lem}

\begin{rem}
Before proving the lemma, let us comment on its content. The first point is that the datum $(A, \iota, \lambda, \eta)$ is \emph{not} a priori required to satisfy the conditions $(a), (b), (c)$ in \S~\ref{mod-unit}; the second point is that the isomorphism $\phi$ in the definition of the moduli problem represented by $\overline{\mathrm{Ig}}_T^b$ is replaced by a quasi-isogeny $\rho$ (at the price of changing the notion of isomorphism of the data we are parametrising). Both points will be crucial in the next theorem.
\end{rem}

\begin{proof}
Let $R$ be an $\bar{\mathbb{F}}_p$-algebra. An isomorphism class of data $(A, \iota, \lambda, \eta, \phi)$ corresponding to an $R$-point of $\overline{\mathrm{Ig}}_T^b$ gives rise to a tuple as in the statement of the lemma. Conversely, let $(A, \iota, \lambda, \eta, \rho)$ be a datum as in the statement of the lemma; multiplying $\rho$ by suitable powers of the elements $x_\p$ we obtain an isogeny, abusively still denoted by the same symbol, $\rho: A[p^\infty]\rightarrow \mathbb{X}^b_T\times_{\bar{\mathbb{F}}_p}R$, commuting with $\mathcal{O}_{D_T}$-action and polarisation (the latter up to a $\Q_p^\times$-factor on each component). Letting $B=A/\ker(\rho)$, the map $\rho$ induces an isomorphism $\bar{\rho}: B[p^\infty]\xrightarrow{\sim} \mathbb{X}^b_T\times_{\bar{\mathbb{F}}_p}R$, and this property characterises uniquely $B$ among abelian schemes in the $p$-power isogeny class of $A$. In order to complete the proof we need to endow $B$ with extra structures $\iota_B, \lambda_B, \eta_B, \phi$ in such a way that $(B, \iota_B, \lambda_B, \eta_B, \phi)$ is a point of $\overline{\mathrm{Ig}}^b_T$, and the quotient map $q: A\rightarrow B$ respects the extra structures as in the statement of the lemma.

Let us start by defining the $\mathcal{O}_{D_T}$-action on $B$. Given $o \in \mathcal{O}_{D_T}$ we consider the self-quasi-isogeny $\iota_B(o):=q \circ \iota(o) \circ q^{-1}$ of $B$. Let us look at the diagram
\begin{center}
\begin{tikzcd}
A[p^\infty] \arrow[r, "q"]\arrow[d, "\iota(o)"] & B[p^\infty] \arrow[r, "\bar{\rho}"] \arrow[d, "\iota_B(o)"] & \mathbb{X}^b_T\times_{\bar{\mathbb{F}}_p}R \arrow[d, "\iota_T^b(o)"]\\
A[p^\infty] \arrow[r, "q"] & B[p^\infty] \arrow[r, "\bar{\rho}"] & \mathbb{X}^b_T\times_{\bar{\mathbb{F}}_p}R.
\end{tikzcd}
\end{center}
By assumption the largest square and the left square in the diagram commute; it follows that the right square also commutes. We deduce that $\iota_B(o): B \rightarrow B$ is a morphism (and not just a quasi-isogeny) as the same is true for $\iota_T^b(o): \mathbb{X}^b_T\times_{\bar{\mathbb{F}}_p}R \rightarrow \mathbb{X}^b_T\times_{\bar{\mathbb{F}}_p}R$. 
Hence the map sending $o$ to $\iota_B(o)$ endows $B$ with an $\mathcal{O}_{D_T}$-action, which commutes with the quotient map $A \rightarrow B$ by construction.

Let us now define the polarisation $\lambda_B$. We consider the quasi-isogeny $\lambda'_B:=(q^\vee)^{-1}\circ \lambda \circ q^{-1}: B \rightarrow B^\vee$. It fits in the diagram
\begin{center}
\begin{tikzcd}
A[p^\infty] \arrow[r, "q"]\arrow[d, "\lambda"] & B[p^\infty] \arrow[r, "\bar{\rho}"] \arrow[d, "\lambda'_B"] & \mathbb{X}^b_T\times_{\bar{\mathbb{F}}_p}R \arrow[d, "\lambda_{\mathbb{X}^b_T}"]\\
A^\vee[p^\infty] & B^\vee[p^\infty] \arrow[l, "q^\vee"] & \mathbb{X}^{b, \vee}_T\times_{\bar{\mathbb{F}}_p}R \arrow[l, "\bar{\rho}^\vee"].
\end{tikzcd}
\end{center}
By assumption the map $\rho=\bar{\rho}\circ q$ can be written as $\rho=(\rho_\p)_{\p \in \Sigma_p}$ where each $\rho_\p$ respects polarisations on the $\p$-component of the relevant $p$-divisible groups up to a factor $c_\p \in \Q_p^\times$, i. e. we have $(\rho_\p^\vee)^{-1} \circ \lambda_\p \circ (\rho_\p)^{-1}=c_\p\lambda_{\mathbb{X}^{b_\p}_T}$. For each $\p$ let $v_\p$ be the valuation of $c_\p$, and set $\lambda_B:=(\prod_{\p \in \Sigma_p} x_\p^{-v_\p})\lambda'_B$; then the isomorphism $\bar{\rho}$ commutes with the polarisations $\lambda_B$ and $\lambda_{\mathbb{X}^b_T}$ up to a $\Z_p^\times$-factor on each component. 

Let $\eta_B$ be the (prime to $p$) level structure on $B$ induced by $\eta$ and $q$. Let the isomorphism $\phi$ be given by $\bar{\rho}$. Notice that changing the polarisation on $A$ by a product of powers of the elements $x_\p$ does not affect the polarisation $\lambda_B$, hence the isomorphism class of $(B, \iota_B, \lambda_B, \eta_B, \phi)$ only depends on the isomorphism class of $(A, \iota, \lambda, \eta, \rho)$ (in the sense of the lemma).

Furthermore by construction the map $q$ commutes with polarisations up to a product of powers of the elements $x_\p$, and respects all the other additional structures; therefore the data $(A, \iota, \lambda, \eta, \rho)$ and $(B, \iota_B, \lambda_B, \eta_B, \phi)$ are isomorphic in the sense of the lemma.

Finally, the datum $(B, \iota_B, \lambda_B, \eta_B)$ satisfies conditions $(a), (b), (c)$ in \S~\ref{mod-unit}, hence is a point of $\overline{\mathrm{Ig}}_T^b$. This follows from Remark \ref{cond-pdiv} and from the fact that the $p$-divisible group $\mathbb{X}^b_T$ comes from an abelian variety with extra structure satisfying the same conditions.
\end{proof}

\begin{thm}\label{isoigusa}
Let $b \in B(H_{T, \Q_p}, \mu_T)$ and let $B\subset \Sigma_p \smallsetminus T$ be the subset of places $\p$ such that $b_\p$ is basic. Let $T'=T \coprod B$ and let $b' \in B(H_{T', \Q_p}, \mu_{T'})$ be the element associated with $b$. Then there is an isomorphism
\begin{equation*}
\mathrm{Ig}^b_T \simeq \mathrm{Ig}^{b'}_{T'}
\end{equation*}
whose induced map in cohomology is equivariant with respect to the action of the Hecke operators outside $p$.
\end{thm}
\begin{proof}
It suffices to produce an isomorphism $\overline{\mathrm{Ig}}^b_T \simeq \overline{\mathrm{Ig}}^{b'}_{T'}$ which satisfies the requirements in the theorem and is in addition $\O_{F, (p)}^{\times, +}$-equivariant. 
Recall that in \S~\ref{mod-unit} we have chosen isomorphisms
\begin{equation*}
(D_\emptyset, *_\emptyset, \mathcal{O}_{D_\emptyset}) \simeq (D_T, *_T, \mathcal{O}_{D_T}),
\end{equation*}
i.e. isomorphisms $\theta_T: D_\emptyset\toisom D_T$ respecting orders and involutions; we use these isomorphisms to identify the above data, and we denote them by $(D, *, \mathcal{O}_D)$ in this proof. With this notation,  Lemma \ref{ig-isog} tells us that, for every $T \subset \Sigma_p$ and every $\bar{\mathbb{F}}_p$-algebra $R$, the set $\overline{\mathrm{Ig}}_T^b(R)$ is the set of isomorphism classes of data $(A, \iota, \lambda, \eta, \rho)$ where
\begin{enumerate}
\item $A/\mathrm{Spec} \; R$ is an abelian scheme of dimension $4g$.
\item $\iota: \mathcal{O}_{D}\rightarrow \mathrm{End}_R(A)$ is an embedding.
\item $\lambda: A \rightarrow A^\vee$ is a $\Z_{(p)}^\times$ polarisation whose attached Rosati involution coincides with $*$ on $\mathcal{O}_{D}$.
\item $\eta$ is a level structure, defined as in \S~\ref{mod-unit}.
\item $\rho: A[p^\infty]\rightarrow \mathbb{X}^b_T\times_{\bar{\mathbb{F}}_p}R$ is a quasi-isogeny commuting with $\mathcal{O}_{D}$-action and polarisation (the latter up to a constant).
\end{enumerate}
In particular, the only datum in the description of the functor of points of $\overline{\mathrm{Ig}}_T^b$ which depends on $T$ is the quasi-isogeny $\rho: A[p^\infty]\rightarrow \mathbb{X}^b_T\times_{\bar{\mathbb{F}}_p}R$. Hence, in order to complete the proof it suffices to show that, if $b$ is associated with $b'$, then there is a quasi-isogeny $\rho_{T, T'}: \mathbb{X}^b_T\rightarrow \mathbb{X}^{b'}_{T'}$ commuting with the extra structure. The functor sending $(A, \iota, \lambda, \eta, \rho)$ to $(A, \iota, \lambda, \eta, \rho_{T, T'}\circ \rho)$ will then give the desired isomorphism $\overline{\mathrm{Ig}}^b_T \simeq \overline{\mathrm{Ig}}^{b'}_{T'}$; as $\O_{F, (p)}^{\times, +}$ only acts modifying polarisations by prime to $p$ quasi-isogenies, this isomorphism is $\O_{F, (p)}^{\times, +}$-equivariant.
The existence of a quasi-isogeny $\rho_{T, T'}$ as above follows from the construction in the proof of \cite[Lemma 5.18]{tixi16}. Indeed, letting $\D_T$ be the Dieudonné module of $\mathbb{X}^b_T$, the argument in \cite[proof of Lemma 5.18, p. 2185]{tixi16} produces a Dieudonné module $p\D_T \subset N \subset \D_T$, giving rise to a closed finite subgroup scheme $Z\subset \mathbb{X}^b_T[p]$; moreover $\mathbb{X}^b_T/Z$ is isomorphic to $\mathbb{X}^{b'}_{T'}$, and the construction in {\it loc. cit.} endows the quotient $\mathbb{X}^b_T/Z$ with extra structures respected by the isogeny $\mathbb{X}^b_T/Z\rightarrow \mathbb{X}^b_T$ whose composite with the quotient map is multiplication by $p$.
\end{proof}

\begin{rem}\leavevmode
\begin{enumerate}
\item The argument in \cite[p. 47]{tixi16} mentioned above constructs $N$ starting from the unitary Dieudonné module of $\mathbb{X}^b_T$ and using the Frobenius operator; a key point is that this results in a change of signature. The reader may find it helpful to check directly this phenomenon in the explicit examples given in \cite[(3.2)]{buwe06}.
\item With the notation as in the above theorem, the main result of \cite{tixi16} implies that the Newton strata indexed by $b$ and $b'$ in the special fibres of the Shimura varieties attached to $H_T$ and $H_{T'}$ are isomorphic. The above result shows that the same is true for the corresponding Igusa varieties. This is natural to expect, in view of Lemma~\ref{ig-isog} and of the fact that the isomorphism in~\cite{tixi16} is obtained from a quasi-isogeny on the level of $p$-divisible groups. The existence of such isomorphisms between Igusa varieties is established much more systematically in the function field setting in Sempliner's PhD thesis.
\end{enumerate}
\end{rem}

\section{The geometry of the Hodge--Tate period morphism}\label{sec:geometryHT}

The goal of this section is to establish the infinite-level Mantovan product formula for our
abelian-type unitary Shimura varieties. As a consequence, we can compute the fibers of
the Hodge--Tate period morphism in terms of Igusa varieties. 

\subsection{The product formula}\label{almprod} Take $p$ as in the previous section and choose a subset $T \subset \Sigma_p$; in order to simplify the notation, unless otherwise stated we will keep $T$ fixed throughout this section, and omit it from the notation. Fix an element $b=(b_\p)_{\p \in \Sigma_p} \in B(H_{\Q_p}, \mu)$, and let $(\mathbb{X}^b, \iota^b, \lambda^b)$ be a $p$-divisible group with extra structure attached to $b$, giving rise to a collection of $p$-divisible groups with extra structure $(\mathbb{X}^{b_\p}, \iota^{b_\p}, \lambda^{b_\p})$, for $\p \in \Sigma_p$, as in the previous section.

\subsubsection{} Let $\breve{\O}_\wp$ be the ring of integers of the completion of the maximal unramified extension of $E_\wp$, and let $\mathrm{Nilp}_{\breve{\O}_\wp}$ be the category of $\breve{\O}_\wp$-algebras in which $p$ is nilpotent. Following \cite[Definition 4.3.11]{cs17} we consider the functor $\X^b: \mathrm{Nilp}_{\breve{\O}_\wp} \rightarrow \mathrm{Sets}$ sending $R$ to the set of isomorphism classes of data $(A, \iota, \lambda, \eta, \varphi)$, where
\begin{enumerate}
\item $(A, \iota, \lambda, \eta)/\mathrm{Spec} \; R$ is a datum as in \S~\ref{mod-unit}.
\item
\begin{equation*}
\varphi=\oplus_{\p \in \Sigma_p}\varphi_\p: A[p^\infty]\times_{R} R/p=\oplus_{\p \in \Sigma_p}A[\p^\infty]\times_{R} R/p\rightarrow \oplus_{\p \in \Sigma_p} \mathbb{X}^{b_\p}\times _{\bar{\mathbb{F}}_p}R/p
\end{equation*}
is a quasi-isogeny commuting with the $\O_D$-action, and such that each $\varphi_\p$ respects the polarisation up to a $\Q_p^\times$-factor.
\end{enumerate}

We will now decompose $\X^b$ as a product of Rapoport--Zink spaces and a formal lift of an Igusa variety we introduced before, as in \cite[Lemma 4.3.12]{cs17}. The existence of such lifts is a consequence of the next lemma, which is analogous to \cite[Corollary 4.3.5]{cs17}.

\begin{lem}
The scheme $\overline{\mathrm{Ig}}^b$ is perfect.
\end{lem}
\begin{proof}
We have to show that, for every ring $R$, the absolute Frobenius $\mathrm{Fr}: \overline{\mathrm{Ig}}^b \rightarrow \overline{\mathrm{Ig}}^b$ induces a bijection on $R$-points. We will use the description of the functor of points of $\overline{\mathrm{Ig}}^b$ given in Lemma \ref{ig-isog}. As the map $\mathrm{Fr}$ is not a morphism of $\bar{\mathbb{F}}_p$-schemes, in this proof we need to take the $\bar{\mathbb{F}}_p$-algebra structure on $R$ into account. Let $(r, A, \iota, \lambda, \eta, \rho)$ be an $R$-point of $\overline{\mathrm{Ig}}^b$; its image via Frobenius is the $R$-point $(r^{(p)}, A^{(p)}, \iota^{(p)}, \lambda^{(p)}, \eta^{(p)}, \rho^{(p)})$ where:
\begin{enumerate}
\item $r^{(p)}$ is the composite of $r$ and the Frobenius on $\bar{\mathbb{F}}_p$.
\item $A^{(p)}:=A \times_{R, \mathrm{Fr}_R}R$.
\item ($\iota^{(p)}, \lambda^{(p)}, \eta^{(p)}$) are induced from $(\iota, \lambda, \eta)$ by functoriality.
\item $\rho^{(p)}:=\rho \times \mathrm{Id}: A[p^\infty] \times_{R, \mathrm{Fr}_R}R \rightarrow (\mathbb{X}^b\times_{\overline{\mathbb{F}}_p, r} R) \times_{R, \mathrm{Fr}_R} R = \mathbb{X}^b\times_{\overline{\mathbb{F}}_p, r^{(p)}} R$.
\end{enumerate}
Let $F_A: A \rightarrow A^{(p)}$ be the relative Frobenius morphism. Then $F_A$ commutes with $\iota, \iota^{(p)}$ (as well as with the level structures) by functoriality. Furthermore we have
\begin{equation*}
(F_A)^\vee \circ \lambda^{(p)} \circ F_A= (F_A)^\vee \circ F_{A^\vee} \circ \lambda=p \lambda
\end{equation*}
where the first equality holds true by functoriality of relative Frobenius and the second follows from the fact that $(F_A)^\vee$ is the Verschiebung on $A^\vee$, and the composite of Verschiebung and Frobenius is multiplication by $p$. Finally, we have
\begin{equation*}
\rho^{(p)}\circ F_A=F_{\mathbb{X}^b}\circ \rho
\end{equation*}
where $F_{\mathbb{X}^b}: \mathbb{X}^b \times_{\overline{\mathbb{F}}_p, r} R \rightarrow \mathbb{X}^b \times_{\overline{\mathbb{F}}_p, r^{(p)}} R$ is the relative Frobenius. Writing $\prod_{\p \in \Sigma_p}x_\p=pt$ with $t \in \O_{F, (p)}^{\times, +}$, we deduce that the $R$-point $(r^{(p)}, A^{(p)}, \iota^{(p)}, \lambda^{(p)}, \eta^{(p)}, \rho^{(p)})$ coincides with the $R$-point
\begin{equation*}
\left(r^{(p)}, A, \iota, \frac{1}{t}\lambda, \eta, F_{\mathbb{X}^b}\circ \rho\right).
\end{equation*}
Let us denote by $r^{(-p)}$ the composition of $r$ and the inverse of Frobenius on $\overline{\mathbb{F}}_p$; we have the Verschiebung map $V_{\mathbb{X}^b}: \mathbb{X}^b \times_{\bar{\mathbb{F}}_p, r} R \rightarrow \mathbb{X}^b \times_{\bar{\mathbb{F}}_p, r^{(-p)}} R$, and the map $\overline{\mathrm{Ig}}^b \rightarrow \overline{\mathrm{Ig}}^b$ which on $R$-points is given by
\begin{equation*}
(r, A, \iota, \lambda, \eta, \rho) \mapsto \left(r^{(-p)}, A, \iota, t\lambda, \eta, \frac{1}{p}V_{\mathbb{X}^b} \circ \rho\right)
\end{equation*}
is the inverse of the map induced by $\mathrm{Fr}: \overline{\mathrm{Ig}}^b \rightarrow \overline{\mathrm{Ig}}^b$.
\end{proof}

\subsubsection{The formal Igusa variety}\label{form-igvar} Since the Igusa variety $\overline{\mathrm{Ig}}^b$ is perfect, it lifts canonically to a flat $p$-adic formal scheme $\overline{\mathfrak{Ig}}^b: \mathrm{Nilp}_{\breve{\O}_\wp} \rightarrow \mathrm{Sets}$ described as follows (see \cite[p. 719]{cs17}): fix a lift (up to quasi-isogeny) $\hat{\mathbb{X}}^{b}/\breve{\O}_\wp$ of $\mathbb{X}^b$ with extra structure. 
For $R \in \mathrm{Nilp}_{\breve{\O}_\wp}$, the $R$-points of $\overline{\mathfrak{Ig}}^b$ are isomorphism classes of data $(A, \iota, \lambda, \eta, \phi)$ as in \S~\ref{igvar-def}, except that the target of $\phi$ is $\hat{\mathbb{X}}^{b}_R:=\hat{\mathbb{X}}^{b}\times_{\breve{\O}_\wp}R$.

The group $\Delta$ defined in \S~\ref{igvar} acts on $\overline{\mathfrak{Ig}}^b$, and we set $\mathfrak{Ig}^b:=\overline{\mathfrak{Ig}}^b/\Delta$.

\subsubsection{Rapoport--Zink spaces}\label{razin-fin} For every $\p \in \Sigma_p$ we have the Rapoport--Zink space $\mathfrak{M}^{b_\p}$ which is the formal scheme $\mathrm{Nilp}_{\breve{\O}_\wp} \rightarrow \mathrm{Sets}$ sending $R$ to the set of isomorphism classes of data $(\G, \iota, \lambda, \rho)$, where:
\begin{enumerate}
\item $\G$ is a $p$-divisible group over $R$.
\item $\iota: \O_{D, \p}\rightarrow \mathrm{End}_R(\G)$ is an action of $\O_{D, \p}:=\O_D\otimes_{\O_F}\O_{F_\p}$.
\item $\lambda$ is a polarisation compatible with the involution on $\O_{D, \p}$ induced by $*$ and such that the Kottwitz condition is satisfied.
\item $\rho: \G \times_{R}R/p \rightarrow \mathbb{X}^{b_\p}\times_{\bar{\mathbb{F}}_p}R/p$ is a quasi-isogeny commuting with the $\O_{D, \p}$-action, and respecting polarisations up to a $\Q_p^\times$-factor.
\end{enumerate}
Furthermore we require $\lambda$ to be principal if $\p \not \in T$. If $\p \in T$, we ask the cokernel of the map induced by $\lambda$ on Lie algebras of the universal vector extensions of the relevant $p$-divisible groups to be locally free of rank two over $R \otimes_{\Z_p}\mathcal{O}_E/\p\mathcal{O}_E$. 

Two tuples $(\G, \iota, \lambda, \rho)$ and $(\G', \iota', \lambda', \rho')$ are regarded as isomorphic if there exists an isomorphism $\G \rightarrow \G'$ commuting with $\iota$, $\iota'$, $\rho$ and $\rho'$, and respecting polarisations up to a $\Z_p^\times$-factor.
Let $\mathfrak{M}^b:=\prod_{\p \in \Sigma_p} \mathfrak{M}^{b_\p}$.

\subsubsection{} We will now define a map $\alpha: \overline{\mathfrak{Ig}}^b \times_{\breve{\O}_\wp} \mathfrak{M}^b \rightarrow \X^b$. We will denote the object to which some extra structure is attached by a subscript.

Fix $R \in \mathrm{Nilp}_{\breve{\O}_\wp}$ and take
\begin{equation*}
(A, \iota_A, \lambda_A, \eta_A, \phi_A) \in \overline{\mathfrak{Ig}}^b(R) \text{ and } (\G_\p, \iota_{\G_\p}, \lambda_{\G_\p}, \rho_{\G_\p})_{\p \in \Sigma_p} \in \mathfrak{M}^b(R).
\end{equation*}
Let $\G=\oplus_{\p \in \Sigma_p}\G_\p$ and let $\rho_	\G: \G \rightarrow \hat{\mathbb{X}}^b_R$ be the quasi-isogeny lifting $\oplus_{\p \in \Sigma_p}\rho_{\G_\p}$; we get a quasi-isogeny $\rho_\G^{-1}\circ \phi_A: A[p^\infty]\rightarrow \G$ which commutes with the action of $\O_D$ and such that each component respects polarisations up to a $\Q_p^\times$-factor. We need to construct a point $(B, \iota_{B}, \lambda_{B}, \eta_{B}, \varphi_{B}) \in \X^b(R)$.

Let $B$ be the unique abelian scheme in the $p$-isogeny class of $A$ such that $\rho_\G^{-1}\circ \phi_A$ induces an isomorphism $\gamma=(\gamma_\p)_{\p \in \Sigma_p}: B[p^\infty] \toisom \G$. We endow $B$ with an action $\iota_B:\O_D \rightarrow \mathrm{End}_R(B)$ and with a level structure $\eta_{B}$ defined as in the proof of Lemma \ref{ig-isog}. To define the polarisation $\lambda_{B}$, look at the diagram

\begin{center}
\begin{tikzcd}
A[p^\infty] \arrow[r, "q"]\arrow[d, "\lambda_A"] & B[p^\infty] \arrow[r, "\gamma"] & \G \arrow[d, "\lambda_{\G}"]\\
A^\vee[p^\infty] & B^\vee[p^\infty] \arrow[l, "q^\vee"] & \G^\vee \arrow[l, "{\gamma}^\vee"].
\end{tikzcd}
\end{center}

\noindent As in the proof of Lemma \ref{ig-isog} we may rescale the quasi-isogeny $(q^\vee)^{-1}\circ \lambda_A\circ q^{-1}$ and get a polarisation $\lambda_B$ on $B$ such that $\gamma$ interchanges it with $\lambda_\G$ up to a $\Z_p^\times$-factor on each component. Finally, let $\bar{\gamma}$ be the reduction of $\gamma$ modulo $p$ and
\begin{equation*}
\varphi_B:=(\oplus_{\p \in \Sigma_p}\rho_{\G_\p})\circ \bar{\gamma}: B[p^\infty]\times_R R/p\rightarrow \mathbb{X}^b\times_{\bar{\mathbb{F}}_p}R/p.
\end{equation*}
Notice that by construction the polarisation $\lambda_B$ is unchanged if $\G$ is replaced by a $p$-divisible group $\G'$ endowed with an isomorphism to $\G$ respecting the extra structures, the polarisation being respected up to a $\Z_p^\times$-factor on each component. Therefore, the resulting map $\alpha: \overline{\mathfrak{Ig}}^b \times_{\breve{\O}_\wp} \mathfrak{M}^b \rightarrow \X^b$ is well-defined.

\begin{lem}\label{prod-fin}
The map $\alpha: \overline{\mathfrak{Ig}}^b \times_{\breve{\O}_\wp} \mathfrak{M}^b \rightarrow \X^b$ is an isomorphism.
\end{lem}
\begin{proof}
We will construct an inverse $\beta: \X^b \rightarrow \overline{\mathfrak{Ig}}^b \times_{\breve{\O}_\wp} \mathfrak{M}^b$ of the map $\alpha$. Take $(B, \iota_{B}, \lambda_{B}, \eta_{B}, \varphi_{B}) \in \X^b(R)$ and let $\G:=B[p^\infty]$, endowed with the additional structures coming from $B$, and with the trivialisation $\rho_\G:=\varphi_B$. This gives an $R$-point of  $\mathfrak{M}^b$. Finally, define $(A, \iota_A, \lambda_A, \eta_A, \phi_A)$ as follows: the abelian scheme $A$ is the unique one in the $p$-isogeny class of $B$ such that the lift of $\varphi_{B}$ induces an isomorphism $\phi_A: A[p^\infty] \rightarrow \hat{\mathbb{X}}^b_R$, and the additional structures are obtained from those on $B$ as in the discussion before the statement of the lemma.

With the notations introduced in the construction of the map $\alpha$, the composite $\beta \circ \alpha$ sends
\begin{equation*}
(\G, \iota_{\G}, \lambda_{\G}, \oplus_{\p \in \Sigma_p}\rho_{\G_\p}), (A, \iota_A, \lambda_A, \eta_A, \phi_A)
\end{equation*}
to
\begin{equation*}
(B[p^\infty], \iota_B[p^\infty], \lambda_B[p^\infty], \oplus_{\p \in \Sigma_p}\rho_{\G_\p} \circ \bar{\gamma}), (A, \iota_A, \lambda_A, \eta_A, \phi_A).
\end{equation*}
The map $\gamma$ gives an isomorphism, respecting polarisations up to $\Z_p^\times$ on each factor, between $(B[p^\infty], \iota_B[p^\infty], \lambda_B[p^\infty], \oplus_{\p \in \Sigma_p}\rho_{\G_\p} \circ \bar{\gamma})$ and $(\G, \iota_{\G}, \lambda_{\G}, \oplus_{\p \in \Sigma_p}\rho_{\G_\p})$, hence $\beta \circ \alpha$ is the identity map. Similarly one checks that $\alpha \circ \beta$ is the identity.
\end{proof}

\subsection{The product formula at infinite level and the Hodge--Tate period map} We now wish to establish an analogue in our situation of \cite[Lemma 4.3.20]{cs17}.
\subsubsection{Generic fibres of formal schemes}\label{genfib-for} We will need to work with the adic generic fibres of the formal schemes introduced in the previous section. Recall that to a formal scheme $\mathfrak{M}$ over $\breve{\O}_\wp$ (locally admitting a finitely generated ideal of definition) one can attach an adic space $\mathfrak{M}^{ad}$ as in \cite[Proposition 2.2.1]{sw13} over $\Spa(\breve{\O}_\wp, \breve{\O}_\wp)$; one can then take the adic generic fibre
\begin{equation*}
\mathcal{M}:=\mathfrak{M}^{ad}\times_{\Spa(\breve{\O}_\wp, \breve{\O}_\wp)}\Spa(\breve{E}_\wp, \breve{\O}_\wp)
\end{equation*}
where $\breve{E}_\wp$ is the fraction field of $\breve{\O}_\wp$. With this notation, the product formula in Lemma \ref{prod-fin} becomes, on the generic fibre:
\begin{equation}\label{prod-gen}
\overline{\mathcal{I}g}^b \times_{\Spa(\breve{E}_\wp, \breve{\O}_\wp)} \mathcal{M}^b \xrightarrow{\sim} \mathcal{X}^b.
\end{equation}
The functors of points on complete affinoid $(\breve{E}_\wp, \breve{\O}_\wp)$-algebras of the adic spaces in the previous formula can be described using \cite[Proposition 2.2.2]{sw13}. It is easier, and sufficient for our purposes, to give such a description restricted to the category $\mathrm{Perf}_{\breve{E}_\wp}$ of perfectoid Huber pairs $(R, R^+)$ over $(\breve{E}_\wp, \breve{\O}_\wp)$. For example, the functor of points of $\mathcal{X}^b$ is the sheafification in the analytic topology of the functor
\begin{align*}
\mathrm{Perf}_{\breve{E}_\wp} & \rightarrow \mathrm{Sets}\\
(R, R^+) & \mapsto \varprojlim_n \mathfrak{X}^b(R^+/p^n).
\end{align*}
In other words, up to sheafification, an $(R, R^+)$-point of $\mathcal{X}^b$ is a compatible collection of data $(A_n, \iota_n, \lambda_n, \eta_n, \varphi_n)$ over $R^+/p^n$ as defined in \S~\ref{almprod}. Notice that giving a compatible sequence of quasi-isogenies $\varphi_n$ just amounts to giving $\varphi_1$; on the other hand we may see the compatible sequence $(A_n, \iota_n, \lambda_n, \eta_n)$ as a formal scheme with extra structure over $\Spf R^+$, which as above can be regarded as an adic space; we will denote it by $(\mathcal{A}, \iota, \lambda, \eta)$. Hence the compatible sequence of data $(A_n, \iota_n, \lambda_n, \eta_n, \varphi_n)$ uniquely corresponds to a datum of the form $(\mathcal{A}, \iota, \lambda, \eta, \varphi)$.

A similar description can be given of the functor of points of the other objects appearing in \eqref{prod-gen}, as well as of the good reduction locus $\mathcal{Y}_{U}(H)^\circ$ inside the analytification $\mathcal{Y}_{U}(H)$ of the space $Y_{U}(H)_{E_\wp}$; the space $\mathcal{Y}_{U}(H)^\circ$ is defined as the generic fibre of the completion along the special fibre of $Y_{U, \O_\wp}(H)$. In particular \cite[Lemma 2.2.2]{sw13} applies and yields a description of the functor of points of $\mathcal{Y}_{U}(H)^\circ$ similar to the one discussed above for $\mathcal{X}^b$. Notice that the inclusion $\mathcal{Y}_{U}(H)^\circ \subset \mathcal{Y}_{U}(H)$ is an equality if $T$ is non-empty, but it is strict if $T$ is empty.

We will now introduce versions with infinite level at $p$ of the spaces considered so far.

\subsubsection{The space $\mathcal{M}^b_{\infty}$}\label{rz-infin} We will first define the Rapoport--Zink space at infinite level $\mathcal{M}^b_{\infty}:=\prod_{\p \in \Sigma_p}\mathcal{M}^{b_\p}_{\infty}$. Each $\mathcal{M}^{b_\p}_{\infty}$ is a pro-étale cover of the generic fibre $\mathcal{M}^{b_\p}$ of the adic space attached to the formal scheme $\mathfrak{M}^{b_\p}$. 

Fix $\p \in \Sigma_p$; as recalled above, the functor of points $\mathrm{Perf}_{\breve{E}_\wp} \rightarrow \mathrm{Sets}$ of $\mathcal{M}^{b_\p}$ is the sheafification of the functor sending $(R, R^+)$ to $\varprojlim_n \mathfrak{M}^{b_\p}(R^+/p^n)$. Giving such a compatible system of $(R^+/p^n)$-points amounts to giving a compatible collection of $p$-divisible groups with extra structure $(\mathcal{G}_n, \iota_n, \lambda_n)$ as in \S~\ref{razin-fin} and a quasi-isogeny (respecting extra structures) $\rho: \mathcal{G}_1 \rightarrow \mathbb{X}^{b_\p} \times_{\bar{\mathbb{F}}_p} R^+/p$. 
Now, by~\cite[Lemma (4.16)]{mes72}, giving a compatible sequence of $p$-divisible groups with extra-structure $(\mathcal{G}_n, \iota_n, \lambda_n)$ is equivalent to giving a $p$-divisible group with extra structure $(\mathcal{G}, \iota, \lambda)$ over $R^+$. For every $n \geq 1$, the Tate module functor $T(\mathcal{G}_n)$ from $R^+/p^n$-algebras to sets sending $S$ to $\varprojlim_{k}\mathcal{G}_n[p^k](S)$ is represented by an affine scheme, flat over $\Spec(R^+/p^n)$, by \cite[Proposition 3.3.1]{sw13}. Taking the limit over $n$ and passing to the adic generic fibre we obtain an adic space which we denote by $\mathcal{T}(\mathcal{G})$.

The space $\mathcal{M}^{b_\p}_{\infty}$ parametrises trivialisations of the local system on $\mathcal{M}^{b_\p}$ given by $\mathcal{T}(\mathcal{G})$. More precisely, let $\Lambda_\p=\left(\begin{smallmatrix}
\mathcal{O}_{E_\p} & \mathcal{O}_{E_\p}\\
\mathcal{O}_{E_\p} & \mathcal{O}_{E_\p}
\end{smallmatrix}\right)$ if $\p \in \Sigma_p \smallsetminus T$ and $\Lambda_\p=\left(\begin{smallmatrix}
\p\mathcal{O}_{E_\p} & \mathcal{O}_{E_\p}\\
\p\mathcal{O}_{E_\p} & \mathcal{O}_{E_\p}
\end{smallmatrix}\right)$ if $\p \in T$ (as in \cite[p. 2154]{tixi16}). Following \cite[Definition 6.5.3]{sw13} we define 
\[
\mathcal{M}^{b_\p}_{\infty}: \mathrm{Perf}_{\breve{E}_\wp}\rightarrow \mathrm{Sets}
\] 
as the sheafification of the functor sending $(R, R^+)$ to the set of isomorphism classes of data
$(\mathcal{G}, \iota, \lambda, \rho, \alpha)$,
where $(\mathcal{G}, \iota, \lambda, \rho)$ is a datum as above, and
\begin{equation*}
\alpha: \Lambda_\p \rightarrow \mathcal{T}(\mathcal{G})(R, R^+)
\end{equation*}
is an $\O_{D, \p}$-linear map. We require that there is a compatible choice of $p$-power roots of unity in $R^+$, yielding a map $\varepsilon: \Z_p \rightarrow \mathcal{T}(\mu_{p^\infty})(R, R^+)$, such that the pairing $\mathrm{Tr} (\psi):\Lambda_\p \times \Lambda_\p \rightarrow \Z_p$ and the pairing $\mathcal{T}\mathcal{G}(R, R^+) \times \mathcal{T}\mathcal{G}(R, R^+) \rightarrow \mathcal{T}(\mu_{p^\infty})(R, R^+)$ induced by the polarisation are identified via the maps $\alpha$ and $\varepsilon$. Furthermore we ask that for every map $x: \Spa(K, K^+) \rightarrow \Spa(R, R^+)$ with $K$ a perfectoid field, the induced map $\Lambda_\p \rightarrow \mathcal{T}(\mathcal{G})(K, K^+)$ is an isomorphism. %

\subsubsection{The space $\mathcal{Y}_{U^p}(H)^{\circ}$}
Let $\mathcal{Y}_{U^p}(H)/\breve{E}_\wp$ be the inverse limit (as a diamond) of the analytifications of the varieties $Y_{U^pU_p}(H)_{\breve{E}_\wp}$, as the compact open subgroup $U_p\subset H(\Q_p)$ varies (cf. Remark~\ref{rem-models}). We denote by $\mathcal{Y}_{U^p}(H)^\circ\subset \mathcal{Y}_{U^p}(H)$ the preimage of $\mathcal{Y}_{U}(H)^\circ_{\breve{E}_\wp}$ via the natural projection map. The same argument used in \S~\ref{genfib-for} shows that $\mathcal{Y}_{U}(H)^\circ: \mathrm{Perf}_{\breve{E}_\p} \rightarrow \mathrm{Sets}$ is the sheafification of the functor sending $(R, R^+)$ to the set of isomorphism classes of data $(\mathcal{A}, \iota, \lambda, \eta)/\Spf R^+$. On the other hand $\mathcal{Y}_{U^p}(H)$ is the inverse limit of the analytifications of schemes over $Y_{U}(H)_{\breve{E}_\wp}$ relatively representing trivialisations (respecting extra structure) of the $p^n$-torsion in the universal abelian scheme. It follows that $\mathcal{Y}_{U^p}(H)^\circ: \mathrm{Perf}_{\breve{E}_\wp} \rightarrow \mathrm{Sets}$ 
is the sheafification of the functor sending $(R, R^+)$ to the set of isomorphism classes of data $(\mathcal{A}, \iota, \lambda, \eta, \alpha)$, where $\alpha: \Lambda_p=\oplus_{\p \in \Sigma_p}\Lambda_\p \rightarrow \mathcal{T}(\mathcal{A}[p^\infty])(R, R^+)$ is a trivialisation in the sense of \S~\ref{rz-infin}.

There is a continuous specialisation map $\mathcal{Y}_{U}(H)^\circ \rightarrow Y_{U}(H)\times_{\breve{\O}_\wp} \bar{\F}_p$; define $\mathcal{Y}_{U^p}(H)^{b} \subset \mathcal{Y}_{U^p}(H)^{\circ}$ to be the preimage of the Newton stratum in $Y_{U}(H)\times_{\breve{\O}_\wp} \bar{\F}_p$ corresponding to $b \in B(H_{\Q_p}, \mu)$ via the composition of the specialisation map and the projection map $\mathcal{Y}_{U^p}(H)^{\circ} \rightarrow \mathcal{Y}_{U}(H)^\circ_{\breve{E}_\wp}$. Hence $\mathcal{Y}_{U^p}(H)^{b}$ is a locally closed subspace of (the topological space underlying) $\mathcal{Y}_{U^p}(H)^{\circ}$.

\subsubsection{The space $\mathcal{X}^b_{\infty}$} Finally, we will need the infinite level version of the space $\mathcal{X}^b$. This is the sheafification of the functor sending $(R, R^+) \in \mathrm{Perf}_{\breve{E}_\wp}$ to the set of isomorphism classes of data $(\mathcal{A}, \iota, \lambda, \eta, \varphi, \alpha)$ where $(\mathcal{A}, \iota, \lambda, \eta, \varphi)$ is as in \S~\ref{genfib-for}, and $\alpha$ is a trivialisation of $\mathcal{T}(\mathcal{A}[p^\infty])(R, R^+)$ in the sense defined above.

\subsubsection{} From now on, unless otherwise stated we will regard all our objects as diamonds over $\Spd(\breve{E}_\wp, \breve{\O}_\wp)$; in particular, to define maps between them it suffices to give natural transformations between the functors on $\mathrm{Perf}_{\breve{E}_\wp}$ described above. We have a map $\mathcal{X}^b_{\infty}\rightarrow \mathcal{M}^b_\infty$, obtained sending an abelian scheme with extra structure to the associated $p$-divisible group with extra structure. This fits into a cartesian diagram

\begin{center}
\begin{tikzcd}
\mathcal{X}^b_{\infty} \arrow[r] \arrow[d] & \mathcal{M}^b_\infty \arrow[d] \\
\mathcal{X}^b \arrow[r] & \mathcal{M}^b.
\end{tikzcd}
\end{center}

\noindent Indeed, the above diagram is clearly cartesian on the non-sheafified functors of points, and sheafification commutes with finite limits. It follows that $\mathcal{X}^b_{\infty}$ is representable by an adic space. Furthermore, the product formula \eqref{prod-gen} is still true at infinite level, as
\begin{equation}\label{eq:infinite level product}
\mathcal{X}^b_\infty \simeq 
 \overline{\mathcal{I}g^b}\times_{\Spd(\breve{E}_\wp, \breve{\O}_\wp)} \mathcal{M}^b_\infty.
\end{equation}

The diamond $\mathcal{X}^b_\infty$ maps to $\mathcal{Y}_{U^p}(H)^{\circ}$ forgetting the quasi-isogeny $\varphi$; the underlying map of topological spaces factors through a map 
\begin{equation*}
q: \mathcal{X}^b_\infty \rightarrow \mathcal{Y}_{U^p}(H)^{b}.
\end{equation*}
We warn the reader that the subset $\mathcal{Y}_{U^p}(H)^{b}\subset \mathcal{Y}_{U^p}(H)$ is not necessarily generalising, hence it may not be the underlying topological space of a subdiamond of $\mathcal{Y}_{U^p}(H)$. For $(R, R^+) \in \mathrm{Perf}_{\breve{E}_\wp}$, we will denote by $\mathcal{Y}_{U^p}(H)^{b}(R, R^+)$ the set of maps from $\mathrm{Spa}(R, R^+)$ to $\mathcal{Y}_{U^p}(H)$ whose image is contained in $\mathcal{Y}_{U^p}(H)^{b}$.

Finally, we have global and local Hodge--Tate period maps
\begin{align*}
\pi_{\HT}^\circ: \mathcal{Y}_{U^p}(H)^{\circ} & \rightarrow \mathscr{F}\ell_{H_{\Q_p}, \mu}\\
\pi_{\HT}^b: \mathcal{M}^b_\infty & \rightarrow \mathscr{F}\ell_{H_{\Q_p}, \mu}.
\end{align*}

We can now state and prove an analogue of \cite[Lemma 4.3.20]{cs17}.
\begin{prop}\label{locglob-ht}
The following diagram (of functors on $\mathrm{Perf}_{\breve{E}_\wp}$) is cartesian.
\begin{center}
\begin{tikzcd}
\mathcal{X}^b_{\infty} \arrow[r] \arrow[d, "q"] & \mathcal{M}^b_\infty \arrow[d, "\pi_{\HT}^b"] \\
\mathcal{Y}_{U^p}(H)^{b} \arrow[r, "\pi^\circ_{\HT}"] & \mathscr{F}\ell_{H_{\Q_p}, \mu}.
\end{tikzcd}
\end{center}
\end{prop}
\begin{proof}
As sheafification commutes with finite limits, it is enough to prove the statement at the level of non-sheafified functors of points.
\begin{enumerate}
\item First we show that the diagram commutes. It is enough to prove this for $\Spa(K, K^+)$-points with $K$ a complete algebraically closed field and $K^+ \subset K$ an open bounded valuation subring, and replacing $\mathcal{Y}_{U^p}(H)^{b}$ with $\mathcal{Y}_{U^p}(H)$. Because $\mathscr{F}\ell_{H_{\Q_p}, \mu}$ is separated (it is even proper), we are reduced to prove commutativity on $\Spa(K, \cO_K)$-points. In this case it follows from the compatibility of the Hodge--Tate filtration for an abelian variety and the associated $p$-divisible group \cite[Proposition 4.15]{sch12}.
\item It remains to prove that the diagram is cartesian. Let $(R, R^+) \in \mathrm{Perf}_{\breve{E}_\wp}$ and take two points
\begin{equation*}
x=(\mathcal{G}, \iota_\G, \lambda_\G, \rho_\G, \alpha_\G) \in \mathcal{M}^b_\infty(R, R^+),
\end{equation*}
\begin{equation*}
y=(\mathcal{A}, \iota_\mathcal{A}, \lambda_\mathcal{A}, \eta_\mathcal{A}, \alpha_\mathcal{A}) \in \mathcal{Y}_{U^p}(H)^{b}(R, R^+)
\end{equation*}
mapping to the same $(R, R^+)$-point of the flag variety.
The composite $\alpha_\G \circ \alpha_A^{-1}: \mathcal{T}_p(\mathcal{A})(R, R^+)\rightarrow \mathcal{T}_p(\mathcal{G})(R, R^+)$ becomes an isomorphism when pulled back to any geometric rank one point $\Spa(K, \cO_K)$, and it is induced by an isomorphism $\phi: \mathcal{A}[p^\infty]_R\toisom \mathcal{G}_R$ respecting the extra structures. Furthermore, as $x$ and $y$ map to the same point in the flag variety, the map $\alpha_\G \circ \alpha_A^{-1}$ respects the Hodge--Tate filtrations (on Tate modules tensored by $K$). Finally, since the image of $y$ is contained in $\mathcal{Y}_{U^p}(H)^{b}$, the Newton polygon of $\mathcal{A}[p^\infty]\times_{R^+}R^+/p$ is constant. Hence, by \cite[Theorem B]{sw13} and \cite[Lemma 4.2.15]{cs17}, the isomorphism $\phi$ extends to an isomorphism, abusively denoted by the same symbol, $\phi: \mathcal{A}[p^\infty]\toisom \mathcal{G}$. We get a (unique) $(R, R^+)$-point $(\mathcal{A}, \iota_\mathcal{A}, \lambda_\mathcal{A}, \eta_\mathcal{A}, \alpha_\mathcal{A}, \rho_\G \circ \phi)$ of $\mathcal{X}^b_{\infty} $ mapping to $x$ and $y$.
\end{enumerate}
\end{proof}

\subsubsection{} The space $\mathcal{Y}_{U^p}(H)=\varprojlim_{U_p}\mathcal{Y}_{U^pU_p}$ maps to the inverse limit $\mathcal{S}h_{U^p}(H)$ of the analytifications $\mathcal{S}h_{U^pU_p}(H)$ of the Shimura varieties with level $U^pU_p$; this map is a torsor for the group $\Delta=\varprojlim_k \Delta_k$ defined in \S~\ref{igvar}. We let $\mathcal{S}h_{U^pU_{p}}(H)^\circ=\mathcal{Y}_{U^pU_p}^\circ/\Delta$ and $\mathcal{S}h_{U^p}(H)^\circ=\mathcal{Y}^\circ_{U^p}(H)/\Delta$, so that $\mathcal{S}h_{U^p}(H)^\circ/U_p\simeq \mathcal{S}h_{U^pU_p}(H)^\circ$. Similarly we set $\mathcal{S}h_{U^p}(H)^b=\mathcal{Y}_{U^p}(H)^b/\Delta$ for $b \in B(H_{\Q_p}, \mu)$.

\subsection{Fibres of the Hodge--Tate period map} Recall that we have fixed a subset $T \subset \Sigma_p$; we have a Hodge--Tate period map
\begin{equation*}
\pi^\circ_{\HT}: \mathcal{S}h_{U^p}(H)^\circ \rightarrow \mathscr{F}\ell_{H_{\Q_p}, \mu}
\end{equation*}
induced by the map (abusively denoted with the same symbol) $\pi_{\HT}^\circ: \mathcal{Y}_{U^p}(H)^{\circ} \rightarrow \mathscr{F}\ell_{H_{\Q_p}, \mu}$ introduced above, which is equivariant with respect to the $\Delta$-action (trivial on the target). 
Take an element $b \in B(H_{\Q_\p}, \mu)$, and fix a $p$-divisible group $\mathbb{X}^b$ (with extra structure) corresponding to it.

\begin{prop}\label{fibreig} 
For every geometric point $\tilde{x}$ of $\mathscr{F}\ell^b_{H_{\Q_p}, \mu}$ and every $i \geq 0$ there is a Hecke-equivariant isomorphism
\begin{equation*}
(R^i\pi^\circ_{\HT, *}\F_{\ell})_{\tilde{x}}\simeq H^i(\mathrm{Ig}^b, \F_{\ell}).
\end{equation*}
\end{prop}
\begin{proof}
This follows from Proposition \ref{locglob-ht} with the same argument as in \cite[Section 4.4]{cs17}. More precisely, we may assume that $\tilde{x}= \mathrm{Spa}(C, \mathcal{O}_C)$ is a rank one point; any rank one point in the preimage of $\tilde{x}$ in $\mathcal{Y}_{U^p}(H)^{\circ}$ is contained in $\mathcal{Y}_{U^p}(H)^{b}$. The cartesian diagram in Proposition \ref{locglob-ht}, \cite[Lemma 4.2.18]{cs17} and the product formula for $\mathcal{X}^{b}_\infty$ \eqref{eq:infinite level product} imply that the the fibre of $\pi_{\HT}^\circ: \mathcal{Y}_{U^p}(H)^{b} \rightarrow \mathscr{F}\ell^b_{H_{\Q_p}, \mu}$ at $\tilde{x}$ is isomorphic to $\overline{\mathcal{I}g^b}_{(C, \mathcal{O}_C)}$, hence the fibre of $\pi_{\HT}^\circ: \mathcal{S}h_{U^p}(H)^{b} \rightarrow \mathscr{F}\ell^b_{H_{\Q_p}, \mu}$ is isomorphic to $\mathcal{I}g^b_{(C, \mathcal{O}_C)}$. The result now follows as in \cite[pp. 728, 729]{cs17}.
\end{proof}

\section{The structure of the $\mu$-ordinary stratum at infinite level}

In this section, we show that the $\mu$-ordinary locus at infinite level is parabolically induced from the corresponding perfectoid Igusa variety. 

\subsection{Setup} Fix a prime $p$ and a subset $T \subsetneq \Sigma_p$ as in the previous section; we will mostly omit $T$ from our notations: for example, we  denote $H_T$ by $H$. As explained in \S~\ref{kotsets} we have $B(H_{\Q_p}, \mu)=\prod_{\p \in \Sigma_p} B(H_{\p}, \mu_{\p})$, where $B(H_{\p}, \mu_{\p})$ is a singleton if $\p \in T$ and has two elements - the basic one and the $\mu_\p$-ordinary one - if $\p \in \Sigma_p \smallsetminus T$. Consider the $\mu$-ordinary element $b^{\mathrm{ord}}=(b_\p)_{\p \in \Sigma_p}$ where $b_\p$ is the non-basic element for every $\p \in \Sigma_p \smallsetminus T$. Our aim is to study the structure of the stratum $\mathcal{S}h_{U^p}(H)^{b^\mathrm{ord}}$, proving that it is parabolically induced from the corresponding perfectoid Igusa variety. This rests on the product formula
\begin{equation*}
\mathcal{X}_\infty^{b^\mathrm{ord}}\simeq \overline{\mathcal{I}g^{b^\mathrm{ord}}}\times \mathcal{M}_\infty^{b^{\mathrm{ord}}}
\end{equation*}
we proved in the previous section and on the fact that the relevant Rapoport--Zink spaces are parabolically induced.
\subsubsection{The group $\mathcal{J}^{b^\mathrm{ord}}$} Let $\mathcal{J}^{b^\mathrm{ord}}=\prod_{\p \in \Sigma_p}\mathcal{J}^{b_\p}$ be the group of self-quasi-isogenies (respecting extra structure) of $\mathbb{X}^{b^\mathrm{ord}}$, seen as a functor $\mathrm{Perf}_{\breve{E}_\wp} \to \mathrm{Sets}$ sending $(R, R^+)$ to the set of self-quasi-isogenies of $\mathbb{X}^{b^\mathrm{ord}}\times_{\bar{\F}_p}R^+/p$. If $\p \in T$ then $\mathcal{J}^{b_\p}$ is just the constant group diamond attached to $H_\p(\Q_p)$, which will be denoted by $\underline{H_\p(\Q_p)}$. If $\p \in \Sigma_p \smallsetminus T$ then we can write $\mathcal{J}^{b_\p}=\underline{L_{\p}(\Q_p)} \ltimes \mathcal{J}_\p^{U}$ where $L_{\p}\subset H_\p$ is a Levi and $\mathcal{J}_\p^{U}$ is a (positive dimensional) group diamond (cf. \cite[Proposition 4.2.11]{cs17}). On the other hand we can attach to $\mathbb{X}^{b^\mathrm{ord}}$ an $H_{\Q_p}$-bundle $\mathcal{E}^{b^\mathrm{ord}}_{(R, R^+)}$ on the Fargues--Fontaine curve $X_{(R^\flat, R^{\flat,+})}$ for each $(R, R^+) \in \mathrm{Perf}_{\breve{E}_\wp}$, and look at the corresponding automorphism group functor $\mathrm{Aut}(\mathcal{E}^{b^\mathrm{ord}}): \mathrm{Perf}_{\breve{E}_\wp} \to \mathrm{Sets}$. 

We claim that $\mathrm{Aut}(\mathcal{E}^{b^\mathrm{ord}})\simeq \mathcal{J}^{b^\mathrm{ord}}$. Indeed, for $\p \in \Sigma_p \smallsetminus T$, \cite[Proposition III.5.1]{fargues-scholze} describes $\mathrm{Aut}(\mathcal{E}^{b_\p})$ in terms of $\underline{L_{\p}(\Q_p)}$ and of Banach--Colmez spaces. On the other hand, the computation in the proof of \cite[Proposition 4.2.11]{cs17} gives a similar descripion of $\mathcal{J}^{b_\p}$, with the universal cover of suitable $p$-divisible groups in place of Banach--Colmez spaces. The desired isomorphism follows from the identification between these two objects, cf. \cite[\S~15.2]{sw20}.

\subsubsection{} The description of the Igusa variety in Lemma \ref{ig-isog} implies that $\mathcal{J}^{b^\mathrm{ord}}$ acts on the left on $\overline{\mathcal{I}g}^{b^{\mathrm{ord}}}$, changing the trivialisation $\rho$. For the same reason, $\mathcal{J}^{b^\mathrm{ord}}$ acts on the left on $\mathcal{M}_\infty^{b^{\mathrm{ord}}}$ and $\mathcal{X}_\infty^{b^\mathrm{ord}}$. The stratum $\mathcal{Y}_{U^p}(H)^{b^\mathrm{ord}}\subset \mathcal{Y}_{U^p}(H)^\circ$ is open and contained in the preimage via $\pi_{HT}^\circ$ of the $\mu$-ordinary stratum in the flag variety. As the latter is a diamond, the same is true for $\mathcal{Y}_{U^p}(H)^{b^\mathrm{ord}}$. The map $\mathcal{X}^{b^\mathrm{ord}}_{\infty} \rightarrow \mathcal{Y}_{U^p}(H)^{b^\mathrm{ord}}$ is invariant with respect to the action of $\mathcal{J}^{b^\mathrm{ord}}$ on the source.
\begin{lem}\label{torsor-ord}\leavevmode
\begin{enumerate}
\item The map $\mathcal{J}^{b^\mathrm{ord}} \backslash \mathcal{X}^{b^\mathrm{ord}}_{\infty} \rightarrow \mathcal{Y}_{U^p}(H)^{b^\mathrm{ord}}$ is an isomorphism.
\item The induced map $\mathcal{J}^{b^\mathrm{ord}} \backslash (\mathcal{I}g^{b^{\mathrm{ord}}} \times \mathcal{M}_\infty^{b^{\mathrm{ord}}}) \rightarrow \mathcal{S}h_{U^p}(H)^{b^\mathrm{ord}}$ is an isomorphism.
\end{enumerate}
\end{lem}
\begin{proof}
The second point follows from the first and the product formula: indeed, as the map in \eqref{prod-gen} (as well as its counterpart at infinite level \eqref{eq:infinite level product}) is $\mathcal{J}^{b^\mathrm{ord}}$-equivariant, we get an isomorphism
\begin{equation*}
\mathcal{J}^{b^\mathrm{ord}} \backslash (\overline{\mathcal{I}g^{b^{\mathrm{ord}}}} \times \mathcal{M}_\infty^{b^{\mathrm{ord}}}) \rightarrow \mathcal{Y}_{U^p}(H)^{b^\mathrm{ord}}.
\end{equation*}
The above isomorphism is equivariant with respect to the $\Delta$-action on $\overline{\mathcal{I}g^{b^{\mathrm{ord}}}}$ and on $\mathcal{Y}_{U^p}(H)^{b^\mathrm{ord}}$; hence quotienting by this action we obtain the desired isomorphism.

To prove the first point, we need to show that the map (of pro-étale sheaves) $\mathcal{X}^{b^\mathrm{ord}}_{\infty} \rightarrow \mathcal{Y}_{U^p}(H)^{b^\mathrm{ord}}$ is surjective. Let $(R, R^+) \in \mathrm{Perf}_{\breve{E}_\wp}$ and let $\mathcal{G}/R^+$ be the $p$-divisible group with extra structure attached to an $(R, R^{+})$-point of the target. Then, after base change to any geometric point of $\mathrm{Spa}(R, R^+)$, we have a quasi-isogeny respecting extra structures between $\mathcal{G}\times_{R^+}R^+/p$ and $\mathbb{X}^{b^\mathrm{ord}}\times_{\bar{\mathbb{F}}_p}R^+/p$. Therefore the isocrystal attached to the $p$-divisible group with extra structure $\mathcal{G}\times_{R^+}R^+/p$ gives rise to an $H$-bundle $\mathcal{E}_{\mathcal{G}}$ on the Fargues--Fontaine curve $X_{(R^\flat, R^{\flat,+})}$ which is geometrically fibrewise isomorphic to the $H$-bundle $\mathcal{E}^{b^{\mathrm{ord}}}$ corresponding to $b^{\mathrm{ord}}$. By \cite[Proposition III.5.3]{fargues-scholze} there is, pro-étale locally on $\mathrm{Spa}(R^\flat, R^{+, \flat})$, an isomorphism $\mathcal{E}_{\mathcal{G}}\simeq \mathcal{E}^{b^{\mathrm{ord}}}$, yielding, pro-étale locally, a quasi-isogeny respecting extra-structures between $\mathcal{G}\times_{R^{+}}R^{+}/p$ and $\mathbb{X}^{b^\mathrm{ord}}\times_{\bar{\mathbb{F}}_p}R^{+}/p$.
\end{proof}

\subsection{The structure of $\mu$-ordinary Rapoport--Zink spaces}

\subsubsection{Moduli spaces of shtukas} Given $\p \in \Sigma_p$, we can attach to $(H_\p, b_\p, \mu_\p)$ a diamond $Sht_\infty(H_\p, b_\p, \mu_\p)$ over $\Spd(\breve{E}_\wp, \breve{\mathcal{O}}_{\wp})$ parametrising $H_\p$-shtukas with infinite level structure, as in~\cite[\S 23]{sw20}. More precisely, 
the functor $Sht_\infty(H_\p, b_\p, \mu_\p): \mathrm{Perf}_{\bar{\F}_p}\rightarrow \mathrm{Sets}$ sends $(R, R^+)$ to the set of isomorphism classes of data of the form $(S^\sharp, \mathcal{E}, \mathcal{E}', \phi, \rho, \rho_{\infty})$, where
\begin{enumerate}
\item $S^\sharp/\Spa(\breve{E}_\wp, \breve{\mathcal{O}}_\wp)$ is an untilt of $S:=\mathrm{Spa}(R, R^+)$.
\item $\mathcal{E}, \mathcal{E}'$ are $H_\p$-bundles over the relative Fargues-Fontaine curve $X_S$.
\item $\phi$ is an isomorphism between $\mathcal{E}$ and $\mathcal{E}'$ outside $S^\sharp$, which is a modification of type $\mu_\p$ at $S^\sharp$.
\item $\rho: \mathcal{E}' \rightarrow \mathcal{E}^{b_\p}$ is an isomorphism, where $\mathcal{E}^{b_\p}$ is the $H_\p$-bundle on $X_S$ coming from (the isocrystal attached to) $\mathbb{X}^{b_\p}$.
\item $\rho_\infty: \mathcal{E}^1 \rightarrow \mathcal{E}$ is an isomorphism, where $\mathcal{E}^1$ is the trivial $H_\p$-bundle on $X_S$. 
\end{enumerate}
By \cite[Corollary 24.3.5]{sw20}, the diamond $Sht_\infty(H_\p, b_\p, \mu_\p)$ is isomorphic to $\mathcal{M}^{b_\p}_\infty$.

We have a left (resp. right) action of $\mathcal{J}^{b_\p}=\mathcal{A}ut(\mathcal{E}^{b_\p})$ (resp. $\underline{H_\p(\Q_p)}=\mathcal{A}ut(\mathcal{E}^1)$) on $Sht_\infty(H_\p, b_\p, \mu_\p)$, changing the trivialisation $\rho$ (resp. $\rho_\infty$). The former corresponds to the action of $\mathcal{J}^{b_\p}$ on $\mathcal{M}^{b_\p}_\infty$ recalled above.

\subsubsection{The space $\mathcal{M}^{b_\p}_\infty$ for $\p \in T$} If $\p \in T$ then the cocharacter $\mu_\p$ is central, hence the flag variety $\mathscr{F}\ell_{H_{\p}, \mu_\p}$ is a point. By \cite[Proposition 19.4.2]{sw20} and \cite[proof of Proposition 23.3.3]{sw20} the diamond $\mathcal{M}^{b_\p}_\infty$ is isomorphic to the diamond $\underline{H_\p(\Q_p)}$ attached to the locally profinite set $H_\p(\Q_p)$. Hence the product $\mathcal{M}_{T, \infty}:=\prod_{\p \in T} \mathcal{M}^{b_\p}_\infty$ is isomorphic to $\prod_{\p \in T}\underline{H_\p(\Q_p)}$.

\subsubsection{The space $\mathcal{M}^{b_\p}_\infty$ for $\p \in \Sigma_p \smallsetminus T$}
Now fix $\p \in \Sigma_p \smallsetminus T$ and let $b_\p \in B(H_\p, \mu_\p)$ be the $\mu_\p$-ordinary element. The structure of $Sht_\infty(H_\p, b_\p, \mu_\p)$ is studied in \cite{gaim16} (generalising \cite{han18}), whose main result we now recall in the situation of interest to us - we warn the reader that our notation differs from the one in \cite{gaim16}. We have a forgetful map from $Sht_\infty(H_\p, b_\p, \mu_\p)$ to the Hecke stack $\mathrm{Hecke}(H_\p, b_\p, \mu_\p)$ which parametrises data $(S^\sharp, \mathcal{E}, \mathcal{E}', \phi)$ as above, such that $\mathcal{E}'$ (resp. $\mathcal{E}$) is fibrewise isomorphic to $\mathcal{E}^{b_\p}$ (resp. $\mathcal{E}^1$). The Grassmannian $Gr(H_\p, b_\p, \mu_\p)$ is the $\underline{H_\p(\Q_p)}$-torsor over $\mathrm{Hecke}(H_\p, b_\p, \mu_\p)$ parametrising trivialisations $\rho_\infty: \mathcal{E}^1 \rightarrow \mathcal{E}$. Consider the subfunctor $\mathcal{C}(H_\p, b_\p, \mu_\p)\subset Gr(H_\p, b_\p, \mu_\p)$ obtained imposing the condition that $\rho_\infty$ and $\phi$ are compatible with the filtrations, in the sense of \cite[p. 9]{gaim16}. The right action of $\underline{H_\p(\Q_p)}$ on $Gr(H_\p, b_\p, \mu_\p)$ induces a right action of a parabolic $\underline{P_\p(\Q_p)}\supset \underline{L_\p(\Q_p)}$ on the subspace $\mathcal{C}(H_\p, b_\p, \mu_\p)\subset Gr(H_\p, b_\p, \mu_\p)$. The natural map $\mathcal{C}(H_\p, b_\p, \mu_\p)\times \underline{H_\p(\Q_p)}\rightarrow Gr(H_\p, b_\p, \mu_\p)$ is invariant with respect to the left action of $\underline{P_\p(\Q_p)}$ on the source given by the product of the inverse of the right action on $\mathcal{C}(H_\p, b_\p, \mu_\p)$ and the left action by multiplication on $\underline{H_\p(\Q_p)}$. By \cite[Proposition 4.13]{gaim16} the induced map
\begin{equation}\label{grassm-induced}
\underline{P_\p(\Q_p)} \backslash \left(\mathcal{C}(H_\p, b_\p, \mu_\p)\times_{\Spd(\breve{E}_\wp, \breve{\mathcal{O}}_{\wp})} \underline{H_\p(\Q_p)}\right)\rightarrow Gr(H_\p, b_\p, \mu_\p)
\end{equation}
is an isomorphism.

Let us now consider the moduli space of shtukas with ``parabolic level structure" $Sht_{P_\p}(H_\p, b_\p, \mu_\p)$, defined as the fibre product
\begin{center}
\begin{tikzcd}
Sht_{P_\p}(H_\p, b_\p, \mu_\p) \arrow[r] \arrow[d] & Sht_\infty(H_\p, b_\p, \mu_\p) \arrow[d] \\
\mathcal{C}(H_\p, b_\p, \mu_\p) \arrow[r] & Gr(H_\p, b_\p, \mu_\p).
\end{tikzcd}
\end{center}
The isomorphism \eqref{grassm-induced} yields an isomorphism
\begin{equation}\label{parab-sht}
\underline{P_\p(\Q_p)} \backslash \left(Sht_{P_\p}(H_\p, b_\p, \mu_\p)\times_{\Spd(\breve{E}_\wp, \breve{\mathcal{O}}_{\wp})} \underline{H_\p(\Q_p)}\right)\rightarrow Sht_{\infty}(H_\p, b_\p, \mu_\p).
\end{equation}
\subsubsection{The structure of $Sht_{P_\p}(H_\p, b_\p, \mu_\p)$} The data $b_\p, \mu_\p$ are induced by analogous data for the Levi $L_\p$, which we will abusively denote by the same symbol. In other words, our data are HN-reducible in the sense of \cite[Definition 4.5]{gaim16} (notice that only one of the elements $b, b'$ considered in \emph{loc. cit.} is non-trivial in our case). We can consider the corresponding moduli space of shtukas $Sht_\infty(L_\p, b_\p, \mu_\p)$; in fact, as $L_\p$ is abelian, the same argument used above shows that $Sht_\infty(L_\p, b_\p, \mu_\p)$ is just the diamond attached to the locally profinite set $L_\p(\Q_p)$. Inducing bundles from $L_\p$ to $H_\p$ we obtain a map $Sht_\infty(L_\p, b_\p, \mu_\p)\rightarrow Sht_{P_\p}(H_\p, b_\p, \mu_\p)$. Finally, the group object $\mathcal{J}_\p^{U}$ acts on the left on $Sht_{P_\p}(H_\p, b_\p, \mu_\p)$, hence we obtain a map
\begin{equation*}
Sht_\infty(L_\p, b_\p, \mu_\p)\times_{\Spd(\breve{E}_\wp, \breve{\mathcal{O}}_{\wp})} \mathcal{J}_\p^{U} \rightarrow Sht_{P_\p}(H_\p, b_\p, \mu_\p).
\end{equation*}
By \cite[Proposition 4.21]{gaim16} the above map is an isomorphism. Joining this with \eqref{parab-sht} we obtain the following isomorphism, where fibre products are taken over $\Spd(\breve{E}_\wp, \breve{\mathcal{O}}_{\wp})$, which is omitted from the notation:
\begin{equation*}
Sht_{\infty}(H_\p, b_\p, \mu_\p)\simeq \underline{P_\p(\Q_p)} \backslash \left(\left(Sht_\infty(L_\p, b_\p, \mu_\p)\times \mathcal{J}_\p^{U}\right) \times \underline{H_\p(\Q_p)}\right).
\end{equation*}

Define $Sht_\infty(L^T, b^T, \mu^T):=\prod_{\p \in \Sigma_p \smallsetminus T}Sht_\infty(L_\p, b_\p, \mu_\p)$; similarly, we denote by $P^{T}(\Q_p), \mathcal{J}^{U, T}, H^T(\Q_p)$ the product of the objects $P_\p(\Q_p), \mathcal{J}^U_\p, H_\p(\Q_p)$ for $\p \in \Sigma_p \smallsetminus T$. With this notation, the space $\mathcal{M}^{b_\mathrm{ord}}_\infty$ can be written as
\begin{equation*}
\mathcal{M}^{b_\mathrm{ord}}_\infty=\mathcal{M}_{T, \infty} \times \underline{P^T(\Q_p)} \backslash \left(\left(Sht_\infty(L^T, b^T, \mu^T)\times \mathcal{J}^{U, T}\right) \times \underline{H^T(\Q_p)}\right).
\end{equation*}

\subsection{The almost product formula on the $\mu$-ordinary stratum}

\subsubsection{} We can now join Lemma \ref{torsor-ord} and the above description of $\mathcal{M}^{b_\mathrm{ord}}_\infty$. Let $\mathcal{J}^T:=\prod_{\p \in \Sigma_p \smallsetminus T}\mathcal{J}^{b_\p}$.  Since for $\p \in T$ we have $\mathcal{M}^{b_\p}_\infty \simeq \mathcal{J}^{b_\p}\simeq \underline{H_\p(\Q_p)}$ we obtain the following description of the ordinary stratum $\mathcal{S}h_{U^p}(H)^{b^\mathrm{ord}}$:
\begin{equation*}
\mathcal{S}h_{U^p}(H)^{b^\mathrm{ord}}=\mathcal{J}^T \backslash \left(\mathcal{I}g^{b^\mathrm{ord}}\times \underline{P^T(\Q_p)} \backslash \left(\left(Sht_\infty(L^T, b^T, \mu^T)\times \mathcal{J}^{U, T}\right) \times \underline{H^T(\Q_p)}\right) \right).
\end{equation*}
We will now show that the space on the right hand side is isomorphic to
\begin{equation}\label{prodfor-simp}
\underline{P^{T}(\Q_p)}\backslash \left(\mathcal{I}g^{b^\mathrm{ord}} \times \underline{H^T(\Q_p)} \right),
\end{equation}
where $\underline{P^{T}(\Q_p)}$ acts on the left on $	\mathcal{I}g^{b^\mathrm{ord}}$ via the natural inclusion in $\mathcal{J}^T$ and on $\underline{H^T(\Q_p)}$ via left multiplication.

Sending (the equivalence class of) a point $(i, g)$ to (the equivalence class of) $(i, ((1, 1), g))$ gives a well-defined map from the space in \eqref{prodfor-simp} to $\mathcal{S}h_{U^p}(H)^{b^\mathrm{ord}}$. This map is surjective because $\mathcal{J}^{T}$ acts transitively on $Sht_\infty(L^T, b^T, \mu^T)\times \mathcal{J}^{U, T}$; injectivity can be checked directly. Hence we obtain an isomorphism
\begin{equation*}
\mathcal{S}h_{U^p}(H)^{b^\mathrm{ord}}\simeq \underline{P^T(\Q_p)} \backslash \left(\mathcal{I}g^{b^\mathrm{ord}}\times \underline{H^T(\Q_p)} \right)
\end{equation*}
which is equivariant with respect to the right action of $\underline{H(\Q_p)}$ on the source and the target (here $\underline{H^T(\Q_p)}$ acts on the target via right multiplication on itself).
\begin{notn}
In the next theorem and its proof we will use the following notation. We set $P(\Q_p):=\prod_{\p \in T}H_\p(\Q_p)\times \prod_{\p \in \Sigma_p \smallsetminus T}P_\p(\Q_p) \subset H(\Q_p)$, and $P(\Z_p):=\prod_{\p \in T}H_\p(\Z_p)\times \prod_{\p \in \Sigma_p \smallsetminus T}P_\p(\Z_p) \subset H(\Z_p)$. We will denote by $\mathrm{Ind}$ smooth induction. 
\end{notn}

\begin{thm}\label{ord-par-ind}\leavevmode
\begin{enumerate}
\item There is an $\underline{H(\Q_p)}$-equivariant isomorphism
\begin{equation*}
\mathcal{S}h_{U^p}(H)^{b^\mathrm{ord}}\simeq \underline{P^T(\Q_p)} \backslash \left(\mathcal{I}g^{b^\mathrm{ord}}\times \underline{H^T(\Q_p)} \right).
\end{equation*}
\item  Let us turn the left action of $\underline{P(\Q_p)}$ on $\mathcal{I}g^{b^\mathrm{ord}}$ into a right action taking the inverse. Then the cohomology of $\mathcal{I}g^{b^\mathrm{ord}}$ (resp. $\mathcal{S}h_{U^p}(H)^{b^\mathrm{ord}}$) carries a left action of $P(\Q_p)$ (resp. $H(\Q_p)$).
For each $i \geq 0$, we have an isomorphism of
smooth $H(\Q_p)$-representations
\[
H^i(\mathcal{S}h_{U^p}(H)^{b^\mathrm{ord}}, \F_{\ell})\simeq 
\mathrm{Ind}_{P(\Q_p)}^{H(\Q_p)}\left( H^i(\mathcal{I}g^{b^\mathrm{ord}}, \F_{\ell})\right).
\]
\end{enumerate}
\end{thm}

\begin{proof}
The first point was established before the statement of the theorem; let us prove the second point.
Fix $i \geq 0$ and let $I:=\mathcal{I}g^{b^\mathrm{ord}}$ and $V:=H^i(I, \F_\ell)$.
The subgroup $P(\Z_p)\subset P(\Q_p)$ acts freely on $\overline{\mathcal{I}g^{b^\mathrm{ord}}}$. On the other hand, for $u \in \mathcal{O}_{E}^{\times}\cap U^p$, the action of $N_{E/F}(u)$ (changing polarisation) on $\overline{\mathcal{I}g^{b^\mathrm{ord}}}$ agrees with the action of $u^{-1} \in P(\Z_p)$. The induced action of $P(\Z_p)$ on $I$ hence factors through a free action of $P(\Z_p)/C$, where $C:=\varprojlim_k(\mathcal{O}_E^{\times}\cap U^p/\{u \in \mathcal{O}_E^{\times}\cap U^p \mid u \equiv 1 \pmod{p^k}\})$.

The action of $P(\Q_p)$ on $I$ induces a map $\rho: P(\Q_p) \rightarrow \mathrm{Aut}(V)$. The map $\iota: I \rightarrow \underline{P(\Q_p)}\backslash I \times \underline{H(\Q_p)}$ sending $x$ to $(x, 1)$ is equivariant with respect to the right $\underline{P(\Q_p)}$-action on source and target, hence it induces a $P(\Q_p)$-equivariant map
\begin{equation*}
\iota^*: H^i(\underline{P(\Q_p)}\backslash I \times \underline{H(\Q_p)}, \F_\ell)\rightarrow V.
\end{equation*}
We therefore have an $H(\Q_p)$-equivariant map 
\[
a_{\iota^*}: H^i(\underline{P(\Q_p)}\backslash I \times \underline{H(\Q_p)}, \F_\ell)\rightarrow \mathrm{Ind}_{P(\Q_p)}^{H(\Q_p)}(V)
\] 
induced by $\iota^*$ via Frobenius reciprocity, which sends a cohomology class $c$ to the function $H(\Q_p)\rightarrow V$ sending $h$ to $\iota^*\circ h^{*}(c)$.
For every compact open subgroup $K \subset H(\Q_p)$ we have the space $I_{K}:=\underline{P(\Q_p)}\backslash \left(I \times \underline{H(\Q_p)}/\underline{K}\right)$ and the natural map $q_K: \underline{P(\Q_p)}\backslash \left(I \times \underline{H(\Q_p)}\right) \rightarrow I_K$. On the other hand let
\begin{equation*}
S_{K}:=\{f: H(\Q_p) \rightarrow V \mid \forall p \in P(\Q_p), h \in H(\Q_p), k \in K, f(ph)=\rho(p)f(h), f(hk)=f(h)\},
\end{equation*}
so that $\mathrm{Ind}_{P(\Q_p)}^{H(\Q_p)}(\rho)=\varinjlim_{K} S_{K}$. We may, and will, restrict to pro-$p$ compact open subgroups $K\subset H(\Z_p)$; to complete the proof, we will show that the composite $a_{\iota^*}\circ q_K^*$ induces an isomorphism $H^i(I_{K}, \F_\ell) \simeq S_{K}$.

Choose a set of representatives $R=\{h_1, \ldots, h_r\}$ of the (finite) double coset $P(\Q_p) \backslash H(\Q_p)/K$ such that each $h_j$ belongs to $H(\Z_p)$ (this is possible because of the Iwasawa decomposition of $H(\Q_p)$). Evaluation at elements of $R$ yields an injection $S_K \hookrightarrow V^R$; an explicit computation shows that the image is $\oplus_{R}V^{\Gamma_j}$, where $\Gamma_j:=P(\Q_p) \cap h_jKh_j^{-1}$. Consider the composite of the map $a_{\iota^*}\circ q_K^*$ and of the injection $S_K \hookrightarrow V^R$, and let $\phi_j$ be its $h_j$-th component. Then $\phi_j$ sends $c\in H^i(I_K, \F_\ell)$ to $\iota^*\circ h_j^* \circ q_K^*(c)$.

We can write
\begin{equation}\label{split-igusa}
I_{K}=\coprod_R \underline{P(\Q_p)}\backslash (I \times \underline{P(\Q_p)h_jK}/\underline{K})\simeq \coprod_{R}\underline{\Gamma_j} \backslash I.
\end{equation}
The isomorphism is induced on the $h_j$-th component by the composite of the map 
\begin{equation*}
\underline{P(\Q_p)}\backslash (I \times \underline{P(\Q_p)h_jKh_{j}^{-1}}/\underline{h_jKh_j^{-1}}) \rightarrow \underline{P(\Q_p)}\backslash (I \times \underline{P(\Q_p)h_jK}/\underline{K})
\end{equation*}
given by right multiplication by $h_j$ and the map
\begin{equation*}
\underline{\Gamma_j} \backslash I \rightarrow \underline{P(\Q_p)}\backslash (I \times \underline{P(\Q_p)h_jKh_{j}^{-1}}/\underline{h_jKh_j^{-1}})
\end{equation*}
sending everything to the identity on the second component. Therefore, via the isomorphism in \eqref{split-igusa}, the map $\phi_j$ is identified with the pullback $H^i(\underline{\Gamma_j}\backslash I, \F_\ell)\rightarrow H^i(I, \F_\ell)$.

It remains to show that each of the pullback maps above induces an isomorphism $H^i(\underline{\Gamma_j}\backslash I, \F_\ell)\xrightarrow{\sim} V^{\Gamma_j}$. Each group $\Gamma_j$ is contained in $P(\Z_p)$, hence the action of $\Gamma_j/(\Gamma_j\cap C)$ on $I$ is free. Therefore, as $K$ is a pro-$p$ group and $p \neq \ell$, the desired isomorphism follows from \cite[Proposition 4.3.2]{wei17} and \cite[Theorem 2.2.7]{cgh20}. 
\end{proof}

\section{The structure of the cohomology in the non-Eisenstein case}

The aim of this section is to prove our main results
on the cohomology of Hilbert modular varieties after localisation at 
a non-Eisenstein maximal ideal. These are Theorems~\ref{mainpart2} 
and~\ref{mainpart2-ref} stated below. 
We also establish analogous results for quaternionic
Shimura varieties: see Theorem~\ref{quat-maithm}. 

\subsection{Main results: statements}
\begin{thm}\label{mainpart2}
Let $\ell>2$ be a prime, $K\subset G(\A_f)$ a neat compact open and $\m \subset \mathbb{T}$ 
a maximal ideal in the support of $H^i(Sh_K(G), \F_{\ell})$. Let $\bar{\rho}_{\m}$ be 
the Galois representation associated to $\m$ by Theorem~\ref{constr-gal-repn}. 
Assume that the image of $\bar{\rho}_{\m}$ is not solvable.
Then $H^i(Sh_K(G), \F_{\ell})_\m=H^i_c(Sh_K(G), \F_{\ell})_\m$ is non-zero only for $i=g$.
\end{thm}

\begin{cor}\label{torsfree-Zl}
With the same notations and assumptions as in Theorem \ref{mainpart2}, $H^i(Sh_K(G), \Z_{\ell})_\m\neq 0$ if and only if $i=g$, 
and $H^g(Sh_K(G), \Z_{\ell})_\m$ is torsion-free. The same is true for compactly supported cohomology.
\end{cor}

\begin{proof}
Looking at the long exact sequence in Betti cohomology coming from the short exact sequence $0 \rightarrow \Z_{\ell} 
\rightarrow \Z_{\ell} \rightarrow \F_{\ell}\rightarrow 0$ we see that multiplication by $\ell$ is surjective on $H^i(Sh_K(G), \Z_{\ell})_\m$ 
if $i \neq g$, hence $H^i(Sh_K(G), \Z_{\ell})_\m=0$ 
for $i \neq g$. On the other hand, we have
\begin{equation*}
\ldots H^{g-1}(Sh_K(G), \F_{\ell})_\m \rightarrow H^g(Sh_K(G), \Z_{\ell})_\m \rightarrow H^g(Sh_K(G), \Z_{\ell})_\m \ldots;
\end{equation*}
as the first term vanishes, we deduce that $H^g(Sh_K(G), \Z_{\ell})_\m$ is torsion-free, hence free. 
Finally, the cokernel of the map $H^g(Sh_K(G), \Z_{\ell})_\m \rightarrow H^g(Sh_K(G), \F_{\ell})_\m$ 
injects into $H^{g+1}(Sh_K(G), \Z_{\ell})_\m=0$, hence $H^g(Sh_K(G), \Z_{\ell})_\m \neq 0$. 
The result for $H^*_{c}(Sh_K(G), \Z_{\ell})_\m$ follows as in Lemma \ref{cohom-noneis}.
\end{proof}

To proceed, we make the following definition. 

\begin{defn}\label{defn:genericity}\leavevmode
\begin{enumerate}
\item  Let $v$ be a place of $F$ above a prime $p\not=\ell$ which splits completely in $F$. 
We say that $\bar{\rho}_\m$ is \emph{generic at $v$} if it is unramified at $v$ and
the eigenvalues of Frobenius at $v$ have ratio different from $p^{\pm 1}$.
\item Let $p \neq \ell$ be a prime. We say that $p$ is a \emph{decomposed generic} prime for $\bar{\rho}_\mathfrak{m}$ if $p$ splits completely in $F$ and $\bar{\rho}_{\m}$ is generic at $v$ for every place $v$ of $F$ above $p$.  
\item We say that $\bar{\rho}_{\m}$ is \emph{decomposed generic} if there exists a prime $p \neq \ell$ 
which is decomposed generic for $\bar{\rho}_{\m}$. 
\end{enumerate}
\end{defn}

\subsubsection{}\label{notransfer} Assume that $\bar{\rho}_{\m}$ is generic at some place $v$ 
of $F$. Then any characteristic zero lift of the restriction of $\bar{\rho}_{\m}$ to $\Gamma_{F_v}$ 
cannot be associated via the local Langlands correspondence to an irreducible smooth  
representation of $\GL_2(F_v)$ which transfers to the non-split quaternion algebra over $F_v$.

\begin{notn} Let $p \neq \ell$ be a prime which splits completely in $F$ 
and such that $\bar{\rho}_\m$ is unramified at every place of $F$ above $p$. 
We will denote by $\delta_p(\m)$ the number of places 
above $p$ at which $\bar{\rho}_\m$ is \emph{not} generic.
\end{notn}

\begin{thm}\label{mainpart2-ref}
Let $\ell>2$ be a prime. Let $p \neq \ell$ be an odd prime which splits completely in $F$ and such that $K=K^pK_p$ with $K_p$ hyperspecial. Let $\m \subset \mathbb{T}$ be a non-Eisenstein maximal ideal. Then
$H^*_c(Sh_K(G), \F_{\ell})_\m=H^*(Sh_K(G), \F_{\ell})_\m$ vanishes outside the interval $[g-\delta_p(\m), g+\delta_p(\m)]$.
\end{thm}

\begin{rem}
For a maximal ideal $\m \subset \mathbb{T}$ in the support of $H^i(Sh_K(G), \F_{\ell})$, the projective image $I$ of the Galois representation $\bar{\rho}_\m$ is a finite subgroup of $\PGL_2(\bar{\F}_\ell)$.  By Dickson's theorem, cf.~\cite[Theorem 2.47 (b)]{ddt}, the group $I$ is either conjugate to a subgroup of the upper triangular matrices, or to $\PGL_2(\F_{\ell^k})$ or $\PSL_2(\F_{\ell^k})$,
for some $k\geq 1$, or it is isomorphic to one of $D_{2n}$, for some $n\in \Z_{> 1}$ prime to $\ell$, 
$A_4$, $S_4$, or $A_5$. It follows that the image of $\bar{\rho}_\m$ is not solvable if and only if the following holds:
\begin{enumerate}
\item if $\ell=3$ then $I$ is isomorphic to $A_5$ or it contains a conjugate of $\PSL_2(\F_9)$;
\item if $\ell>3$ then $I$ is isomorphic to $A_5$ or it contains a conjugate of $\PSL_2(\F_{\ell})$.
\end{enumerate}
Theorem \ref{mainpart2} follows from Theorem \ref{mainpart2-ref} in view of the following lemma (more precisely, observe that the proof of the lemma allows to produce $p$ as in Definition~\ref{defn:genericity}$(3)$ satisfying the conditions of Theorem \ref{mainpart2-ref}).
\end{rem}

\begin{lem}\label{cebot}
Assume that $\ell>2$ and that the image of $\bar{\rho}_{\m}$ is not solvable.
Then $\bar{\rho}_{\m}$ is decomposed generic in the sense of Definition~\ref{defn:genericity}. 
\end{lem}

\begin{proof}
This is a variation of \cite[Lemma 2.3]{na20}; following \emph{loc. cit.}, 
we prove that there are infinitely many primes $p\equiv 1 \pmod{\ell}$, 
totally split in $F$, and such that $\bar{\rho}_\m$ is generic at every place above $p$.
This is deduced from our large image assumption using the Chebotarev density theorem. 

We first make some preliminary reductions. Let $\tilde{F}$ be the normal closure of $F$ in $\C$; 
notice that $\tilde{F}$ is also a totally real field.  
Let $\mathrm{pr}: \GL_2(\bar{\F}_{\ell})\rightarrow \PGL_2(\bar{\F}_{\ell})$ be the projection map. 
As observed before the statement of the theorem, the image $I$ of 
\[
\mathrm{pr} \circ \bar{\rho}_\m: \Gamma_F \rightarrow \PGL_2(\bar{\F}_{\ell})
\]
is isomorphic to $A_5$, or conjugate to one of $\PGL_2(\F_{\ell^k})$ or $\PSL_2(\F_{\ell^k})$, 
where $k\geq 2$ if $\ell = 3$ and $k\geq 1$ if $\ell>3$.
We claim that the same is true for the image 
$\tilde{I}$ of $\mathrm{pr} \circ \bar{\rho}_{\m|\Gamma_{\tilde{F}}}: \Gamma_{\tilde{F}} \rightarrow \PGL_2(\bar{\F}_\ell)$. 
Since $\tilde{F}/F$ is Galois, the group $\tilde{I}$ is a normal subgroup of $I$. 

\begin{enumerate}
\item Assume first that the group $I$ is isomorphic to $A_5$. Since $A_5$ is simple and 
$\tilde{I}$ is a normal subgroup of $I$, it is enough to show that $\tilde{I}$ is non-trivial. This is true because 
$\tilde{I}$ contains the image under $\mathrm{pr}\circ \bar{\rho}_{\m|\Gamma_{\tilde{F}}}$ 
of any complex conjugation in $\Gamma_{\tilde{F}}$, which is conjugate
to the matrix $\left(\begin{smallmatrix}
1 & 0\\
0 & -1
\end{smallmatrix}\right)$ because $\bar{\rho}_{\m}$ is totally odd.  

\item Assume now that the group $I$ is conjugate to 
$\PGL_2(\F_{\ell^k})$ or $\PSL_2(\F_{\ell^k})$. We 
conjugate everything so that $I$ is identified 
with $\PGL_2(\F_{\ell^k})$ or $\PSL_2(\F_{\ell^k})$ 
inside $\PGL_2(\bar{\F}_\ell)$. It is enough to prove that
$\tilde{I}\supseteq  \PSL_2(\F_{\ell^k})$. As the representation 
$\bar{\rho}_{\m}$ is totally odd, the image 
via $\bar{\rho}_{\m|\Gamma_{\tilde{F}}}$ of any complex conjugation 
$c\in \Gamma_{\tilde{F}}$ 
is a non-scalar semisimple element. 
As a normal subgroup of $I$, the group
$\tilde{I}$ contains the projective image of 
every $\SL_2(\F_{\ell^k})$-conjugate of $\bar{\rho}_{\m}(c)$. 
In particular, $\tilde{I}\cap \PSL_2(\F_{\ell^k})$ 
contains the ratio of any two distinct such conjugates and is, 
therefore, a non-trivial normal subgroup of $\PSL_2(\F_{\ell^k})$. 
Finally, the groups $\PSL_2(\F_{\ell^k})$, with $k\geq 2$ if $\ell =3$ and 
$k\geq 1$ if $\ell>3$, are simple. 
Therefore, $\tilde{I}\supseteq  \PSL_2(\F_{\ell^k})$. 
\end{enumerate}

\noindent This proves the claim. 
Up to replacing $\tilde{F}$ by a finite abelian extension $F'$ 
and conjugating $\bar{\rho}_{\m}$, 
we may ensure that the image $I'$ of $\mathrm{pr} \circ \bar{\rho}_{\m|\Gamma_{F'}}$ equals either $A_5$ or 
$\PSL_2(\F_{\ell^k})$, with $k\geq 2$ if $\ell =3$ and 
$k\geq 1$ if $\ell>3$. 
As $A_5$ and $\PSL_2(\F_{\ell^k})$ are both simple, $I'$ is unchanged if we replace $F'$ by its normal closure. 
For the same reason we may adjoin $\zeta_{\ell}$ to $F'$ without changing $I'$. 

Let $\Gamma$ denote either one of the finite simple groups $A_5$ or 
$\PSL_2(\F_{\ell^k})$. Let $L$ be the normal closure of the extension of $F'$ 
cut out by $I'$. Since $\Gamma$ is simple, Goursat's lemma implies 
that $\Gal(L/F')\simeq \Gamma^t$ for some $t\geq 1$. We claim that 
we can choose an element $1\neq g\in \Gamma$ which is semisimple
when viewed as an element of $\PGL_2(\bar{\F}_{\ell})$. In the case
$\Gamma = \PSL_2(\F_{\ell^k})$, this is clear. In the case $\Gamma = A_5$,
choose any element of order $2$; since $\ell\not=2$, 
such an element must be semisimple. The Chebotarev density theorem 
ensures the existence of a place $\p$ of $F'$ with residue 
field $\Fp$ with $p \neq \ell$ unramified in $F$ and 
such that $\Frob_\p^{L/F'}$ is conjugate to $(g, g, \ldots, g)$. Since
$F'/\Q$ is Galois, $p$ is totally split in $F'$.
We claim that, 
for every place $\p'$ of $F'$ above $p$, 
the element $\bar{\rho}_\m(\Frob_{\p'})$ is semisimple and 
different from the identity in $\PGL_2(\bar{\F}_{\ell})$.
In the case $\Gamma = \PSL_2(\F_{\ell^k})$, this follows 
as in~\cite[Lemma 2.3]{na20}. In the case $\Gamma = A_5$,
the argument in \emph{loc. cit.} shows that
each such element has order $2$ in $\PGL_2(\bar{\F}_{\ell})$, which implies that it 
must also be semisimple and different from the identity. 
As $\zeta_{\ell} \in F'$ we have $p \equiv 1 \pmod {\ell}$, 
hence the eigenvalues of each $\bar{\rho}_\m(\Frob_{\p'})$ cannot have ratio $p^{\pm 1}$.
\end{proof}

\subsection{Proof of Theorem \ref{mainpart2-ref}}\label{main2-ref-proof}

\subsubsection{Step 1: setup and choice of the auxiliary data}\label{main-setup}

Notice that the quantity $\delta_p(\m) = \delta_p(\m^\vee)$. By Lemma \ref{cohom-noneis} it suffices to show the following implication:
\begin{equation}\label{vanish-lowcsupp}
i < g-\delta_p(\m) \Rightarrow H^i_c(Sh_K(G), \F_{\ell})_\m=0.
\end{equation}
Choose an auxiliary $CM$ extension $E$ of $F$ such that every $\p \in \Sigma_p$ is inert in $E$, and choose $K_E=(\mathcal{O}_E\otimes \Z_p)^\times K_E^{p}\subset T_E(\A_f)$ sufficiently small with respect to $K$. With the notation of \S~\ref{unit-datum}, consider the unitary group $H:=H_{\emptyset}$ and let $U \subset H(\A_f)$ be the image of $K \times K_E$. Recall that the Hecke algebra $\mathbb{T}$ (defined in \S~\ref{def-hecke}) acts on the cohomology of $Sh_{U}(H)$, and Corollary \ref{Hilbtounitary} ensures that, for $i \geq 0$,
\begin{equation*}
H^i(Sh_{K}(G), \F_{\ell})_\m \neq 0 \Leftrightarrow H^i(Sh_{U}(H), \F_{\ell})_\m \neq 0.
\end{equation*}
Therefore, for the sake of proving Theorem \ref{mainpart2-ref}, we may (and will) work with $Sh_U(H)$. Let $\Shs$ be the base change to $\bar{\F}_p$ of its integral model. We have isomorphisms, for $i \geq 0$,
\begin{equation*}
H^i_c(Sh_U(H), \F_{\ell})\simeq H^i_c(\Shs, R\Psi \F_{\ell})\simeq H^i_c(\Shs, \F_{\ell}).
\end{equation*}
The first is obtained 
as in \cite[Corollary 5.20]{ls18}; the second follows from the fact that the complex of nearby cycles $R\Psi\F_{\ell}$ is quasi-isomorphic to the constant sheaf $\F_{\ell}$ as the integral model is smooth.
Hence it suffices to show the implication
\begin{equation}\label{toshow}
i<g-\delta_p(\m) \Rightarrow H^i_c(\Shs, \F_{\ell})_\m=0.
\end{equation}

\subsubsection{Step 2: the Newton stratification}

We have a Newton stratification on $\Shs$ indexed by elements $b=(b_1, \ldots, b_g)$ with $b_j \in B(H_{\p_j}, \mu_{\p_j})$. Recall that each $B(H_{\p_j}, \mu_{\p_j})$ consists of one $\mu$-ordinary and one basic element. We denote by $\Shs^b \subset \Shs$ the stratum corresponding to $b$. More generally, for each $0 \leq k \leq g$ and every $c \in \prod_{j \leq k}B(H_{\p_j}, \mu_{\p_j})$ we let $\Shs^c \subset \Shs$ be the union of the Newton strata corresponding to elements $b=(b_1, \ldots, b_g)$ such that $(b_1, \ldots, b_k)=c$. In particular for $k=0$ we obtain $\Shs$ and for $k=g$ we recover Newton strata.

As we are excluding from $\mathbb{T}$ the Hecke operators at places above $p$, the Hecke algebra acts on the cohomology of each stratum. The first (elementary) observation is that it suffices to show \eqref{toshow} for each Newton stratum.
\begin{lem}\label{stratlem}
Assume that $H^i_{c}(\Shs^b, \F_{\ell})_\m=0$ for every $i<g-\delta_p(\m)$ and every $b \in B(H_{\Q_p}, \mu)$. Then $H^i_{c}(\Shs, \F_{\ell})_\m=0$ for $i<g-\delta_p(\m)$.
\end{lem}
\begin{proof}
This follows from additivity of compactly supported cohomology. To be precise, we show by descending induction on $0 \leq k \leq g$ that
\begin{equation*}
\forall \; c \in \prod_{j \leq k}B(H_{\p_j}, \mu_{\p_j}), \forall \; i <g-\delta_p(\m), H^i_c(\Shs^c, \F_{\ell})_\m=0.
\end{equation*}
By hypothesis the statement is true for $k=g$. Now take $k<g$ and assume that the statement is true for $k+1$. Take $c \in \prod_{j \leq k}B(H_{\p_j}, \mu_{\p_j})$. Let $o$ (resp. $b$) be the non-basic (resp. basic) element in $B(H_{\p_{k+1}}, \mu_{\p_{k+1}})$ and consider $(c, o), (c, b) \in \prod_{j \leq k+1}B(H_{\p_j}, \mu_{\p_j})$. We have
\begin{equation*}
\Shs^{(c, o)} \subset \Shs^c \supset \Shs^{(c, b)};
\end{equation*}
furthermore $\Shs^{(c, o)}$ is open in $\Shs^c$ and $\Shs^{(c, b)}$ is the closed complement. By induction we know that $H^i_c(\Shs^{(c, o)}, \F_{\ell})_\m=H^i_c(\Shs^{(c, b)}, \F_{\ell})_\m=0$ for $i<g-\delta_p(\m)$. The exact sequence
\begin{equation*}
\ldots H^i_c(\Shs^{(c, o)}, \F_{\ell})_\m \rightarrow H^i_c(\Shs^c, \F_{\ell})_\m \rightarrow H^i_c(\Shs^{(c, b)}, \F_{\ell})_\m \ldots
\end{equation*}
implies that $H^i_c(\Shs^{c}, \F_{\ell})_\m=0$ for $i<g-\delta_p(\m)$.
\end{proof}

We need to show that the assumption of the lemma holds true in our situation. The first key ingredient is the following.

\begin{lem}\label{str-aff}
Let $b \in B(H_{\Q_p}, \mu)$.
\begin{enumerate}
\item If $b$ is not the $\mu$-ordinary element then the stratum $\Shs^b$ is smooth, affine and of dimension the number of non-basic coordinates of $b$.
\item If $b$ is the $\mu$-ordinary element then $\Shs^b$ is smooth of dimension $g$, and the partial minimal compactification of $\Shs^b$ (i. e. the union of $\Shs^b$ and the cusps in the minimal compactification of $\Shs$) is affine.
\end{enumerate}
\end{lem}
\begin{proof}
Dimension of strata can be obtained from \cite[Proposition 4.7]{tixi16}. As $p$ splits completely in $F$ Newton strata coincide with Ekedahl--Oort strata. Each of them is smooth by \cite[Theorem 3.4.7]{shezha18}. Furthermore, recall that $\Shs$ is a finite union of quotients of connected components of integral models of Hodge-type Shimura varieties by finite groups. Each Ekedahl--Oort stratum in $\Shs$ decomposes accordingly, hence its partial minimal compactification is affine by \cite[Proposition 6.3.1]{goko19} (see also~\cite{boxer}). Finally, the $\mu$-ordinary stratum is the only one intersecting the boundary, hence the lemma follows.
\end{proof}

\begin{cor}\label{bistr-die}
Let $b \in B(H_{\Q_p}, \mu)$ be such that $\dim \Shs^b\geq g-\delta_p(\m)$. Then $H^i_c(\Shs^b, \F_{\ell})_\m=0$ for $i<g-\delta_p(\m)$.
\end{cor}
\begin{proof}
Let us first assume that $b$ is not the $\mu$-ordinary element and let us set $d_b:=\dim \Shs^b$. By Artin vanishing $H^i(\Shs^b, \F_{\ell})=0$ for $i>d_b$, and by Poincaré duality $H^i_c(\Shs^b, \F_{\ell})=0$ for $i<d_b$. In particular $H^i_c(\Shs^b, \F_{\ell})_\m=0$ for $i<g-\delta_p(\m)$.

Now let us consider the ordinary stratum $\Shs^b$ and its partial minimal compactification $\Shs^{b, *}$, with boundary $\partial$. Let $j^b: \Shs^b \rightarrow \Shs^{b, *}$ and $i^b: \partial \rightarrow \Shs^{b, *}$ be the inclusions. By Artin vanishing $H^i(\Shs^{b, *}, j^b_!\F_{\ell})=0$ for $i>g$, and by Verdier duality $H^i_c(\Shs^{b, *}, Rj^b_*\F_{\ell})=0$ for $i<g$. Finally, we have an exact sequence
\begin{equation*}
\ldots \rightarrow H^i_c(\Shs^b, \F_{\ell})\rightarrow H^i_c(\Shs^{b, *}, Rj^b_*\F_{\ell}) \rightarrow H^i_c(\partial, i^{b, *}Rj^b_*\F_{\ell})\rightarrow \ldots.
\end{equation*}
To end the proof it suffices to show that $H^*_c(\partial, i^{b, *}Rj^b_*\F_{\ell})_\m=0$. Denoting by $j: \Shs \rightarrow \Shs^{*}$ and $i: \partial \rightarrow \Shs^{*}$ the inclusions we have $i^{b, *}Rj^{b}_*\F_\ell \simeq i^{*}Rj_*\F_\ell$. It follows that it suffices to prove that, after localisation at $\m$, the maps $H^*_c(\Shs, \F_\ell)\rightarrow H^*_c(\Shs^{*}, Rj_*\F_\ell)$ are bijective. In other words we have to show that the natural maps
\begin{equation*}
H^*_c(\Shs, \F_\ell)_\m\rightarrow H^*(\Shs, \F_\ell)_\m
\end{equation*}
are isomorphisms. Using \cite[Corollary 5.20]{ls18} we see that we can replace $\Shs$ by $Sh_U(H)$; since $\m$ is non-Eisenstein, the conclusion then follows from Lemma \ref{cohom-noneis} and Corollary \ref{Hilbtounitary}.
\end{proof}

To show Theorem \ref{mainpart2-ref} it remains to prove the following result.

\begin{prop}\label{smallstr-die}
Let $b \in B(H_{\Q_p}, \mu)$ and let $\m \subset \mathbb{T}$ be a non-Eisenstein maximal ideal. If $\dim \Shs^b<g-\delta_p(\m)$ then for every $i\geq 0$ we have $H^i(\Shs^b, \F_{\ell})_\m=0$.
\end{prop}

Indeed, applying the above proposition to $\m^\vee$ and using Corollary \ref{bistr-die} we obtain that $H^i_c(\Shs^b, \F_{\ell})_\m=0$ for $i<g-\delta_p(\m)$ and for every $b \in B(H_{\Q_p}, \mu)$. By Lemma \ref{stratlem} we conclude that $H^i_{c}(\Shs, \F_{\ell})_\m=0$ for $i<g-\delta_p(\m)$.

Finally, we will deduce Proposition \ref{smallstr-die} from the following result, whose proof will be given below.

\begin{prop}\label{vanish-smallig}
If $\dim \Shs^b<g-\delta_p(\m)$ then $H^i(\mathrm{Ig}^b, \F_{\ell})_\m=0$ for every $i$.
\end{prop}

\subsubsection{} Assuming Proposition~\ref{vanish-smallig}, let us prove Proposition \ref{smallstr-die}. Observe that the stratum $\Shs^b$ consists of a unique leaf; let $\mathbb{X}^b$ be the $p$-divisible group attached to a geometric point in the leaf, and $\Gamma^b$ the automorphism group of $\mathbb{X}^b$ (respecting extra structures). Let $Y_U(H)_{\bar{\F}_p}^{b, \mathrm{perf}}$ be the perfection of $Y_U(H)_{\bar{\F}_p}^b$;
the forgetful map $\overline{\mathrm{Ig}}^b\rightarrow Y_U(H)_{\bar{\F}_p}^b$ factors through a map $\overline{\mathrm{Ig}}^b\rightarrow Y_U(H)_{\bar{\F}_p}^{b, \mathrm{perf}}$, which is a $\Gamma^b$-torsor (cf. \cite[Proposition 2.2.6]{cs19}). For $u \in \mathcal{O}_{E}^{\times}\cap U^p$, the action of $N_{E/F}(u)$ on $\overline{\mathrm{Ig}}^b$ changing the polarisation agrees with the action of $u^{-1}$ on $\overline{\mathrm{Ig}}^b$ via the inclusion $\mathcal{O}_{E}^{\times} \hookrightarrow \Gamma^b$. Hence the induced map $\mathrm{Ig}^b\rightarrow \Shs^{b, \mathrm{perf}}$ is a torsor for the group $\Gamma^b/\varprojlim_k(\mathcal{O}_E^{\times}\cap U/\{u \in \mathcal{O}_E^{\times}\cap U \mid u \equiv 1 \pmod{p^k}\})$. 
Therefore there is a Hecke-equivariant Hochschild--Serre spectral sequence relating the cohomology of $\mathrm{Ig}^b$ and the cohomology of $\Shs^{b, \mathrm{perf}}$. The latter agrees with étale cohomology of $\Shs^b$ by topological invariance of the étale site; hence if $H^i(\mathrm{Ig}^b, \F_{\ell})_\m=0$ for every $i$ then the same is true for $H^i(\Shs^b, \F_{\ell})_\m$.

\subsection{The excision triangle for diamonds} In the proof of Proposition~\ref{vanish-smallig}, given in the next section, we will make use of the long exact sequence relating, under suitable assumptions, sheaves on a diamond to their restriction to an open subdiamond and its closed complement. The existence of such an exact sequence is well-known; for the reader's convenience, we will briefly explain how to obtain it.

\subsubsection{} Fix a locally spatial diamond $X$; let $X_{\et}$ be the étale site of $X$, defined in \cite[Definition 14.1]{scholze-diamonds}, and $|X|$ the underlying topological space of $X$, defined in \cite[Definition 11.14]{scholze-diamonds}. One can attach to each $x \in |X|$ a quasi-pro-étale map $\bar{x}=\Spa(C(x), C(x)^+)\rightarrow X$, with $C(x)$ an algebraically closed perfectoid field, mapping the closed point of $\Spa(C(x), C(x)^+)$ to $x$ (see \cite[Proposition 14.3]{scholze-diamonds} and \cite[Lemma 2.2.2]{cgh20}). We call $\bar{x}$ a \emph{geometric point} of $X$ (with the important caveat that as a set $\bar{x}$ may not be a singleton), and we denote by $F_{\bar{x}}$ the stalk of a sheaf $F$ at $\bar{x}$. By \cite[Proposition 14.3]{scholze-diamonds} we may check if a sequence of étale sheaves on $X$ is exact looking at stalks at each geometric point.

\subsubsection{} Fix an open subspace of $|X|$, corresponding, by \cite[Proposition 11.15]{scholze-diamonds}, to a subdiamond $j: U \rightarrow X$; assume that the closed complement $|X|\smallsetminus |U|$ is generalising, so that, by \cite[Theorem 2.42, Propoistion 2.45]{han16}, it is the underlying topological space of a canonical locally spatial subdiamond $i: Z \rightarrow X$.\footnote{Recall that, unlike for schemes, it is not always the case that a closed subset of $|X|$ underlies a subdiamond of $X$, the issue being that maps of analytic adic spaces are generalising.} Since $|Z|$ is generalising, geometric points of $X$ are a disjoint union of geometric points of $U$ and of $Z$. We have the extension by zero functor $j_!: \mathrm{Ab}(U_{\et})\rightarrow \mathrm{Ab}(X_{\et})$, defined as for schemes (cf. the discussion before \cite[Lemma 2.2.6]{cgh20}), which is exact and left-adjoint to $j^*$. The stalk of $j_!F$ at $x \in |U|$ (resp. $x \in |Z|$) is isomorphic to $F_{\bar{x}}$ (resp. is zero).

\begin{lem}
Let $x \in |X|$ and let $\bar{x}$ be the associated geometric point; let $F$ be an étale sheaf (of abelian groups) on $Z$. If $x \not \in |Z|$ then $(i_*F)_{\bar{x}}=0$; if $x \in |Z|$ then $(i_{*}F)_{\bar{x}}=F_{\bar{x}}$.
\end{lem}

\begin{proof}
If $x \not \in |Z|$ then $x \in |U|$ hence there is a cofinal systems of étale neighbourhoods $V \rightarrow X$ of $x$ factoring through $U$; therefore $i_*F(V)=0$ and $(i_*F)_{\bar{x}}=0$. Now take $x \in |Z|$; then $(i_*F)_{\bar{x}}=\varinjlim_{V} F(V \times_X Z)$, where $V$ runs over the objects $V\rightarrow X \in X_{\et}$ through which $\bar{x}$ factors. We may restrict to $V$ spatial and such that $V \rightarrow X$ factors through a fixed spatial open subdiamond of $X$. Using \cite[Proposition 14.9]{scholze-diamonds}, the description of $\bar{x}$ before \cite[Lemma 2.2.2]{cgh20}, and the fact that $\bar{x} \rightarrow X$ factors through $Z$, we have
\begin{equation*}
\varinjlim_V F(V \times_X Z)=H^0(\varprojlim_V (V \times_X Z), F)=H^0(\bar{x}\times_X Z, F)=F_{\bar{x}}.
\end{equation*}
\end{proof}

\begin{prop}\label{fund-triang}
Let $X$ be a locally spatial diamond, and $j: U \rightarrow X$ an open subdiamond with generalising closed complement underlying a subdiamond $i: Z\rightarrow X$. For every étale sheaf $F$ of abelian groups on $X$ the sequence
\begin{equation*}
0 \rightarrow j_!j^{*}F \rightarrow F \rightarrow i_{*}i^{*}F \rightarrow 0
\end{equation*}
is exact.
\end{prop}
\begin{proof}
It suffices to check this on stalks at each geometric point $\bar{x} \rightarrow X$. If $\bar{x}$ factors through $U$ we get $0 \rightarrow F_{\bar{x}}\xrightarrow{\Id} F_{\bar{x}}\rightarrow 0 \rightarrow 0$. If it factors through $Z$ then $(j_!j^{*}F)_{\bar{x}}=0$. On the other hand setting $G=i^{*}F$ we have $(i_{*}G)_{\bar{x}}=G_{\bar{x}}$ by the previous lemma, and $G_{\bar{x}}=F_{\bar{x}}$. Hence we get $0 \rightarrow 0 \rightarrow F_{\bar{x}}\xrightarrow{\Id} F_{\bar{x}}\rightarrow 0$.
\end{proof}

\subsection{Proof of Proposition~\ref{vanish-smallig}} 

\subsubsection{} Let $b=(b_\p)_{\p \in \Sigma_p}\in B(H_{\Q_p}, \mu)$ be an element such that $R\Gamma(\mathrm{Ig}^b, \F_{\ell})_\m$ is non-trivial and such that the dimension $d^{\mathrm{min}}$ of $\Shs^b$ is as small as possible. Let $\varepsilon$ be the cardinality of the set $T\subset \Sigma_p$ consisting of places $\p$ such that $b_\p$ is basic. Then $d^{\mathrm{min}}=g-\varepsilon$, so we need to prove that
\begin{equation*}
\varepsilon\leq \delta_p(\m).
\end{equation*}
If $\varepsilon=0$ there is nothing to prove; hence let us assume $\varepsilon>0$, so that $T$ is non-empty. Let $H_T$ be the unitary group attached to $T$, defined in \S~\ref{unit-datum}, and $\mathrm{Ig}_T^{\mathrm{ord}}$ the associated $\mu_T$-ordinary Igusa variety. Consider the Hodge--Tate period map
\begin{equation*}
\pi_{\HT}^T: \mathcal{S}h_{U^p}(H_T)\rightarrow \mathscr{F}\ell_{H_{T}, \mu_T}.
\end{equation*}
Our assumption and Theorem \ref{isoigusa} imply that $R\Gamma(\mathrm{Ig}_T^{\mathrm{ord}}, \F_{\ell})_\m$ is non-trivial. Because of Proposition \ref{fibreig} we deduce that $(R\pi_{\HT, *}^T \F_{\ell})_\m$ has non-trivial cohomology on the $\mu_T$-ordinary stratum. Better, we have:
\begin{lem}\label{lem-ord-max}
The $\mu_T$-ordinary stratum is the only one where $(R\pi_{\HT, *}^T\F_{\ell})_\m$ has non-zero cohomology.
\end{lem}
\begin{proof}
Assume the contrary. Then using Proposition \ref{fibreig} we find that there is a non-ordinary element $b'_T \in B(H_{T, \Q_p}, \mu_T)$ such that $H^*(\mathrm{Ig}^{b'_T}_T, \F_{\ell})_\m$ is non-zero. Consider the element $b' \in B(H_{\Q_p}, \mu)$ which is basic precisely at places in $T$ and at places outside $T$ where $b'_T$ is basic. By Theorem \ref{isoigusa} the cohomology $H^*(\mathrm{Ig}^{b'}, \F_{\ell})_\m$ is not identically zero; however the dimension of $\Shs^{b'}$ is strictly smaller than the dimension of $\Shs^b$, hence we obtain a contradiction. 
\end{proof}

\subsubsection{}\label{end-vanish-proof} We now come back to the proof of Proposition~\ref{vanish-smallig}. Perversity of the complex $R\pi^T_{\HT, *}\F_{\ell}[g-\varepsilon]$ and the fact that it is concentrated on one stratum after localisation at $\m$ imply that $H^i(\mathrm{Ig}_T^{\mathrm{ord}}, \F_{\ell})_\m$ is non-zero only in middle degree $g-\varepsilon$ (cf. \cite[Corollary 6.1.4]{cs17}). Theorem~\ref{ord-par-ind} (together with the fact that smooth parabolic induction of a non-zero representation is non-zero) implies that
\begin{equation}\label{igcohom-middle}
H^i(\mathcal{S}h_{U^p}(H_T)^{b^\mathrm{ord}}, \F_{\ell})_\m\neq 0 \Leftrightarrow i=g-\varepsilon.
\end{equation}

Let $Z\subset \mathscr{F}\ell_{H_{T}, \mu_T}$ be the $\mu_T$-ordinary locus, and $V:=\mathscr{F}\ell_{H_{T}, \mu_T}\smallsetminus Z$. We have an inclusion
\begin{equation}\label{2strat}
\mathcal{S}h_{U^p}(H_T)^{b^\mathrm{ord}}\subset (\pi_{\HT}^{T})^{-1}(Z)
\end{equation}
which induces a bijection on rank one points. The ordinary locus $\mathcal{S}h_{U^p}(H_T)^{b^\mathrm{ord}}$ is open and quasicompact, as it is the preimage of a quasicompact open subspace in the special fibre via the specialisation map, which is continuous and spectral. By \cite[Lemma 4.4.2]{cs17} the map induced in cohomology by \eqref{2strat} is an isomorphism, hence the equivalence in \eqref{igcohom-middle} holds true for the cohomology of $(\pi_{\HT}^{T})^{-1}(Z)$ as well. Let $i_Z: Z \rightarrow \mathscr{F}\ell_{H_{T}, \mu_T}$ and $i_V : V \rightarrow \mathscr{F}\ell_{H_{T}, \mu_T}$ be the inclusion maps. The assumptions of Proposition \ref{fund-triang} are satisfied, hence for every $j \geq 0$ we have an exact sequence 
\begin{equation*}
0 \rightarrow i_{V, !}i_V^*(R^j\pi^T_{\HT, *}\F_{\ell})\rightarrow R^j\pi^T_{\HT, *}\F_{\ell}\rightarrow i_{Z, *}i_Z^*(R^j\pi^T_{\HT, *}\F_{\ell})\rightarrow 0.
\end{equation*}
Lemma \ref{lem-ord-max} ensures that the first term vanishes after localisation at $\m$, so that $H^i(\mathscr{F}\ell_{H_{T}, \mu_T}, R^j\pi^T_{\HT, *}\F_{\ell})_\m=H^i(Z, R^j\pi^T_{\HT, *}\F_{\ell})_\m$. The outcome of our discussion is that cohomology of $\mathcal{S}h_{U^p}(H_T)$ coincides with cohomology of its ordinary locus, after localisation at $\m$; hence $H^i(\mathcal{S}h_{U^p}(H_T), \F_{\ell})_\m \neq 0$ if and only if $i=g-\varepsilon$.

Now $\mathcal{S}h_{U^p}(H_T)$ is the inverse limit (as a diamond) of the spaces $\mathcal{S}h_{U^pU_{T, p}}(H_T)$ for $U_{T, p}\subset H_T(\Q_p)$ running over compact open subgroups. There is a cofinal system of subgroups $U_{T, p}$ which are images in $H_T(\Q_p)$ of subgroups of the form $K_{T, p}\times K_{E, p}\subset G_T(\Q_p)\times T_E(\Q_p)$, where $K_{E, p}$ is such that $K_{E, p}K_{E}^p$ is sufficiently small with respect to $K_{T, p}K^p$, in the sense of Definition~\ref{defn:suff small}. For $U_{T, p}$ of this form and small enough, we must have
\begin{equation*}
H^i(\mathcal{S}h_{U^pU_{T, p}}(H_T), \F_{\ell})_\m=0 \text{ for }i\neq g-\varepsilon, \; H^{g-\varepsilon}(\mathcal{S}h_{U^pU_{T, p}}(H_T), \F_{\ell})_\m\neq 0,
\end{equation*}
where the first equality holds because a small enough subgroup $U_{T, p}$ is a pro-$p$ group, hence the map from the cohomology of $\mathcal{S}h_{U^pU_{T, p}}(H_T)$ to the cohomology of $\mathcal{S}h_{U^p}(H_T)$ is injective. The argument in the proof of Corollary \ref{torsfree-Zl} shows that $H^{g-\varepsilon}(Sh_{U^pU_{T, p}}(H_T), \Z_{\ell})_\m$ is non-zero and torsion free, therefore we deduce that
\begin{equation*}
H^{g-\varepsilon}(Sh_{U^pU_{T, p}}(H_T), \Q_{\ell})_\m\neq 0.
\end{equation*}
Applying Corollary \ref{Hilbtounitary} once more we obtain that $H^{g-\varepsilon}(Sh_{K^pK_{T, p}}(G_T), \Q_{\ell})_\m\neq 0$. We claim that this implies that $\bar{\rho}_\m$ cannot be generic at places in $T$; this yields the desired inequality $\varepsilon \leq \delta_p(\m)$ and ends the proof of Proposition \ref{vanish-smallig}. 

To justify our claim, recall that $H^{g-\varepsilon}(Sh_{K^pK_{T, p}}(G_T), \bar{\Q}_{\ell})$ can be described in terms of automorphic representations of $G_T$ as in \cite[(5.9)-(5.11)]{nek18}. The Hecke algebra acts on the universal part of cohomology (corresponding to one-dimensional automorphic representations, and called case $(A)$ in \emph{loc. cit.}) via the degree character. Hence this part of cohomology vanishes after localisation at $\m$, since $\m$ is non-Eisenstein. Therefore, as in \cite[(5.11)(B)]{nek18}, there is an automorphic representation $\pi_T$ of $G_T$ which transfers to an automorphic representation $\pi$ of $\GL_{2, F}$ attached to a (holomorphic) cuspidal Hilbert newform $f$ such that almost all the Hecke eigenvalues of $f$ modulo $\ell$ are given by the image of the map $\mathbb{T} \rightarrow \mathbb{T}/\m$. Letting $\rho_\pi: \Gamma_F \rightarrow \GL_2(\bar{\Q}_{\ell})$ be the Galois representation attached to $\pi$, we deduce that the reduction modulo $\ell$ of (a lattice in) $\rho_\pi$ is isomorphic to $\bar{\rho}_\m$. Let $v$ be a place in $T$; by \cite[Theorem A]{car86} the Galois representation attached to $\pi_v$ lifts $\bar{\rho}_{\m|\Gamma_{F_v}}$. As $\pi_v$ is the Jacquet--Langlands transfer of a representation of the non-split quaternion algebra over $F_v$, we deduce using \S~\ref{notransfer} that $\bar{\rho}_\m$ is not generic at $v$.

\begin{rem}
In the above proof, we used the parabolically induced structure of the ordinary locus to deduce that $H^*(\mathcal{S}h_{U^p}(H_T)^{b^\mathrm{ord}}, \F_{\ell})_\m$ is non-zero and concentrated in one degree from the analogous properties of cohomology of the Igusa variety. One could also argue computing cohomology of $\mathcal{S}h_{U^p}(H_T)^{b^\mathrm{ord}}$ via the Leray spectral sequence, and using the fact that the ordinary locus in the flag variety is a profinite set (hence a sheaf on it with a non-zero stalk has non-zero global sections and vanishing higher cohomology).
\end{rem}

\subsection{Cohomology of quaternionic Shimura varieties}

\subsubsection{} Let $\ell$ be a prime number. In this section we explain how the arguments used in this paper can be extended to study the cohomology with $\F_\ell$-coefficients of quaternionic Shimura varieties attached to non-split quaternion algebras. In a nutshell, the strategy we used for Hilbert modular varieties can be applied in this generality, with the difference that we do not need to restrict to non-Eisenstein maximal ideals, as Shimura varieties attached to non-split quaternion algebras have no boundary. The construction of Galois representations is also simpler in this case, therefore some results also apply to the case $\ell = 2$. 

We fix a non-split quaternion algebra $B$ over a totally real number field $F$ of degree $g$. Let $R$ (resp. $R_\infty$) be the set of places (resp. infinite places) of $F$ where $B$ is ramified. In this section we will denote by $G$ the group $\mathrm{Res}_{F/\Q}B^\times$; we fix a neat compact open subgroup $K \subset G(\A_f)$ and we denote by $Sh_K(G)$ the corresponding Shimura variety, of dimension $d:=g-|R_\infty|$. Let $\mathbb{T}$ be the Hecke algebra generated by operators at places of residue characteristic different from $\ell$ where $B$ is unramified and $K$ is hyperspecial. Given a prime $p$ totally split in $F$, fix an isomorphism $\bar{\Q}_p \simeq \C$, inducing a bijection between the set $\Sigma_\infty$ of infinite places and the set $\Sigma_p$ of $p$-adic places of $F$. We denote by $R_p$ the set of $p$-adic places corresponding to $R_\infty$. The cocharacter of $G_{\bar{\Q}_p}$ induced by the cocharacter coming from the Shimura datum attached to $G$ is central at places in $R_p$. In the rest of the section we will discuss the proof of the following result.

\begin{thm}\label{quat-maithm}
Let $\ell$ be a prime number and $\m \subset \mathbb{T}$ a maximal ideal in the support of $H^*(Sh_K(G), \F_\ell)$.
\begin{enumerate}
\item There exists a unique Galois representation $\bar{\rho}_\m: \Gamma_F \rightarrow \GL_2(\bar{\F}_\ell)$ attached to $\m$, in the sense of Theorem \ref{constr-gal-repn}.
\item Let $p\neq \ell$ be an odd prime number which splits completely in $F$ and such that $B$ is unramified at every place above $p$ and $K=K^pK_p$ with $K_p \subset G(\Q_p)$ hyperspecial. Let $\delta_p(\m)$ be the cardinality of the set of places $v \in \Sigma_p \smallsetminus R_p$ such that $\bar{\rho}_\m$ is not generic at $v$. Then
\begin{equation*}
i \not \in [d-\delta_p(\m), d+\delta_p(\m)]\Rightarrow H^i(Sh_K(G), \F_\ell)_\m = 0.
\end{equation*}
\item If $\ell>2$ and the image of $\bar{\rho}_{\m}$ is not solvable, then 
\[
i \neq d \Rightarrow H^i(Sh_K(G), \F_\ell)_\m = 0. 
\]
\end{enumerate}
\end{thm}

\subsubsection{} Take a prime $p$ as in point $(2)$ of the theorem. We choose an auxiliary $CM$ extension $E/F$ such that every finite place in $R$ is inert in $E$, and a place of $F$ above $p$ is split (resp. inert) in $E$ if it belongs (resp. does not belong) to $R_p$. We consider the auxiliary unitary group $H_\emptyset$ constructed as in \S~\ref{unit-datum}; we will denote $H_\emptyset$ by $H$ for simplicity. We can write the Kottwitz set $B(H_{\Q_p}, \mu)$ as a product $B(H_{\Q_p}, \mu)=\prod_{\Sigma_p}B(H_{\p}, \mu_{\p})$, where $B(H_{\p}, \mu_{\p})$ has one element (resp. two elements, the basic and the $\mu_\p$-ordinary element) if $\p \in R_p$ (resp. if $\p \in \Sigma_p \smallsetminus R_p$). We fix a compact open subgroup $K_E=(\mathcal{O}_E \otimes \Z_p)^\times K_E^{p}$ sufficiently small with respect to $K$, and we let $U\subset H(\A_f)$ be the image of $K_E \times K$. Given $b=(b_\p)_{\p \in \Sigma_p} \in B(H_{\Q_p}, \mu)$ we have the corresponding Newton stratum $\overline{Sh}^b_{U}(H)$ which is affine and smooth, of dimension $d-\varepsilon_b$, where $\varepsilon_b$ denotes the number of places $\p \in \Sigma_p \smallsetminus R_p$ such that $b_\p$ is basic.

Given a set $T \subset \Sigma_p \smallsetminus R_p$, let $B_T$ be the quaternion algebra over $F$ ramified at places in $R$ as well as at places in $T$ and at the corresponding infinite places. We let $G_T:=\mathrm{Res}_{F/\Q}B_T^\times$; fix an isomorphism $G(\A_f^{(p)})\simeq G_T(\A_f^{(p)})$, allowing us to see $K^p$ as a subgroup of $G_T(\A_f^{(p)})$. The arguments in \S~\ref{quatvsun} and \S~\ref{mod-unit} go through in this setting; in particular one can define moduli problems giving rise to integral models of the Shimura varieties attached to the groups $H$ and $H_T$; this makes it possible to define Igusa varieties, and Theorem \ref{isoigusa} has an analogue in this setting. Similarly, we have a description of the fibres of the relevant Hodge--Tate period maps in terms of Igusa varieties, and of the $\mu$-ordinary locus at infinite level as being parabolically induced from the corresponding Igusa variety.

Theorem~\ref{quat-maithm} will follow from the following proposition.

\begin{prop}\label{key-quat-prop}
Take $b=(b_\p)_{\p \in \Sigma_p} \in B(H_{\Q_p}, \mu)$ corresponding to a Newton stratum with smallest possible dimension such that $H^*(\overline{Sh}^b_{U}(H), \F_\ell)_\m \neq 0$. Let $T \subset \Sigma_p \smallsetminus R_p$ be the set of places $\p$ such that $b_\p$ is basic. Then there is a compact open subgroup $K_{T, p}\subset G_T(\Q_p)$ such that
\begin{equation*}
H^*(Sh_{K^{p}K_{T, p}}(G_T), \Q_\ell)_\m \neq 0.
\end{equation*}
\end{prop}

\begin{proof}
As $H^*(\overline{Sh}^b_{U}(H), \F_\ell)_\m \neq 0$ we deduce that $H^*(\mathrm{Ig}^b, \F_\ell)_\m \neq 0$, hence by Theorem \ref{isoigusa} we get $H^*(\mathrm{Ig}_T^{{\mathrm{ord}}}, \F_\ell)_\m \neq 0$. Denoting by $\pi_{\HT}^T$ the Hodge--Tate period map for $\mathcal{S}h_{U^p}(H_T)$, we deduce as in Lemma \ref{lem-ord-max} that $(R\pi_{\HT, *}^T \F_{\ell})_\m$ is supported only on the $\mu_T$-ordinary locus. Perversity (up to shift) of $R\pi_{\HT, *}^T \F_{\ell}$ implies that $H^*(\mathrm{Ig}_T^{\mathrm{ord}}, \F_\ell)_\m$ is non-zero only in middle degree; the same argument as in \S~\ref{end-vanish-proof} then shows that $H^*(\mathcal{S}h_{U^p}(H_T), \F_\ell)_\m$ is concentrated in middle degree; we can finally descend to an appropriate finite level and apply Corollary \ref{Hilbtounitary} to deduce that $H^*(Sh_{K^{p}K_{T, p}}(G_T), \F_\ell)_\m$ is concentrated in middle degree. It follows that $H^*(Sh_{K^{p}K_{T, p}}(G_T), \Q_\ell)_\m\neq 0$.
\end{proof}

\subsubsection{Proof of Theorem \ref{quat-maithm}} The existence of $\bar{\rho}_\m$ follows from Proposition \ref{key-quat-prop}, the description of the cohomology of $Sh_{K^{p}K_{T, p}}(G_T)$ with characteristic zero coefficients in terms of automorphic forms and the existence of Galois representations attached to quaternionic automorphic forms. The second point is proved adapting the argument in \S~\ref{main2-ref-proof} and \S~\ref{end-vanish-proof}. The third point follows from the second and from Lemma~\ref{cebot}.  

\section{Compatibility with the $p$-adic local Langlands correspondence}\label{sec:p-adic local Langlands}

Fix a prime $p$ throughout this section, which will play the role of the prime denoted by $\ell$ above.
The aim of this section is to relate the completed homology of Hilbert modular varieties
and the $p$-adic Langlands correspondence, building on \cite{gee-newton}.

\subsection{Completed homology}

\subsubsection{} Following \cite{gee-newton} we work with locally symmetric spaces attached to $\PGL_{2, F}$ rather than $\GL_{2, F}$. Let $G:=\mathrm{Res}_{F/\Q}\GL_{2}$ and $\bar{G}:=\mathrm{Res}_{F/\Q}\PGL_{2}$. Given a compact open subgroup $\bar{K}\subset\bar{G}(\A_f)$ we have the locally symmetric space 
\[
X_{\bar{K}}(\bar{G}):=\bar{G}(\Q)\backslash (\C\smallsetminus \R)^g\times \bar{G}(\A_f)/\bar{K}.
\] 
Fix a compact open subgroup $\bar{K}^p\bar{K}_0$ which is good in the sense of \cite[Section 2.1]{gee-newton}, with $\bar{K}_0=\prod_{v \mid p}\PGL_2(\O_{F_v})$. Consider the completed homology groups
\begin{equation*}
\tilde{H}_*(\bar{K}^p, \Z_p):=\varprojlim_{\bar{K}_p}H_*(X_{\bar{K}_p\bar{K}^p}, \Z_p).
\end{equation*}
Let $\mathbb{T}(\bar{K}^p)$ be the big Hecke algebra defined in \cite[Definition 2.1.11]{gee-newton} (beware that our notation differs from the one in \emph{loc. cit.}). It maps to the Hecke algebra $\mathbb{T}(\bar{K}^p\bar{K}_p, \Z/p^s\Z)$ acting on homology with $\Z/p^s\Z$-coefficients of the locally symmetric space of level $\bar{K}^p\bar{K}_p$ for each $\bar{K}_p$ and each $s \geq 1$. Completed homology $\tilde{H}_*(\bar{K}^p, \Z_p)$ carries an action of $\mathbb{T}(\bar{K}^p)$ as well as of $\bar{G}(\Q_p)$; we fix a maximal ideal $\m\subset \mathbb{T}(\bar{K}^p)$. For $\bar{K}_p$ small enough, the ideal $\m$ is the preimage of a maximal ideal in $\mathbb{T}(\bar{K}^p\bar{K}_p, \Z/p\Z)$, which we will abusively denote by the same symbol (in fact any pro-$p$ group $\bar{K}_p$ does the job, cf. the proof of \cite[Lemma 2.1.14]{gee-newton}). Hence, setting $\bar{K}=\bar{K}^p\bar{K}_p$, we have $H^*(X_{\bar{K}}(\bar{G}), \F_p)_\m\simeq H_*(X_{\bar{K}}(\bar{G}), \F_p)_\m \neq 0$.

\subsubsection{}\label{compcohom-proj} Let us now assume that $p>3$; we take $\bar{K}$ as above, and assume in addition that it is the image of a compact open subgroup $K \subset G(\A_f)$. The natural map $Sh_K(G) \rightarrow X_{\bar{K}}(\bar{G})$ is a finite étale Galois cover, hence $H^*(Sh_K(G), \F_p)_{\m}\neq 0$. By Theorem \ref{constr-gal-repn} there is a Galois representation $\bar{\rho}_\m: \Gamma_F \rightarrow \GL_2(\bar{\F}_p)$ attached to $\m$. Let us suppose that the projective image of $\bar{\rho}_\m$ contains a conjugate of $\PSL_2(\F_p)$ or is isomorphic to $A_5$; then Theorem \ref{mainpart2} implies that, for every $K'_p\subset K_{p}$, we have 
\[
H^i(Sh_{K^pK'_p}(G), \F_p)_\m=H^i_c(Sh_{K^pK'_p}(G), \F_p)_\m=0
\]
for $i<g$; the Hochschild--Serre spectral sequence attached to the Galois cover $Sh_{K^pK'_p}(G)\rightarrow X_{\overline{K^pK'_p}}(\bar{G})$ implies that the same property holds true for the cohomology of $X_{\overline{K^pK'_p}}(\bar{G})$, hence $H^*(X_{\overline{K^pK'_p}}(\bar{G}), \F_p)_\m\simeq H_*(X_{\overline{K^pK'_p}}(\bar{G}), \F_p)_\m$ is concentrated in degree $g$. Therefore assumptions $(a), (b)$ of \cite[Proposition 4.2.1]{gee-newton} are satisfied (notice that assumption $(b)$ trivially holds since, with the notation of \emph{loc. cit.}, we have $l_0=0$ in our situation). Thanks to \cite[Proposition 4.2.1]{gee-newton} we obtain that $\tilde{H}_i(\bar{K}^p, \Z_p)_\m=0$ if $i \neq g$ and $\tilde{H}_g(\bar{K}^p, \Z_p)_\m$ is a projective, $p$-torsion-free $\Z_p[[\bar{K}_{0}]]$-module. Furthermore, there is a Galois representation 
\[
\rho_\m: \Gamma_F \rightarrow \GL_2(\mathbb{T}(\bar{K}^p)_\m)
\]
lifting $\bar{\rho}_\m$,
as in \cite[Conjecture 3.3.3(2)]{gee-newton}. It can be constructed as in \cite[\S~5]{sch18}: the argument in \emph{loc. cit.}, glueing representations valued in Hecke algebras at increasing finite level, can be applied in our situation as (co)homology with $\Z_p$-coefficients is concentrated in one degree and torsion-free after localisation at $\m$, and cohomology with $\Q_p$-coefficients localised at $\m$ is described in terms of Hilbert cusp forms.

\subsection{$p$-adic local Langlands} Assume from now on that $p$ is totally split in $F$. We want to describe the relation between completed homology of Hilbert modular varieties and the $p$-adic Langlands correspondence for $\GL_2(\Q_p)$, using the machinery of \cite{gee-newton}. As remarked in \cite[\S~5.4]{gee-newton}, assumption $(a)$ of \cite[Proposition 4.2.1]{gee-newton}, which we established above, is the key input needed to apply the results in \emph{loc. cit.}.
\subsubsection{} We first need to introduce some notation. We replace $\Z_p$ by the ring of integers $\O$ of a finite extension of $\Q_p$, with residue field $k$, so that $\bar{\rho}_\m$ takes values in $\GL_2(k)$. Furthermore, up to further extending $\mathcal{O}$, we may, and will, assume that $k$ contains all the eigenvalues of the elements in the image of $\bar{\rho}_\m$ (as in \cite[\S~3.2]{gee-newton}). We denote by $L$ the fraction field of $\O$. For a place $v \mid p$, under suitable assumptions on $\bar{\rho}_{\m_{|\Gamma_{F_v}}}$ (for example, if it is absolutely irreducible) we have the universal local deformation ring $R_v^{\mathrm{def}}$ of $\bar{\rho}_{\m \mid \Gamma_{F_v}}$. We denote by $\pi_v$ the $k$-representation of $\PGL_2(F_v)$ attached to $\bar{\rho}_{\m \mid \Gamma_{F_v}}$, and we let $P_v$ be the projective envelope of $\pi_v^\vee$ in the Pontryagin dual of the category of locally admissible $\O$-representations of $\PGL_2(F_v)$. We set $R_p^{\mathrm{loc}}:=\hat{\otimes}_{v\mid p, \O} R_v^{\mathrm{def}}$; then $P:=\hat{\otimes}_{v \mid p, \O}P_v$ has an $R_p^{\mathrm{loc}}$-module structure, cf.~\cite[Proposition 6.3, Corollary 8.7]{paskunas} (see also \cite[\S~5.1]{gee-newton}). 
Finally, the representation $\rho_\m$ gives rise to a map $R_{p}^{\mathrm{loc}}\rightarrow \mathbb{T}(\bar{K}^p)_\m$. Let $N(\bar{\rho}_\m)$ be the prime-to-$p$-conductor of $\bar{\rho}_\m$ and let $\bar{K}_1(N(\bar{\rho}_\m))$ be the image of $\{M \in \GL_2(\hat{\mathcal{O}}_F)\mid M \equiv \left(\begin{smallmatrix}* & *\\
0 & 1\end{smallmatrix}\right) \pmod{N(\bar{\rho}_\m)}\}$ (more precisely, if this group is not good, we make it smaller at places where $\bar{\rho}_\m$ admits no ramified deformations).

\begin{thm}
Let $p>3$ be a prime which splits completely in $F$. Assume that the following assertions hold true.
\begin{enumerate}
\item The projective image of the Galois representation $\bar{\rho}_\m$ attached to $\m$ contains a conjugate of $\PSL_2(\F_p)$ or is isomorphic to $A_5$.
\item If $\bar{\rho}_\m$ is ramified at some place $v$ not lying above $p$, then $v$ is not a vexing prime.
\item For each place $v \mid p$, the restriction of $\bar{\rho}_\m$ to $\Gamma_{F_v}$ is absolutely irreducible.
\end{enumerate}
Then there is an isomorphism of $\mathbb{T}(\bar{K}_1(N(\bar{\rho}_\m))^p)_\m[\bar{G}(\Q_p)]$-modules
\begin{equation*}
\tilde{H}_g(\bar{K}_1(N(\bar{\rho}_\m))^p, \O)_\m \simeq \mathbb{T}(\bar{K}_1(N(\bar{\rho}_\m))^p)_\m \hat{\otimes}_{R_p^{\mathrm{loc}}}P^{\oplus m}
\end{equation*}
for some $m \geq 1$.
\end{thm}
\begin{proof}
This is \cite[Proposition 5.1.4]{gee-newton}, which follows from \cite[Conjecture 5.1.2]{gee-newton}. Let us explain why this and other conjectures formulated in \emph{loc. cit.} hold true in our setting, and the various assumptions made in \emph{loc. cit.} are satisfied.
\begin{enumerate}
\item In \cite[\S~4.1]{gee-newton} the authors assume that the image of restriction of $\bar{\rho}_\m$ to $\Gamma_{F(\zeta_p)}$ is enormous. This is needed for the construction of 
Taylor--Wiles data, as in \cite[Lemma 3.3.1]{gee-newton}. Our large image assumption and \cite[Lemma 3.2.3]{gee-newton} imply that the group $\bar{\rho}_\m(\Gamma_{F(\zeta_p)})$ is enormous if $p>5$. If $p=5$, our assumption does not guarantee that $\bar{\rho}_\m(\Gamma_{F(\zeta_p)})$ is enormous. However, in this case, one can work under the assumption \cite[Hypothesis 9.1]{sch18}, which originates in~\cite[Theorem 3.5.5]{kisin09} and which is satisfied in our situation. Indeed, if the image of $\mathrm{pr}\circ \bar{\rho}_\m: \Gamma_F \rightarrow \PGL_2(\bar{\F}_5)$ is conjugate to $\PGL_2(\F_5)$ and its kernel fixes $F(\zeta_5)$ then $[F(\zeta_5): F]=2$, which cannot happen if $p=5$ is unramified in $F$. 

\item The assumption that $\bar{\rho}_\m\neq \bar{\rho}_\m\otimes \bar{\epsilon}$, where $\bar{\epsilon}$ is the mod $p$ cyclotomic character (see \cite[p. 18]{gee-newton}) holds. Indeed because $p$ is unramified $F$, the fields $F$ and $\Q(\zeta_p)$ are linearly disjoint over $\Q$; hence $\bar{\epsilon}: \Gamma_F \rightarrow (\Z/p\Z)^\times$ is surjective. In particular there is $\gamma \in \Gamma_F$ such that $\bar{\epsilon}(\gamma)^2 \neq 1$, hence $\det \bar{\rho}_\m(\gamma)\neq \det(\bar{\epsilon}\otimes \bar{\rho}_\m)(\gamma)$. For the same reason the first assumption in \cite[Hypothesis 4.1.3]{gee-newton} is satisfied.
\item \cite[Conjecture 3.3.7]{gee-newton} holds: the representation denoted by $\rho_{\m, Q}$ in \emph{loc. cit.} can be constructed as explained before the statement of the theorem, and the desired local-global compatibility follows from the analogous statement for Hilbert modular forms.
\item The second assumption in \cite[Hypothesis 4.1.3]{gee-newton} holds in view of the hypothesis that there are no vexing primes.
\end{enumerate}
The upshot of the above discussion is that all the assumptions made at the beginning of \cite[\S~5]{gee-newton} are verified. Therefore, by \cite[Corollary 5.3.2]{gee-newton}, Conjecture 5.1.2 follows from Conjecture 5.1.12 in \emph{loc. cit.}. 

It remains to justify why Conjecture 5.1.12 holds. Let $Q$ be a set of Taylor--Wiles primes and $\sigma=(\sigma_v)_{v \mid p}$ an irreducible $L$-representation of $\bar{K}_0$, with $\bar{K}_0$-stable $\O$ lattice $\sigma^\circ$ as in \cite[p. 34]{gee-newton}. Let $X_{\bar{K}_1(Q)}(\bar{G})$ be the space with level $\bar{K}_0$ at $p$ and $\bar{K}_1(Q)_v$ at each place $v \in Q$. We need to show that the action of $R_p^{\mathrm{loc}}$ on $H_*(X_{\bar{K}_1(Q)}(\bar{G}), \sigma^\circ)_{\m_{Q, 1}}$ factors through $R_p^{\mathrm{loc}}(\sigma)$, where the ideal $\m_{Q, 1}$ is defined in \cite[Proposition 3.3.6]{gee-newton}. Let $\mathbb{T}(\bar{K}_1(Q), \sigma)_{\m_{Q, 1}}$ be the image of the Hecke algebra in the endomorphism ring of $H_g(X_{\bar{K}_1(Q)}(\bar{G}), \sigma^\circ)_{\m_{Q, 1}}$. By assumption $(1)$ the homology $H_*(X_{\bar{K}_1(Q)}(\bar{G}), \sigma^\circ)_{\m_{Q, 1}}$ is concentrated in middle degree, and $H_g(X_{\bar{K}_1(Q)}(\bar{G}), \sigma^\circ)_{\m_{Q, 1}}$ is torsion-free; furthermore $\mathbb{T}(\bar{K}_1(Q), \sigma)_{\m_{Q, 1}}$ is reduced. As in the proof of the first part of \cite[Lemma 4.17]{ceggps1}, it suffices to show that, for any $\bar{\Q}_p$-point of $\mathbb{T}(\bar{K}_1(Q), \sigma)_{\m_{Q, 1}}$, the restriction to $\Gamma_{F_v}$ (for $v \mid p$) of the associated Galois representation with $\bar{\Q}_p$-coefficients is crystalline of type $\sigma_v$. The space $H_*(X_{\bar{K}_1(Q)}(\bar{G}), \sigma)_{\m_{Q, 1}}$ can be described in terms of Hilbert cusp forms with algebraic weight, hence the desired result follows from local-global compatibility at places above $p$ \cite[Theorem 4.3]{kis08}. This establishes the first assertion in \cite[Conjecture 5.1.12]{gee-newton}; the second one can be proved similarly, following the strategy in the second part of the proof of \cite[Lemma 4.17]{ceggps1}.
\end{proof}

\begin{remark}\label{rem:further lgc} In fact, one should be able to prove a
version of compatibility with the $p$-adic local Langlands correspondence without
any assumptions on the tame level, 
for example the analogue to this setting of~\cite[Corollary 6.3.6]{pan-LA}. That this
would follow from Theorem~\ref{thm:generic} and the machinery developed by 
Pa{\v{s}}k\={u}nas is more or less known to experts. We sketch the argument in this remark. 

One can consider the completed cohomology of a
Hilbert modular variety with $\Z_p$-coefficients.  
After localisation at a maximal ideal $\mathfrak{m}$ whose associated residual 
Galois representation is non-solvable, Theorem~\ref{thm:generic} 
implies that this is concentrated in one degree. 
As in~\cite[Corollary 6.3.6]{pan-LA}, the key point is to prove the analogue of~\cite[Theorem 6.3.5]{pan-LA} 
that compares two actions of the local pseudo-deformation ring with fixed
determinant on completed cohomology localised at $\mathfrak{m}$. 
(See~\cite[Theorem 3.5.5]{pan-FM} 
for the statement in the case when there are multiple split 
places above $p$, which relies on~\cite[Corollary 3.4.12]{pan-FM}.) 

The comparison of the two pseudo-deformation ring actions can be done using the method of~\cite[Cor. 5.6]{paskunas-ludwig}. 
The key facts that would need checking are: density of those locally 
algebraic vectors for which the smooth part is a specific principal series representation, 
and semi-simplicity of the usual Hecke action on the non-Eisenstein cohomology of Hilbert modular varieties.
The latter can be deduced from the Eichler--Shimura 
isomorphism and the existence of the Petersson inner product on Hilbert cusp forms. 
Density of locally algebraic vectors follows from~\cite[Corollary 7.8]{DPS} 
and the fact that completed cohomology localised at $\mathfrak{m}$ with fixed central character 
is a direct summand of a space of continuous functions with fixed central character 
as in the proof of~\cite[Theorem 9.24]{DPS}. Finally, the latter fact follows from projectivity of completed homology localised at $\mathfrak{m}$ (cf. \cite[Lemma 9.16(i)]{DPS}). As in \S~\ref{compcohom-proj}, to establish this last property one crucially uses Theorem~\ref{thm:generic}. 
\end{remark}

\bibliographystyle{amsalpha}
\bibliography{HMVviaGJL}
\end{document}